\documentclass{amsart}
\usepackage{amssymb}
\usepackage{amsfonts}

\newtheorem{theorem}{Theorem}
\theoremstyle{plain}
\newtheorem{acknowledgement}{Acknowledgement}

\newtheorem{corollary}{Corollary}

\newtheorem{definition}{Definition}

\newtheorem{lemma}{Lemma}

\newtheorem{remark}{Remark}

\numberwithin{equation}{section}
\input{tcilatex}

\begin{document}
\title[On singular $Q$-curvature ...]{On singular $Q$-curvature type
equations}
\author{Mohammed Benalili}
\email{m\_benalili@mail.univ-tlemcen.dz}
\subjclass[2000]{Primary 58J05}
\keywords{$Q$-curvature type equation, Hardy inequality, Sobolev's exponent
growth, Singularities, regularity..}

\begin{abstract}
This paper is devoted to $Q$-curvature type equations with singularities;
mainly we give existence and regularity results of solutions. To have
positive solutions which will be meaningfully in conformal geometry we
restrict ourself to special manifolds.
\end{abstract}

\maketitle

\section{Introduction}

In 1983, Paneitz (\cite{19}) introduced a conformal fourth order operator
defined on $4$-dimensional Riemannian manifolds by%
\begin{equation*}
P_{g}^{4}\left( u\right) =\Delta _{g}^{2}u-div_{g}\left( \frac{2}{3}%
R_{g}g-2Ric_{g}\right) du
\end{equation*}%
where $\Delta _{g}u=-div_{g}\left( \nabla _{g}u\right) $ is the Laplacian of 
$u$ with respect to $g$, $R_{g}$ is the scalar curvature with respect to $g$
and, $Ric_{g}$ is the Ricci curvature of $g$ .

Branson (\cite{6}) generalized the notion to manifolds of dimension $n\geq 5$
by letting 
\begin{equation*}
P_{g}^{n}\left( u\right) =\Delta _{g}^{2}u-div_{g}\left( \frac{\left(
n-2\right) ^{2}+4}{2\left( n-1\right) \left( n-2\right) }R_{g}g-\frac{4}{n-2}%
Ric_{g}\right) du+\frac{n-4}{2}Q_{g}^{n}u
\end{equation*}%
where%
\begin{equation*}
Q_{g}^{n}=\frac{1}{2\left( n-1\right) }\Delta _{g}R_{g}+\frac{%
n^{3}-4n^{2}+16n-16}{8\left( n-1\right) ^{2}\left( n-2\right) ^{2}}R_{g}^{2}-%
\frac{2}{\left( n-2\right) ^{2}}\left\vert Ric_{g}\right\vert ^{2}\text{.}
\end{equation*}%
Both operators $P_{g}^{4}$ and $P_{g}^{n}$ are conformal operators that
means that: for every $u\in C^{\infty }\left( M\right) $, $P_{\widetilde{g}%
}^{4}(\varphi u)=$\ $e^{-4\varphi }P_{g}^{4}(u)$\ where $\widetilde{g}%
=e^{2\varphi }g$, while $P_{g}^{n}\left( \varphi u\right) =\varphi ^{\frac{%
n+4}{n-4}}P_{\widetilde{g}}^{n}\left( u\right) $\ where $\widetilde{g}%
=\varphi ^{\frac{4}{n-4}}g$ and $n\geq 5$, $\varphi \in C^{\infty }(M)$. $%
P_{g}^{4}$ is the analogous of $\Delta _{g}$ in dimension $2$ while $%
P_{g}^{n}$ is the analogous of the conformal Laplacian $L_{g}u=\Delta _{g}u+%
\frac{n-2}{4\left( n-1\right) }R_{g}u$.

Fourth order equations of Sobolev growth have been the subject of intensive
investigation the last tree decades because of theirs applications to
conformal geometry, in particular to Paneitz-Branson operators; we refer the
reader to \cite{2}, \cite{3}, \cite{4}, \cite{5}, \cite{7}, \cite{10}, \cite%
{11}, \cite{12}, \cite{13}, \cite{15}, \cite{16}, \cite{17}, \cite{19}, \cite%
{21}, and \cite{22}.These equations are also interesting because of theirs
analogues in the second order which give applications to the conformal
Laplacian. Recently Madani \cite{18}, has considered the Yamab\'{e} problem
with singularities which he solved under some geometric conditions. Let $%
H_{2}^{2}(M)$ be the completion of the space $C^{\infty }(M)$ in $%
L^{2}\left( M\right) $, for the norm 
\begin{equation*}
\left\Vert u\right\Vert _{H_{2}^{2}(M)}=\left\Vert u\right\Vert
_{2}+\left\Vert \nabla u\right\Vert _{2}+\left\Vert \nabla ^{2}u\right\Vert
_{2}
\end{equation*}%
where $\left\Vert \nabla ^{k}u\right\Vert _{2}=\left( \int_{M}\left\vert
\nabla ^{k}u\right\vert ^{2}dv_{g}\right) ^{\frac{1}{2}}$ , $k=0,1,2$. In
all this paper, we denote by $H_{2}(M)$ the Sobolev's space $H_{2}^{2}(M)$
endowed with the suitable norm 
\begin{equation*}
\left\Vert u\right\Vert _{H_{2}(M)}=\left\Vert \Delta u\right\Vert
_{2}+\left\Vert \nabla u\right\Vert _{2}+\left\Vert u\right\Vert _{2}
\end{equation*}%
equivalent to $\left\Vert u\right\Vert _{H_{2}^{2}(M)}$. In this work we
deal with the $Q$-curvature type equations with singularities which is of
the form%
\begin{equation}
\Delta ^{2}u-\nabla ^{i}\left( a(x)\nabla _{i}u\right) +b(x)u=f\left\vert
u\right\vert ^{N-2}u  \label{222}
\end{equation}%
where the functions $a(x)$ and $b(x)$ are in $L^{s}(M)$, $s>\frac{n}{2}$ and
in $L^{p}(M)$, $p>\frac{n}{4}$ respectively, $N=\frac{2n}{n-4}$ is the
Sobolev critical exponent of the embedding $H_{2}^{2}\left( M\right)
\hookrightarrow L^{N}\left( M\right) $. In the particular case where the
coefficients are of the form $a(x)=\frac{\widetilde{a}(x)}{\rho ^{\gamma }}$
and $b(x)=\frac{\widetilde{b}(x)}{\rho ^{\alpha }}$ where $\widetilde{a}$
and $\widetilde{b}$ are smooth functions on $M$, $\rho $ denotes the
function defined below by (\ref{1}) and $\gamma $, $\alpha $ are real
numbers such that $0<\gamma <\frac{n}{s}<2$ and $0<\alpha <\frac{n}{p}<4$,
equation (\ref{222}) has singularities of order $\gamma $ and $\alpha $.
Mainly we give an existence result of solutions. Because of the lack of a
maximum principle for fourth order elliptic equations, finding positive
solutions is a hard question in the general case so to have positive
solutions which will be meaningfully in conformal geometry we restrict
ourself to special manifolds. Our results state as follows

\begin{theorem}
\label{th01} Let $\left( M,g\right) $ be a compact $n$-dimensional
Riemannian manifold, $n\geq 6$, $a\in L^{s}(M)$, $b\in L^{p}(M)$, with $s>%
\frac{n}{2}$, $p>\frac{n}{4}$, $f$ $\in C^{\infty }(M)$ a positive function
and $P\in M$ such that $f(P)=\max_{x\in M}f(x)$.

For $n\geq 10$,or $n=9$ and $\frac{9}{4}<p<11$ or $n=8$ and $2<p<5$ or $n=7$
and $\frac{7}{2}<s<9$ , $\frac{7}{4}<p<3$ we suppose that 
\begin{equation*}
\frac{n^{2}+4n-20}{6\left( n-6\right) (n^{2}-4)}R_{g}\left( P\right) -\frac{%
n-4}{2n\left( n-2\right) }\frac{\Delta f(P)}{f(P)}>0\text{.}
\end{equation*}%
For $n=6$ and $\frac{3}{2}<p<2$, $3<s<4$, we suppose that 
\begin{equation*}
R_{g}(P)>0\text{.}
\end{equation*}%
Then the equation (\ref{222}) has a non trivial weak solution $u$ in $%
H_{2}^{2}\left( M\right) $. Moreover if $a\in H_{1}^{s}\left( M\right) $,
then

$u\in $ $C^{0,\beta }$, for some $\beta \in \left( 0,1-\frac{n}{4p}\right) $.
\end{theorem}

For $R\in M$, we define the function $\rho $ on $M$ by

\begin{equation}
\rho (Q)=\left\{ 
\begin{array}{c}
d(R,Q)\text{ \ if \ \ \ \ \ \ }d(R,Q)<\delta (M) \\ 
\delta (M)\text{ \ if\ \ \ \ \ }d(R,Q)\geq \delta (M)%
\end{array}%
\right.  \label{1}
\end{equation}%
where $\delta (M)$ denotes the injectivity radius of $M$.

For real numbers $\gamma $ and $\alpha $, consider the equation in the
distribution sense

\begin{equation}
\Delta ^{2}u-\nabla ^{\mu }(\frac{a}{\rho ^{\gamma }}\nabla _{\mu }u)+\frac{%
bu}{\rho ^{\alpha }}=f\left\vert u\right\vert ^{N-2}u  \label{223}
\end{equation}%
where the functions $a$ and $b$ are smooth function $M$, and let 
\begin{equation*}
A=\left\{ u\in H_{2}(M):\int_{M}f\left\vert u\right\vert ^{N}dv_{g}=\left(
1+\left\Vert \frac{a}{\rho ^{\gamma }}\right\Vert _{s}+\left\Vert \frac{b}{%
\rho ^{\alpha }}\right\Vert _{p}\right) ^{\frac{N}{2}}\right\} .
\end{equation*}

If $\ 0<\gamma <\frac{n}{s}<2$ and $0<\alpha <\frac{n}{p}<4$ obviously $%
\frac{a}{\rho ^{\gamma }}\in L^{s}(M)$ and $\frac{b}{\rho ^{\alpha }}\in
L^{p}\left( M\right) $ so as a Corollary of Theorem \ref{th01}, we have

\begin{corollary}
\label{Cor2} Let $0<\gamma <\frac{n}{s}<2$ and $0<\alpha <\frac{n}{p}<4$.
For for $n\geq 8$, or $n=7$ and $\alpha >1$, suppose that 
\begin{equation}
\frac{n^{2}+4n-20}{6\left( n-6\right) (n^{2}-4)}R_{g}-\frac{n-4}{2n\left(
n-2\right) }\frac{\Delta f(P)}{f(P)}>0  \label{224}
\end{equation}%
and for $n=6$ and $\alpha >2$ suppose that 
\begin{equation*}
R_{g}(P)>0\text{.}
\end{equation*}%
Then the equation (\ref{223}) has a non trivial weak solution $u$ in $%
H_{2}\left( M\right) $. Moreover if $\ 0<\gamma <\frac{n}{s}-1$, then $u\in $
$C^{0,\beta }$, for some $\beta \in \left( 0,1\right) $.
\end{corollary}

For any $u\in H_{2}^{2}(M)$, we let 
\begin{equation*}
J_{\gamma ,\alpha }(u)=\int_{M}\left( \Delta u\right) ^{2}dv_{g}+\int_{M}%
\frac{a}{\rho ^{\gamma }}\left\vert \nabla u\right\vert ^{2}dv_{g}+\int_{M}%
\frac{b}{\rho ^{\alpha }}u^{2}dv_{g}
\end{equation*}%
be the energy functional and consider the Sobolev quotient: for any $u\in
H_{2}^{2}(M)-\left\{ 0\right\} $%
\begin{equation*}
Q_{\gamma ,\alpha }(u)=\frac{J_{\gamma ,\alpha }(u)}{\left( \int f\left\vert
u\right\vert ^{N}dv_{g}\right) ^{\frac{2}{N}}}\text{.}
\end{equation*}%
\begin{equation*}
A=\left\{ u\in H_{2}(M):\int_{M}f\left\vert u\right\vert ^{N}dv_{g}=\left(
1+\left\Vert \frac{a}{\rho ^{\gamma }}\right\Vert _{s}+\left\Vert \frac{b}{%
\rho ^{\alpha }}\right\Vert _{p}\right) ^{\frac{N}{2}}\right\}
\end{equation*}%
obviously $A\neq \phi $.

Put 
\begin{equation*}
Q_{\gamma ,\alpha }(M)=\inf_{u\in H_{2}(M)-\left\{ 0\right\} }Q_{\gamma
,\alpha }(u)=\inf_{u\in A}J_{\gamma ,\alpha }(u)
\end{equation*}%
and let $K\left( n,1,-2\right) $ and $K(n,2,-4)$ be the best constants in
the Hardy inclusion $H_{1}^{2}(M)\hookrightarrow L^{\frac{2n}{n-2}}\left(
M,\rho ^{-2}\right) $(obtained in \cite{18} ) and $H_{2}^{2}(M)%
\hookrightarrow L^{N}\left( M,\rho ^{-4}\right) $ ( see Lemma \ref{lem2} in
section2) respectively and where $N=\frac{2n}{n-4}$. In the sharp case $%
\gamma =2$, $\alpha =4$, we get

\begin{theorem}
Let $a$, $b$ and $f$ be smooth functions on $M$ with $f$ positive. Suppose
that 
\begin{equation*}
Q_{2,4}(M)K(n,2)^{2}\left( \sup_{x\in M}f\right) ^{\frac{2}{N}}<\left(
1+\left\Vert \frac{a}{\rho ^{\gamma }}\right\Vert _{s}+\left\Vert \frac{b}{%
\rho ^{\alpha }}\right\Vert _{p}\right) \left( 1+b(P)K(n,2,-4)^{2}\right) .
\end{equation*}%
If moreover we have 
\begin{equation*}
1+b(P)K(n,2,-4)^{2}>0
\end{equation*}%
and 
\begin{equation*}
1+a\left( P\right) K\left( n,1,-2\right) ^{2}+b(P)K(n,2,-4)^{2}>0
\end{equation*}

then the equation in the distribution sense%
\begin{equation*}
\Delta ^{2}u-\nabla ^{\mu }(\frac{a}{\rho ^{2}}\nabla _{\mu }u)+\frac{bu}{%
\rho ^{4}}=f\left\vert u\right\vert ^{N-2}u
\end{equation*}%
has a non trivial weak solution $u_{2,4}\in A$, which fulfilled $%
J_{2,4}(u)=Q_{2,4}(M)$.
\end{theorem}

When $\left( M,h\right) $ is a compact flat manifold, let $g=Ah$ where $%
A=e^{-\rho ^{2-\sigma }}$ and $0<\sigma <\frac{n}{p}-2<4$ and $\rho $ is
given by (\ref{1}), we give a geometric interpretation of ours results.
Denote by $H_{4}^{p}(M,T^{\ast }(M)\otimes T^{\ast }(M))$, with $p>\frac{n}{4%
}$, the Sobolev space of $H_{4}^{p}$- metrics on the manifold $M$. Obviously 
$g=Ah\in H_{4}^{p}(M,T^{\ast }(M)\otimes T^{\ast }(M))$.

\begin{theorem}
\label{th02} Suppose that the manifold $\left( M,h\right) $ is smooth
compact and flat of dimension $n>6$ and consider the conformal metric $g=Ah$%
, where $A=e^{-\rho ^{2-\sigma }}$ , $0<\sigma <\inf \left( \frac{n}{s}-1,%
\frac{n}{p}-2\right) $ where $s>\frac{n}{2}$, $p>\frac{n}{4}$ and $\rho $
defined by (\ref{1}), is supposed sufficiently small. Let $f$ be a $%
C^{\infty }$ positive function on $M$ and $P\in M$ such that $%
f(P)=\max_{x\in M}f(x)$ and $\rho \left( P\right) \neq 0.$

Suppose that 
\begin{equation*}
\frac{n^{2}+4n-20}{6\left( n-6\right) (n^{2}-4)}R_{g}\left( P\right) -\frac{%
n-4}{2n\left( n-2\right) }\frac{\Delta f(P)}{f(P)}>0
\end{equation*}%
then there exists a metric $\widetilde{g}\in H_{4}^{p}(M,T^{\ast }(M)\otimes
T^{\ast }(M))$ conformal to $g$ such that $f$ is the $Q$-curvature of the
manifold $\left( M,\widetilde{g}\right) $.
\end{theorem}

Our paper is organized as follows: in the first section, we give an Hardy
inequality on compact manifolds, in the second one we establish the
regularity of the Paneitz-Branson operator which leads us to construct the
Green's function to the Schr\"{o}dinger biharmonic operator this latter
allows us to obtain a priori estimates to a solution of some biharmonic
equation. The third section is devoted to the study of the $Q$-curvature
equation with singularities of order $0<\gamma <2$ and $0<\alpha <4$. In the
fourth section we consider the sharp singularities i.e. $\gamma =2$ and $%
\alpha =4$. In the last section, we give an interpretation in conformal
geometry.

\section{Hardy inequality on compact manifolds}

Let $(M,g)$ be a compact $n$-dimensional Riemannian manifold. We consider
the space $L^{p}(M,\rho ^{\gamma })$ , $1\leq p\leq \infty $, of measurable
functions $u$ on $M$ such that

\begin{equation}
\left\Vert u\right\Vert _{p,\rho ^{\gamma }}^{p}=\int_{M}\rho ^{\gamma
}\left\vert u\right\vert ^{p}dv_{g}<+\infty  \label{2}
\end{equation}%
where $\rho $ is the function defined by (\ref{1}) and $\gamma \in R$.

The space $L^{p}(M,\rho ^{\gamma })$, endowed with the norm $\left\Vert
u\right\Vert _{p,\rho ^{\gamma }}^{p}$, is a Banach space and we have

\begin{lemma}
\label{lm1}( Hardy inequality )

For any function $u\in C_{o}^{\infty }(R^{n})$, there exists a constant $C>0$
such that

\begin{equation}
\left\Vert \left\vert x\right\vert ^{\frac{\gamma }{p}}u\right\Vert _{p}\leq
C\left\Vert \left\vert x\right\vert ^{\beta }\nabla ^{l}u\right\Vert _{q}
\label{3}
\end{equation}

where $p$, $q$ and $\gamma $ real numbers such that 
\begin{equation*}
1\leq q\leq p\leq \frac{nq}{n-lq}\text{, }n>lq\text{, }\frac{\gamma }{p}%
=\beta -l+n\left( \frac{1}{q}-\frac{1}{p}\right) >-\frac{n}{p}\text{.}
\end{equation*}
\end{lemma}

The optimal constant in (\ref{3}) will be denoted by $K(n,l,\gamma ,\beta )$
and $K(n,l,\gamma )$ when $\beta =0$.

From the above lemma we infer,

\begin{lemma}
\label{lem2} Let $\left( M,g\right) $ be a compact $n$- dimensional
Riemannian manifold, and $p$, $q$ and $\gamma $ real numbers satisfying

\begin{equation*}
1\leq q\leq p\leq \frac{nq}{n-2q}\text{, }n>2q\text{, }\frac{\gamma }{p}%
=-2+n\left( \frac{1}{q}-\frac{1}{p}\right) >-\frac{n}{p}\text{.}
\end{equation*}%
For any $\varepsilon >0$, there is a constant $A(\varepsilon ,q,\gamma )$
such that

\begin{equation*}
\forall f\in H_{2}^{q}(M)\text{, \ }\left\Vert f\right\Vert _{p,\rho
^{\gamma }}^{q}\leq \left( 1+\varepsilon \right) K^{q}(n,q,\gamma
)\left\Vert \nabla ^{2}f\right\Vert _{q}^{q}+A(\varepsilon ,q,\gamma
)\left\Vert f\right\Vert _{q}^{q}\text{.}
\end{equation*}%
In particular in case $\gamma =0$, $K(n,q,0)=K(n,q)$ is the best constant in
Sobolev's inequality.
\end{lemma}

\begin{proof}
Let $\left\{ B_{i}\right\} _{1\leq i\leq m}$ be a finite covering of $M$ by
geodesic balls of small radius $\delta >0$, $\left\{ \left( B_{i},\varphi
_{i}\right) \right\} _{i}$, where $\varphi _{i}=\exp _{p_{i}}^{-1}$, is an
associated atlas and $\left( a_{i}\right) _{1\leq i\leq m}$ is a partition
of unity subordinated to the covering $\left\{ B_{i}\right\} _{1\leq i\leq
m} $.

Let\ $f\in C^{\infty }(M)$%
\begin{equation*}
\left\Vert f\right\Vert _{p,\rho ^{\gamma }}^{q}=\left\Vert \left\vert
f\right\vert ^{q}\right\Vert _{\frac{p}{q},\rho ^{\gamma }}=\left\Vert
\sum_{i=1}^{m}a_{i}\left\vert f\right\vert ^{q}\right\Vert _{\frac{p}{q}%
,\rho ^{\gamma }}\leq \sum_{i=1}^{m}\left\Vert a_{i}^{\frac{1}{q}%
}f\right\Vert _{p,\rho ^{\gamma }}^{q}\text{.}
\end{equation*}%
The function $\widetilde{f}=\left( a_{i}f\right) o\varphi _{i}^{-1}$ may be
considered as a function defined on $R^{n}$ by extending it by $0$ outside
the support, so applying inequality (\ref{3}), with $\beta =0$ and $l=2$ to $%
\widetilde{f}$, we get%
\begin{equation}
\left( \int_{R^{n}}\left\vert x\right\vert ^{\gamma }\left\vert \widetilde{f}%
\right\vert ^{p}dx\right) ^{\frac{1}{p}}\leq K\left( n,2,p,\gamma \right)
\left( \int_{R^{n}}\left\vert \nabla ^{2}\widetilde{f}\right\vert
^{q}dx\right) ^{\frac{1}{q}}\text{.}  \label{4}
\end{equation}%
Now, we have to express the derivatives of the function $\widetilde{f}$ in
terms of the Euclidean derivatives. If the coordinate system is normal at a
point $P\in M$, the expansion of the metric tensor at $P$ writes as 
\begin{equation*}
g_{ij}\left( Q\right) =\delta _{ij}+O\left( \rho ^{2}\right)
\end{equation*}%
where $\rho =d(P,Q)<\delta (M)$, ($\delta (M)$ is the injectivity radius of $%
M$) and the expansions of the Christoffel symbols are of the form%
\begin{equation*}
\Gamma _{ij}^{k}(Q)=O(\rho )\text{.}
\end{equation*}%
Now since 
\begin{equation*}
\nabla _{ij}\widetilde{f}\left( Q\right) =\frac{\partial ^{2}\widetilde{f}}{%
\partial x_{i}\partial x_{j}}\left( Q\right) -\Gamma _{ij}^{s}\left(
Q\right) \frac{\partial \widetilde{f}}{\partial x_{s}}\left( Q\right)
\end{equation*}%
we obtain that%
\begin{equation*}
g^{ik}\left( Q\right) g^{jl}\left( Q\right) \nabla _{ij}\widetilde{f}\left(
Q\right) \nabla _{kl}\widetilde{f}\left( Q\right) \leq \left( 1+O\left(
\delta ^{2}\right) \left\vert \nabla _{E}^{2}\widetilde{f}\left( Q\right)
\right\vert ^{2}\right)
\end{equation*}%
\begin{equation*}
+\left\vert \nabla _{E}^{2}\widetilde{f}\left( Q\right) \right\vert
\left\vert \nabla _{E}\widetilde{f}\left( Q\right) \right\vert O\left(
\delta \right) +\left\vert \nabla _{E}\widetilde{f}\left( Q\right)
\right\vert ^{2}O\left( \delta ^{2}\right) \text{.}
\end{equation*}%
To estimate the rectangular term, we use the inequality%
\begin{equation*}
ab\leq \eta a^{2}+\frac{b^{2}}{4\eta }
\end{equation*}%
valid for any positive real numbers $a$, $b$, $\eta $ and get for any $%
\varepsilon >0$, there is a constant $C(\varepsilon )$ such that

\begin{equation}
\left\vert \nabla ^{2}\widetilde{f}\left( Q\right) \right\vert ^{2}\leq
\left( 1+\varepsilon \right) \left\vert \nabla _{E}^{2}\widetilde{f}\left(
Q\right) \right\vert ^{2}+C(\varepsilon )\left\vert \nabla _{E}\widetilde{f}%
\left( Q\right) \right\vert ^{2}  \label{5}
\end{equation}%
and similarly, we obtain%
\begin{equation}
\left\vert \nabla _{E}^{2}\widetilde{f}\left( Q\right) \right\vert ^{2}\leq
\left( 1+\varepsilon \right) \left\vert \nabla ^{2}\widetilde{f}\left(
Q\right) \right\vert ^{2}+C(\varepsilon )\left\vert \nabla \widetilde{f}%
\left( Q\right) \right\vert ^{2}\text{.}  \label{6}
\end{equation}%
Since $M$ is compact there exist constants $\lambda $, $\mu $ with 
\begin{equation*}
0<\lambda \leq \sqrt{\left\vert g\right\vert }\leq \mu \text{.}
\end{equation*}%
Consequently%
\begin{equation*}
I_{i}=\int_{M}\rho ^{\gamma }\left\vert a_{i}^{\frac{1}{q}}f\right\vert
^{p}dv_{g}=\int_{R^{n}}\left\vert x\right\vert ^{\gamma }\left\vert 
\widetilde{f}\right\vert ^{p}\sqrt{\left\vert g\right\vert }dx\leq \mu
\int_{R^{n}}\left\vert x\right\vert ^{\gamma }\left\vert \widetilde{f}%
\right\vert ^{p}dx
\end{equation*}%
where $\left\vert x\right\vert =d(P,Q)<\delta (M)$ denotes the injectivity
radius and $Q=\exp _{P}x$. So by the inequality (\ref{4}) we write%
\begin{equation*}
I_{i}\leq \mu K(n,q,\gamma )^{p}\left( \int_{R^{n}}\left\vert \nabla _{E}^{2}%
\widetilde{f}\right\vert ^{q}dx\right) ^{\frac{p}{q}}\text{.}
\end{equation*}%
Taking account of the inequality (\ref{6}) and of the following inequality,
let $a$, $b$, $s\geq 1$ be positive real numbers, for any $\varepsilon
_{1}>0 $, there is a constant $C\left( \varepsilon _{1},s\right) $ such that 
\begin{equation}
\left( a+b\right) ^{s}\leq \left( 1+\varepsilon _{1}\right)
a^{s}+C(\varepsilon _{1},s)b^{s}  \label{7}
\end{equation}%
we obtain%
\begin{equation*}
I_{i}^{\frac{q}{p}}\leq \mu ^{\frac{q}{p}}K\left( n,q,\gamma \right)
^{q}\left( \int_{R^{n}}\left\vert \nabla _{E}^{2}\widetilde{f}\right\vert
^{q}dx\right)
\end{equation*}%
\begin{equation*}
\leq \left( \mu ^{\frac{q}{p}}\lambda ^{-1}\right) K\left( n,q,\gamma
\right) ^{q}\left\{ \left( 1+\varepsilon \right) \left\Vert \nabla ^{2}%
\widetilde{f}\right\Vert _{q}^{q}+C(\varepsilon ,q,\gamma )\left\Vert \nabla 
\widetilde{f}\right\Vert _{q}^{q}\right\} \text{.}
\end{equation*}%
Now since $\widetilde{f}=a_{i}^{\frac{1}{q}}f$, there exits a constant $c>0$
such that%
\begin{equation*}
\left\vert \nabla \widetilde{f}\right\vert \leq c\left\vert f\right\vert
+a_{i}^{\frac{1}{q}}\left\vert \nabla f\right\vert
\end{equation*}%
and%
\begin{equation*}
\left\vert \nabla ^{2}\widetilde{f}\right\vert \leq c\left( \left\vert
f\right\vert +\left\vert \nabla f\right\vert \right) +a_{i}^{\frac{1}{q}%
}\left\vert \nabla ^{2}f\right\vert \text{.}
\end{equation*}%
From the interpolation formula ( see \cite{1} page 93 ), for any $\eta >0$,
there is a constant $C(\eta )$ such that%
\begin{equation}
\left\Vert \nabla f\right\Vert _{q}^{q}\leq \eta \left\Vert \nabla
^{2}f\right\Vert _{q}^{q}+C(\eta )\left\Vert f\right\Vert _{q}^{q}  \label{8}
\end{equation}%
applying the inequality (\ref{7}) and letting the injectivity radius small
enough so that $\lambda $ and $\mu $ are close to $1$, we get, for any $%
\varepsilon >0$ there is a constant $A(\varepsilon ,q,\gamma )$%
\begin{equation*}
\left\Vert f\right\Vert _{p,\rho ^{\gamma }}^{q}\leq \left( 1+\varepsilon
\right) K\left( n,q,\gamma \right) ^{q}\left\Vert \nabla ^{2}f\right\Vert
_{q}^{q}+A(\varepsilon ,p,q,\gamma )\left\Vert f\right\Vert _{q}^{q}\text{.}
\end{equation*}
\end{proof}

As a corollary of Lemma \ref{lem2}, we have the following Sobolev inequality

\begin{lemma}
\label{lem3} Let $\left( M,g\right) $ be a compact $n$- dimensional
Riemannian manifold, $n\geq 5$, and $p$, $\gamma $ real numbers satisfying

\begin{equation*}
2\leq p\leq \frac{2n}{n-4}\text{, }\frac{\gamma }{p}=-2+n\left( \frac{1}{2}-%
\frac{1}{p}\right) >-\frac{n}{p}\text{.}
\end{equation*}%
For any $\varepsilon >0$, there is a constant $A(\varepsilon ,q,\gamma )$
such that

\begin{equation*}
\forall f\in H_{2}(M)\text{, \ }\left\Vert f\right\Vert _{p,\rho ^{\gamma
}}^{2}\leq \left( 1+\varepsilon \right) K(n,\gamma )^{2}\left\Vert \Delta
f\right\Vert _{2}^{2}+A(\varepsilon ,\gamma )\left\Vert f\right\Vert _{2}^{2}%
\text{.}
\end{equation*}
\end{lemma}

In fact, it is known on compact manifold (see \cite{1} page 115) that there
is a constant $c>0$ such that 
\begin{equation}
\left\Vert \nabla ^{2}f\right\Vert _{2}^{2}\leq \left\Vert \Delta
f\right\Vert _{2}^{2}+c\left\Vert \nabla f\right\Vert _{2}^{2}  \label{9}
\end{equation}%
and our formula follows from the interpolation formula (\ref{8}).

Now we derive a Kondrakov type result

\begin{lemma}
\label{lem4} Let $(M,g)$ be a compact $n$- dimensional Riemannian manifold
and $p$, $q$, $\gamma <0$ real numbers such that 
\begin{equation*}
1\leq q\leq p\leq \frac{nq}{n-2q}
\end{equation*}%
If $\frac{\gamma }{p}$ $=-2+n\left( \frac{1}{q}-\frac{1}{p}\right) >-\frac{n%
}{p}$, the inclusion $H_{2}^{q}\left( M\right) \subset L^{p}\left( M,\rho
^{\gamma }\right) $ is continuous.

If $\frac{\gamma }{p}$ $>-2+n\left( \frac{1}{q}-\frac{1}{p}\right) $, the
inclusion $H_{2}^{q}\left( M\right) \subset L^{p}\left( M,\rho ^{\gamma
}\right) $ is compact.
\end{lemma}

\begin{proof}
The first part of Lemma \ref{lem4} is a consequence of Lemma \ref{lem2}. To
prove the second part of Lemma \ref{lem4}, we consider the following
inclusions $H_{2}^{q}(M)\subset L^{r}(M)\subset L^{p}(M,\rho ^{\gamma })$
and we have to show that the first inclusion is compact and the second one
is continuous. By the Kondrakov's theorem we must have 
\begin{equation*}
\frac{1}{r}>\frac{1}{q}-\frac{2}{n}\text{.}
\end{equation*}%
The H\"{o}lder's inequality allows us to write%
\begin{equation}
\int_{M}\rho ^{\gamma }\left\vert u\right\vert ^{p}dv_{g}\leq \left(
\int_{M}\left\vert u\right\vert ^{r}dv_{g}\right) ^{\frac{p}{r}}\left(
\int_{M}\rho ^{\gamma r^{\prime }}dv_{g}\right) ^{\frac{1}{r^{\prime }}}
\label{10}
\end{equation}%
with $r^{\prime }=\frac{r}{r-p}$ and $r>p$.

The second integral in the right-hand side of (\ref{10}) will converge if 
\begin{equation*}
\gamma r^{\prime }=\frac{\gamma r}{n-p}>-n\text{ \ \ i.e. \ \ }\frac{1}{r}<%
\frac{\gamma +n}{np}
\end{equation*}%
so by (\ref{10}) we must have%
\begin{equation*}
\frac{1}{q}-\frac{2}{n}<\frac{\gamma +n}{np}
\end{equation*}%
i.e. 
\begin{equation*}
\frac{\gamma }{p}>-2+n\left( \frac{1}{q}-\frac{1}{p}\right) \text{.}
\end{equation*}
\end{proof}

Now, we quote the following Lemma due to Djadli-Hebey-Ledoux (\cite{10}) and
improved by Hebey (\cite{12}) which will be used in the sequel of this paper.

\begin{lemma}
\label{lem5} Let $M$ be a Riemannian compact manifold with dimension $n\geq
5 $. For any $\epsilon >0$ there is a constant $A_{2}(\epsilon )$ such that
for any $u\in H_{2}^{2}\left( M\right) $, \ $\left\Vert u\right\Vert
_{N}^{2}\leq (1+\epsilon )K\left( n,2\right) ^{2}\left\Vert \Delta
u\right\Vert _{2}^{2}+A_{2}(\epsilon )\left\Vert u\right\Vert _{2}^{2}$ ,
where $K\left( n,2\right) $ is the best Sobolev's constant in the Sobolev's
embedding $H_{2}^{2}\left( M\right) \hookrightarrow L^{N}(M)$ and where $N=%
\frac{2n}{n-4}$.
\end{lemma}

\section{Regularity Theorems for Paneitz-Branson type equation}

In this section we give regularity theorems for solutions to Paneiz-Branson
type equation to do so, first we construct Green's function, we follow
Aubin's construction for the Laplacian (\cite{1}) and Madani's one(\cite{18}%
) for Yamabe's operator. We will need the following Giraud's lemma (see \cite%
{14}).

\begin{lemma}
\label{lem6} \bigskip Let $\Phi (x,y)=\int_{M}\psi (x,z)\phi (z,y)dv_{g}(z)$
and suppose that $\left\vert \psi (x,z)\right\vert \leq Cd(x,z)^{i-n}$ and $%
\left\vert \phi (z,y)\right\vert \leq Cd(z,y)^{j-n}$, where $0<i,j<n$. \
Then there exists a constant $C^{\prime }<+\infty $ such that 
\begin{equation*}
\left\vert \Phi (x,y)\right\vert \leq C^{\prime }\left\{ 
\begin{array}{c}
d\left( x,y\right) ^{i+j-n\ \ \ }\text{, \ if\ \ \ \ \ \ }i+j<n \\ 
\left( 1+\left\vert \log d(x,y\right\vert \right) \ \ \text{if \ }i+j=n \\ 
1\ \ \ \ \ \ \ \ \ \ \text{if \ \ \ \ \ \ \ \ \ \ \ \ \ \ \ \ \ \ \ }i+j>n%
\end{array}%
\right. \text{.}
\end{equation*}
\end{lemma}

\begin{corollary}
\label{Cor1} For $p>\frac{n}{4}$ and $4\leq i,j<n$, let $b\in L^{p}(M)$ and $%
\Theta (x,y)=\int_{M}\psi (x,z)b(z)\phi (z,y)dv_{g}(z)$. Then there is a
constant $C^{\prime }<+\infty $ such that 
\begin{equation*}
\left\vert \Theta (x,y)\right\vert \leq C^{\prime }\left\Vert b\right\Vert
_{p}\left\{ 
\begin{array}{c}
d\left( x,y\right) ^{\left( i+j\right) -n\left( 1+\frac{1}{p}\right) \ \ \ }%
\text{, \ if\ \ \ \ \ \ }i+j<\left( 1+\frac{1}{p}\right) n \\ 
\left( 1+\left\vert \log d(x,y\right\vert \right) \ \ \text{if \ }i+j=\left(
1+\frac{1}{p}\right) n \\ 
1\ \ \ \ \ \ \ \ \ \text{if \ \ \ \ \ \ \ \ \ \ \ \ \ \ \ \ \ \ \ }%
i+j>\left( 1+\frac{1}{p}\right) n%
\end{array}%
\right. \text{.}
\end{equation*}
\end{corollary}

Indeed, 
\begin{equation*}
\left\vert \Theta (x,y)\right\vert \leq \int_{M}\left\vert \psi
(x,z)b(z)\phi (z,y)\right\vert dv_{g}(z)
\end{equation*}%
\begin{equation*}
\leq \left\Vert b\right\Vert _{p}\left( \int_{M}\left\vert \psi (x,z)\phi
(z,y)\right\vert ^{\frac{p}{p-1}}dv_{g}\left( z\right) \right) ^{1-\frac{1}{p%
}}
\end{equation*}%
so $\left\vert \psi (x,z)\right\vert ^{\frac{p}{p-1}}\leq Cd(x,z)^{\left(
i-n\right) \frac{p}{p-1}}=Cd(x,z)^{-n+\frac{pi-n}{p-1}}$. Since \ $p>\frac{n%
}{4}$, then $pi-n>n\left( \frac{i}{4}-1\right) \geq 0$, and Corollary \ref%
{Cor1} follows from Lemma \ref{lem6}.

The Green function of the Laplacian operator $\ G$ : $M\times M\rightarrow R$%
, $\left( x,y\right) \rightarrow G(x,y)$ is defined as the solution in the
distribution sense to the equation

\begin{equation}
\Delta G\left( x,.\right) =\delta _{x}-\frac{1}{V(M)}  \label{11}
\end{equation}%
where $V(M)$ is the Riemannian volume of $M$. So for any $\varphi \in
C^{\infty }(M)$, we have

\begin{equation}
\Delta \varphi (x)=\int_{M}G(x,y)\Delta ^{2}\varphi \left( y\right) dv_{g}(y)%
\text{.}  \label{12}
\end{equation}%
Multiplying (\ref{12}) by $G(x,.)$ and integrating over $M$, we get 
\begin{equation}
\varphi (x)=\int_{M\times M}G(x,y)G(y,z)\Delta ^{2}\varphi (z)dv_{g}\left(
y\right) dv_{g}(z)+\frac{1}{V(M)}\int_{M}\varphi (y)dv_{g}(y)  \label{13}
\end{equation}%
and letting 
\begin{equation*}
G_{2}(x,y)=\int_{M}G(x,z)G(z,y)dv_{g}(z)
\end{equation*}%
(\ref{13}) becomes 
\begin{equation*}
\varphi (x)=\int_{M}G_{2}(x,y)\Delta ^{2}\varphi (y)dv_{g}(y)+\frac{1}{V(M)}%
\int_{M}\varphi (y)dv_{g}(y)\text{.}
\end{equation*}%
According to Lemma \ref{lem6}, 
\begin{equation*}
\left\vert G_{2}(x,y)\right\vert \leq Cd(x,y)^{4-n}\text{ \ \ \ if }n>4\text{
and }x\neq y\text{.}
\end{equation*}%
Hence $G_{2}$ is the Green function of the biharmonic operator.

Let $b\in L^{p}\left( M\right) $ with $p>1$ and put $\widetilde{\Gamma }%
_{o}(x,y)=G_{2}(x,y)$ and define recursively for $j\geq 1$,

\begin{equation*}
\widetilde{\Gamma }_{j}(x,y)=-\int_{M}\widetilde{\Gamma }_{j-1}(x,z)b(z)%
\widetilde{\Gamma }_{o}(z,y)dv_{g}(z)
\end{equation*}%
$\widetilde{\Gamma }_{j}$ is well defined and by Corollary \ref{Cor1}, the
following estimates hold for any $y\neq x$ \ 
\begin{equation}
\left\vert \widetilde{\Gamma }_{j}(x,y)\right\vert \leq \left\{ 
\begin{array}{c}
C_{j}\left\Vert b\right\Vert _{p}d(x,y)^{4(j+1)-n\left( 1+\frac{1}{p}\right)
}\text{ \ if }(j+1)<\frac{n}{4}\left( 1+\frac{1}{p}\right) \\ 
C_{j}\left\Vert b\right\Vert _{p}\left( 1+\left\vert \log d(x,y)\right\vert
\right) \text{ if }(j+1)=\frac{n}{4}\left( 1+\frac{1}{p}\right) \\ 
C_{j}\left\Vert b\right\Vert _{p}\text{ \ \ \ \ \ \ \ \ \ \ \ \ \ \ \ \ \ \
\ \ \ \ \ \ \ \ \ \ if\ \ }(j+1)>\frac{n}{4}\left( 1+\frac{1}{p}\right)%
\end{array}%
\right. \text{.}  \label{14'}
\end{equation}%
Let $j_{o}$ such that $\ \ \ $%
\begin{equation*}
\ j_{o}+1=E\left( \frac{n}{4}\left( \frac{1}{p}+\frac{1}{s}\right) \right)
\end{equation*}%
where $E\left( x\right) $ denotes the integer part of $x$ and let 
\begin{equation*}
\Gamma _{j}=\widetilde{\Gamma }_{o}\text{,\ for }j=0\text{,..., }j_{o}-1%
\text{,\ \ \ \ \ }\Gamma _{j}=\widetilde{\Gamma }_{j}\text{, \ }j=j_{o}\text{%
, }j_{o}+1\text{, ... }
\end{equation*}%
For $y\in M$, we consider a function $u_{y}\in H_{2}(M)$ which will be
determined later and define, for $m>\frac{n}{4}\left( 1+\frac{1}{p}\right) $%
\begin{equation*}
H(y,.)=\Gamma _{o}(y,.)+\sum_{j=j_{o}}^{m}\Gamma _{j}(y,.)+u_{y}\text{.}
\end{equation*}%
Since $p>\frac{n}{4}$, thank to Corollary \ref{Cor1} we get%
\begin{equation*}
\int_{M-B_{\epsilon }\left( y\right) }\left\vert \Gamma _{j}(y,z)\right\vert
dv_{g}\left( z\right) \leq C_{j}\left\Vert b\right\Vert
_{p}^{j}\int_{M-B_{\epsilon }\left( y\right) }d(y,z)^{4(j+1)-n\left( 1+\frac{%
1}{p}\right) }dv_{g}\left( z\right)
\end{equation*}%
\begin{equation*}
\leq C_{j}\left\Vert b\right\Vert _{p}^{j}\omega _{n-1}\int_{\epsilon
}^{\delta (M)}r^{4(j+1)-\left( 1+\frac{n}{p}\right) }dr=C_{j}\left\Vert
b\right\Vert _{p}^{j}\omega _{n-1}\left( \delta (M)^{4(j+1)-\frac{n}{p}%
}-\epsilon ^{^{4(j+1)-\frac{n}{p}}}\right)
\end{equation*}%
where $r=d(y,z)$; and since $p>\frac{n}{4}$, $4(j+1)-\frac{n}{p}>0$.

Hence%
\begin{equation*}
\left\Vert \Gamma _{j}(y,.)\right\Vert _{1}<+\infty
\end{equation*}%
and $\ H(y,.)\in H_{2}(M-\left\{ y\right\} )\cap L^{1}\left( M\right) $.

$H(y,.)$ will be a Green function to $P(u)=\Delta ^{2}u-\nabla ^{\mu
}(a\nabla _{\mu }u)+bu$ on $M$ if for every $\varphi \in C^{4}(M)$, $u_{y}$
solves the following equation%
\begin{equation*}
\varphi \left( y\right) =\int_{M}H(y,z)P\left( \varphi \right) (z)dv_{g}(z)+%
\frac{1}{V\left( M\right) }\int_{M}\varphi \left( z\right) dv_{g}\left(
z\right)
\end{equation*}%
\begin{equation*}
=\int_{M}H(y,z)\Delta ^{2}\varphi (z)dv_{g}(z)-\int_{M}H(y,z)\left( \nabla
^{\mu }\left( a\nabla _{\mu }\varphi \right) \right) (z)dv_{g}(z)+
\end{equation*}%
\begin{equation}
\int_{M}H(y,z)b(z)\varphi (z)dv_{g}(z)+\frac{1}{V\left( M\right) }%
\int_{M}\varphi \left( z\right) dv_{g}\left( z\right) \text{.}  \label{14}
\end{equation}%
The first integral of (\ref{14}) reads%
\begin{equation*}
\int_{M}H(y,z).\Delta ^{2}\varphi (z)dv_{g}(z)=\int_{M}\Gamma
_{o}(y,z)\Delta ^{2}\varphi (z)dv_{g}\left( z\right)
\end{equation*}%
\begin{equation*}
-\sum_{j=j_{o}}^{m}\int_{M}\Gamma _{j-1}(y,u)b(u)\int_{M}\Gamma
_{o}(u,v)\Delta ^{2}\varphi (u)dv_{g}(v)dv_{g}\left( u\right)
+\int_{M}\Delta _{z}^{2}u_{y}(z).\varphi \left( z\right) dv_{g}\left(
z\right)
\end{equation*}%
\begin{equation}
=\varphi (y)-\frac{1}{V(M)}\int_{M}\varphi \left( x\right) dv_{g}\left(
x\right) +\frac{1}{V(M)}\int_{M}\varphi
(x)dv_{g}(x)\sum_{j=j_{o}}^{m}\int_{M}\Gamma _{j-1}(y,z)b(z)dv_{g}(z)
\label{15}
\end{equation}%
\begin{equation*}
-\sum_{j=j_{o}}^{m}\int_{M}\Gamma _{j-1}(y,z)b(z)\varphi \left( z\right)
dv_{g}(z)+\int_{M}\Delta _{R}^{2}u_{Q}(x).\varphi \left( x\right)
dv_{g}\left( x\right) \text{.}
\end{equation*}%
Plugging (\ref{15}) into (\ref{14}), we get%
\begin{equation*}
\varphi (y)=
\end{equation*}%
\begin{equation*}
=\varphi (y)+\int_{M}\left( \Delta _{R}^{2}u_{y}(z)-\left( \nabla ^{\mu
}\left( a\nabla _{\mu }u_{y}\right) \right) (z)+b(z)u_{y}\left( z\right)
\right) \varphi \left( z\right) dv_{g}\left( z\right)
\end{equation*}%
\begin{equation*}
-\int_{M}\Gamma _{o}(y,z)\left( \nabla ^{\mu }\left( a\nabla _{\mu }\varphi
\right) \right) (z)dv_{g}(z)
\end{equation*}%
\begin{equation*}
-\sum_{j=j_{o}}^{m}\int_{M}\Gamma _{j}(y,z)\left( \nabla ^{\mu }\left(
a\nabla _{\mu }\varphi \right) \right) (z)dv_{g}(z)
\end{equation*}%
\begin{equation*}
+\int_{M}\Gamma _{m}(y,z)b(z)\varphi (z)dv_{g}(z)\text{.}
\end{equation*}%
Let $B_{\epsilon }(x)\subset M$ be the geodesic ball centred at $Q$ and of
radius $\varepsilon $ ($\varepsilon $ small enough). We have 
\begin{equation*}
\int_{M-B_{\epsilon }(x)}\Gamma _{j}(y,z)\left( \nabla ^{\mu }\left( a\nabla
_{\mu }\varphi \right) \right) (z)dv_{g}(z)=-\int_{M-B_{\epsilon
}(P)}a(z)\left\langle \nabla \Gamma _{j}\left( y,z\right) ,\nabla \varphi
(z)\right\rangle dv_{g}(z)
\end{equation*}%
\begin{equation*}
-\int_{\partial B_{\epsilon }(P)}\Gamma _{j}(y,z)a(z)\partial _{\nu }\varphi
(z)d\sigma (z)
\end{equation*}%
where $\partial _{\nu }\varphi (z)=\left\langle \nabla \varphi ,\nu
\right\rangle \left( z\right) $ and $\nu $ is the outward normal unit vector
to the sphere $B_{\epsilon }(x)$ at the point $z$. 
\begin{equation*}
\int_{\partial B_{\epsilon }(x)}\left\vert \Gamma _{j}(y,z)a(z)\partial
_{\nu }\varphi (z)\right\vert d\sigma (z)\leq \max_{\partial B_{\epsilon
}(x)}\left\vert \partial _{\nu }\varphi \right\vert \left\Vert a\right\Vert
_{s}\left( \int_{\partial B_{\epsilon }(x)}\left\vert \Gamma
_{j}(y,z)\right\vert ^{\frac{s}{s-1}}d\sigma (z)\right) ^{1-\frac{1}{s}}%
\text{.}
\end{equation*}%
On the other hand%
\begin{equation*}
\left( \int_{\partial B_{\epsilon }(x)}\left\vert \Gamma
_{j}(y,z)\right\vert ^{\frac{s}{s-1}}d\sigma (z)\right) ^{1-\frac{1}{s}%
}=\left( \int_{\partial B_{\epsilon }(P)}\left\vert \int_{M}\Gamma
_{j-1}\left( y,u\right) b(u)\Gamma _{o}\left( u,z\right) dv_{g}\left(
u\right) \right\vert ^{\frac{s}{s-1}}d\sigma (z)\right) ^{1-\frac{1}{s}}
\end{equation*}%
\begin{equation*}
\leq C_{j}\left\Vert b\right\Vert _{p}^{j}\left[ \int_{\partial B_{\epsilon
}(P)}\left( \int_{M}\left\vert \Gamma _{j-1}\left( y,u\right) \Gamma
_{o}\left( u,z\right) \right\vert ^{\frac{p}{p-1}}dv_{g}\left( u\right)
\right) ^{\frac{s}{s-1}\left( 1-\frac{1}{p}\right) }d\sigma (z)\right] ^{1-%
\frac{1}{s}}
\end{equation*}%
\begin{equation*}
\leq C_{j}^{^{\prime }}\left\Vert b\right\Vert _{p}^{\left( 2j-1\right) 
\frac{p}{p-1}}\left[ \int_{\partial B_{\epsilon }(P)}d\left( y,z\right) ^{%
\left[ 4\left( j+1\right) -n\left( 1+\frac{1}{p}\right) \right] \frac{s}{s-1}%
}d\sigma (z)\right] ^{1-\frac{1}{s}}\text{.}
\end{equation*}%
Since $j+1\geq \frac{n}{4}\left( \frac{1}{p}+\frac{1}{s}\right) $, we infer
that%
\begin{equation*}
\left( \int_{\partial B_{\epsilon }(x)}\left\vert \Gamma
_{j}(y,z)\right\vert ^{\frac{s}{s-1}}d\sigma (z)\right) ^{1-\frac{1}{s}}\leq
\max_{M}\left\vert \partial _{\nu }\varphi \right\vert \left\Vert
a\right\Vert _{s}C_{j}\left\Vert b\right\Vert _{p}^{\left( 2j-1\right) \frac{%
p}{p-1}}\epsilon ^{4(j+1)-n\left( \frac{1}{s}+\frac{1}{p}\right)
}=o(\epsilon )\text{. }
\end{equation*}%
Hence%
\begin{equation*}
\int_{M-B_{\epsilon }(P)}\Gamma _{j}(Q,R)\left( \nabla ^{\mu }\left( a\nabla
_{\mu }\varphi \right) \right) (R)dv_{g}(R)=
\end{equation*}%
\begin{equation*}
-\int_{M-B_{\epsilon }(P)}a(R)\left\langle \nabla \Gamma _{j}\left(
Q,R\right) ,\nabla \varphi (R)\right\rangle dv_{g}(R)+o(1)\text{ as\ \ \ }%
\epsilon \rightarrow 0
\end{equation*}%
also%
\begin{equation*}
\int_{M-B_{\epsilon }(P)}\left\vert a(R)\left\langle \nabla \Gamma
_{j}\left( Q,R\right) ,\nabla \varphi (R)\right\rangle \right\vert
dv_{g}(R)\leq
\end{equation*}%
\begin{equation*}
C_{j}\left\Vert b\right\Vert _{p}^{j}\left\Vert a\right\Vert _{s}\left(
\int_{M-B_{\epsilon }(P)}\left( \left\vert \nabla \Gamma _{j}\left(
Q,R\right) \right\vert \left\vert \nabla \varphi (R)\right\vert \right) ^{%
\frac{s}{s-1}}dv_{g}(R)\right) ^{1-\frac{1}{s}}
\end{equation*}%
\begin{equation*}
\leq C_{j}\left\Vert b\right\Vert _{p}^{j}\left\Vert a\right\Vert _{s}\left(
\int_{M-B_{\epsilon }(P)}\left\vert \nabla _{\mu }\Gamma _{j}\left(
Q,R\right) \right\vert ^{\frac{2s}{s-1}}dv_{g}(R)\right) ^{\frac{s-2}{2s}%
}\left( \int_{M}\left\vert \nabla _{\mu }\varphi (R)\right\vert ^{\frac{2s}{%
s-1}}dv_{g}(R)\right) ^{\frac{s-2}{2s}}
\end{equation*}%
Taking account of (\ref{14'}) and using the polar coordinates, we get%
\begin{equation*}
\int_{M-B_{\epsilon }(P)}\left\vert \nabla \Gamma _{j}\left( Q,R\right)
\right\vert ^{\frac{2s}{s-1}}dv_{g}(R)\leq C_{j}\int_{M-B_{\epsilon
}(P)}d(P,Q)^{\frac{2s}{s-2}\left( 4(j+1)-n+\frac{n}{p}-1\right) }dv_{g}(R)
\end{equation*}%
\begin{equation}
\leq C_{j}\omega _{n-1}\int_{0}^{\delta (M)}r^{\frac{2s}{s-2}\left( 4(j+1)-n+%
\frac{n}{p}-1\right) +n-1}dr  \label{16}
\end{equation}%
where $\delta (M)$ is the injectivity radius, $\omega _{n-1}$ is the volume
of the unit-sphere in $R^{n}$. The integral (\ref{16}) is convergent if 
\begin{equation}
\frac{2s}{s-2}\left( 4(j+1)-n+\frac{n}{p}-1\right) +n>0\text{.}  \label{17}
\end{equation}%
We may suppose, $(j+1)<\frac{1}{4}\left( n-\frac{n}{p}+1\right) $, and
condition (\ref{17}) is equivalent to%
\begin{equation*}
j+1>\frac{n}{4}\left( \frac{1}{2}-\frac{1}{p}+\frac{1}{s}\right) +\frac{1}{4}%
=j_{o}+1\text{. }
\end{equation*}%
Hence, letting $\varepsilon $ goes to $0$, we get%
\begin{equation*}
\int_{M}\left\vert \nabla \Gamma _{j}\left( Q,R\right) \right\vert ^{\frac{2s%
}{s-1}}dv_{g}(R)<+\infty \text{.}
\end{equation*}%
Hence 
\begin{equation}
\Gamma _{j}\left( Q,R\right) \in H_{\frac{2s}{s-1}}^{1}\left( M\right) \text{%
.}  \label{18}
\end{equation}%
Since $s>\frac{n}{2}$ and $\frac{2s}{s-1}\leq \frac{2n}{n-2}$ we get by
Sobolev's inequality that%
\begin{equation}
\left( \int_{M}\left\vert \nabla \varphi (R)\right\vert ^{\frac{2s}{s-1}%
}dv_{g}(R)\right) ^{\frac{s-2}{2s}}\leq C\left\Vert \varphi \right\Vert
_{H_{2}}  \label{19}
\end{equation}%
where $C>0$, is a constant depending only on the dimension $n$ of $M$. Hence 
$H\left( Q,.\right) $ will be a Green function to $P_{g}$ if $u_{Q}$ solves
weakly the following equation%
\begin{equation}
\Delta _{R}^{2}u_{Q}(R)-\left( \nabla ^{\mu }\left( a\nabla _{\mu
}u_{Q}\right) \right) (R)+b(R)u_{Q}\left( R\right) =\Gamma _{m}(Q,R)b(R)
\label{20}
\end{equation}%
\begin{equation*}
-\left( \nabla ^{\mu }a(R)\nabla _{\mu }\Gamma _{o}\left( Q,R\right) \right)
-\sum_{j=j_{o}}^{m}\nabla ^{\mu }\left( a(R)\nabla _{\mu }\Gamma _{j}\left(
Q,R\right) \right) \text{.}
\end{equation*}%
Consider the functional $T_{Q}$ defined on $H_{2}\left( M\right) $ by 
\begin{equation*}
T_{Q}\left( \varphi \right) =\int_{M}\left( a(R)\Gamma _{m}(Q,R)\varphi
\left( R\right) -\sum_{j=0}^{m}a\left( R\right) \nabla _{R}\Gamma
_{j}(Q,R)\bigtriangledown \varphi \left( R\right) \right) dv_{g}\text{.}
\end{equation*}%
It is obvious from (\ref{18}) and (\ref{19}) that the functional $T_{Q}$ is
continuous on the Hilbert space $H_{2}^{2}\left( M\right) $, so since the
operator $\varphi \rightarrow \int_{M}\varphi P\left( \varphi \right) dv_{g}$
is coercive it follows by Lax-Milgram theorem that there is a unique $%
u_{Q}\in H_{2}$ such that for any $\varphi \in H_{2}$ 
\begin{equation*}
\int_{M}\varphi P\left( u_{Q}\right) dv_{g}=T_{Q}\left( \varphi \right)
\end{equation*}%
that is to say $u_{Q}$ is a weak solution of the equation (\ref{20}).

Next let $U\subset M$ be an open set.

\begin{lemma}
\label{lemm2} Let $h\in L_{loc}^{1}(U)$, $\xi \in C_{o}^{\infty }(U)$. If $%
u\in H_{2,loc}^{2}(U)$ is a weak solution of the equation $\Delta
^{2}u-\nabla ^{i}\left( a\nabla _{i}u\right) +bu=h$, then%
\begin{equation*}
(\xi u)(Q)=\int_{M}H(Q,R)P\left( u\right) (R)\xi \left( R\right)
dv_{g}\left( R\right)
\end{equation*}%
\begin{equation*}
+\int_{M}u(R)H(Q,R)\left\{ \Delta ^{2}\xi \left( R\right) -\nabla ^{\mu
}\left( a\nabla _{\mu }\xi \right) \left( R\right) \right\} dv_{g}\left(
R\right)
\end{equation*}%
\begin{equation*}
+2\left[ \int_{M}H(Q,R)\left( \Delta u.\Delta \xi \left( R\right)
-\left\langle \nabla \Delta \xi +\Delta \nabla \xi ,\nabla u\right\rangle
-\left\langle \nabla \Delta u+\Delta \nabla u,\nabla \xi \right\rangle
\right) dv_{g}(R)\right.
\end{equation*}%
\begin{equation*}
+2\int_{M}H(Q,R)\left( \left\langle \nabla ^{2}\xi \left( R\right) ,\nabla
^{2}u\right\rangle -a(R)\left\langle \nabla u,\nabla \xi \right\rangle
\right) dv_{g}\left( R\right)
\end{equation*}%
\begin{equation*}
+\frac{1}{V(M)}\int_{M}\left( \xi u\right) (Q)dv_{g}\left( Q\right) \text{.}
\end{equation*}
\end{lemma}

\begin{proof}
Let $\left( u_{n}\right) \subset C_{o}^{\infty }(U)$ such that $%
u_{n}\rightarrow u$ in $H_{2,loc}^{2}\left( U\right) $. We extend the
functions $\xi $, $u_{n}$ by $0$ outside $U$ to have a functions defined on
all $M$ and by applying the formula (\ref{11}), we get%
\begin{equation}
(\xi u_{n})(Q)=\int_{M}H(Q,R)P(\xi u_{n})(R)dv_{g}\left( R\right) +\frac{1}{%
V(M)}\int_{M}\left( \xi u_{n}\right) (Q)dv_{g}\left( Q\right) \text{.}
\label{22}
\end{equation}%
Now we compute 
\begin{equation*}
P(\xi u_{n})=\Delta ^{2}\left( \xi u_{n}\right) -\nabla ^{\mu }\left(
a\nabla _{\mu }\xi u_{n}\right) +b\xi u_{n}
\end{equation*}%
\begin{equation*}
=\xi \Delta ^{2}\left( u_{n}\right) +u_{n}\Delta ^{2}\left( \xi \right)
+2\Delta u_{n}\Delta \xi -2\left\langle \nabla \xi ,\Delta \nabla
u_{n}+\nabla \Delta u_{n}\right\rangle -2\left\langle \nabla u_{n},\Delta
\nabla \xi +\nabla \Delta \xi \right\rangle
\end{equation*}%
\begin{equation*}
+2\left\langle \nabla ^{2}u_{n},\nabla ^{2}\xi \right\rangle -\xi \nabla
^{\mu }\left( a\nabla _{\mu }u_{n}\right) -u_{n}\nabla ^{\mu }\left( a\nabla
_{\mu }\xi \right) -2a\left\langle \nabla u_{n},\nabla \xi \right\rangle
+b\xi u_{n}\text{.}
\end{equation*}%
Consequently%
\begin{equation*}
(\xi u_{n})(Q)=\int_{M}\xi \left( R\right) H(Q,R)P\left( u_{n}\right)
(R)dv_{g}\left( R\right)
\end{equation*}%
\begin{equation*}
+\int_{M}H(Q,R)u_{n}(R)\left\{ \Delta ^{2}\xi \left( R\right) -\nabla ^{\mu
}\left( a\nabla _{\mu }\xi \right) \left( R\right) \right\} dv_{g}\left(
R\right)
\end{equation*}%
\begin{equation*}
+2\int_{M}H(Q,R)\left( \Delta \xi \Delta u_{n}-\left\langle \nabla \Delta
\xi +\Delta \nabla \xi ,\nabla u_{n}\right\rangle \right) dv_{g}(R)
\end{equation*}%
\begin{equation*}
-\int_{M}2H(Q,R)\left\langle \nabla \Delta u_{n}+\Delta \nabla u_{n},\nabla
\xi \right\rangle dv_{g}\left( R\right)
\end{equation*}%
\begin{equation*}
-\int_{M}2H(Q,R)\left( a\left( R\right) \left\langle \nabla u_{n},\nabla \xi
\right\rangle +2\left\langle \nabla ^{2}u_{n},\nabla ^{2}\xi \right\rangle
\right) dv_{g}\left( R\right)
\end{equation*}%
\begin{equation*}
+\frac{1}{V(M)}\int_{M}\left( \xi u_{n}\right) (Q)dv_{g}\left( Q\right) 
\text{.}
\end{equation*}%
Let $B_{\epsilon }(Q)$ be the geodesic ball centred at $Q$ and of radius $%
\epsilon $; since $H(P,.)\in L^{1}(M)$, we have%
\begin{equation*}
2\int_{M}H(Q,R)\left( \Delta \xi .\Delta u_{n}-\left\langle \nabla \Delta
\xi +\Delta \nabla \xi ,\nabla u_{n}\right\rangle \right) dv_{g}\left(
R\right) =
\end{equation*}%
\begin{equation*}
2\int_{M-B_{\epsilon }(Q)}H(Q,R)\left( \Delta \xi .\Delta u_{n}-\left\langle
\nabla \Delta \xi +\Delta \nabla \xi ,\nabla u_{n}\right\rangle \right)
dv_{g}\left( R\right) +o(1)\text{ as }\epsilon \rightarrow 0^{+}
\end{equation*}%
\begin{equation*}
=2\int_{M-B_{\epsilon }(Q)}u_{n}\left( \Delta H(Q,R)\Delta \xi \left(
R\right) +H(Q,R)\Delta ^{2}\xi \left( R\right) -2\left\langle \nabla H\left(
Q,R\right) ,\nabla \left( \Delta \xi \right) \right\rangle \right)
dv_{g}\left( R\right)
\end{equation*}%
\begin{equation*}
-2\int_{\partial B_{\epsilon }(Q)}\left( u_{n}\partial _{\nu }\left(
H(Q,R)\Delta \xi \left( R\right) \right) -H(Q,R)\Delta \xi \left( R\right)
\partial _{\nu }u_{n}\right) d\sigma (R)
\end{equation*}%
\begin{equation*}
+2\int_{M-B_{\epsilon }(Q)}u_{n}\left( \left\langle \nabla H(Q,R),\nabla
\Delta \xi +\Delta \nabla \xi \right\rangle +H(Q,R)\left( \nabla ^{\mu
}\Delta \left( \nabla _{\mu }\xi \right) -\Delta ^{2}\xi \right) \right)
dv_{g}(R)\text{ as }\epsilon \rightarrow 0^{+}
\end{equation*}%
and using the estimates (\ref{14'}), we get that%
\begin{equation*}
2\int_{M}H(Q,R)\left( \Delta \xi .\Delta u_{n}-\left\langle \nabla \Delta
\xi +\Delta \nabla \xi ,\nabla u_{n}\right\rangle \right) dv_{g}\left(
R\right) =
\end{equation*}%
\begin{equation*}
=2\int_{M}u_{n}\left( \Delta H(Q,R)\Delta \xi \left( R\right) +H(Q,R)\Delta
^{2}\xi \left( R\right) -2\left\langle \nabla H\left( Q,R\right) ,\nabla
\left( \Delta \xi \right) \right\rangle \right) dv_{g}\left( R\right)
\end{equation*}%
\begin{equation*}
+2\int_{M}u_{n}\left( \left\langle \nabla H(Q,R),\nabla \Delta \xi +\Delta
\nabla \xi \right\rangle +H(Q,R)\left( \nabla ^{\mu }\Delta \left( \nabla
_{\mu }\xi \right) -\Delta ^{2}\xi \right) \right) dv_{g}(R)\text{ }
\end{equation*}%
Similarly, we obtain 
\begin{equation*}
-\int_{M}2H(Q,R)\left\langle \nabla \Delta u_{n}+\Delta \nabla u_{n},\nabla
\xi \right\rangle dv_{g}\left( Q\right) dv_{g}\left( Q\right) =
\end{equation*}%
\begin{equation*}
=2\int_{M}u_{n}\left( \Delta H\left( Q,R\right) \Delta \xi \left( R\right)
+2\left\langle \nabla ^{2}\xi \left( R\right) ,\nabla
^{2}H(Q,R)\right\rangle \right) dv_{g}
\end{equation*}%
\begin{equation*}
-4\int_{M}u_{n}(R)\left\langle \nabla H(Q,R),\Delta \left( \nabla \xi \left(
R\right) \right) \right\rangle dv_{g}
\end{equation*}%
\begin{equation*}
-\int_{M}u_{n}\left\langle \nabla \xi \left( R\right) ,\nabla \left( \Delta
H\left( Q,R\right) \right) +\Delta \left( \nabla H(R,Q\right) \right\rangle
dv_{g}\left( R\right) \text{.}
\end{equation*}%
Consequently%
\begin{equation*}
(\xi u_{n})(Q)=\int_{M}H(Q,R)P\left( u_{n}\right) (R)\xi \left( R\right)
dv_{g}\left( R\right)
\end{equation*}%
\begin{equation*}
+\int_{M}H(Q,R)u_{n}(R)\left\{ \Delta ^{2}\xi \left( R\right) -\nabla ^{\mu
}\left( a\nabla _{\mu }\xi \right) \left( R\right) \right\} dv_{g}\left(
R\right)
\end{equation*}%
\begin{equation*}
+2\int_{M}u_{n}\left( \Delta H(Q,R)\Delta \zeta \left( R\right)
+H(Q,R)\Delta ^{2}\xi \left( R\right) -2\left\langle \nabla H\left(
Q,R\right) ,\nabla \left( \Delta \xi \right) \right\rangle \left( Q\right)
\right) dv_{g}\left( R\right)
\end{equation*}%
\begin{equation*}
+2\int_{M}u_{n}\left( \left\langle \nabla H(Q,R),\nabla \left( \Delta \xi
\right) +\Delta \left( \nabla \xi \right) \right\rangle +H(Q,R)\left( \nabla
^{\mu }\Delta \left( \nabla _{\mu }\xi \right) -\Delta ^{2}\xi \right)
\right) \left( R\right)
\end{equation*}%
\begin{equation*}
-2\int_{M}u_{n}\left( \Delta H\left( Q,R\right) \Delta \xi \left( R\right)
+2\left\langle \nabla ^{2}\xi \left( R\right) ,\nabla
^{2}H(Q,R)\right\rangle +2\left\langle \nabla H(Q,R),\Delta \left( \nabla
\xi \left( R\right) \right) \right\rangle \right)
\end{equation*}%
\begin{equation}
+\int_{M}u_{n}\left\langle \nabla \xi \left( R\right) ,\nabla \left( \Delta
H\left( Q,R\right) \right) +\Delta \left( \nabla H(R,Q\right) \right\rangle
dv_{g}\left( R\right)  \label{23}
\end{equation}%
\begin{equation*}
+\frac{1}{V(M)}\int_{M}\left( \xi u\right) _{n}(Q)dv_{g}\left( Q\right) 
\text{.}
\end{equation*}%
Now by the same procedure as above, we obtain that 
\begin{equation*}
\int_{U}\left\vert \Delta \xi (R)\Delta H(Q,R)\right\vert dv_{g}(R)\leq
\end{equation*}%
\begin{equation*}
C\sup_{R\in U}\left( \left\Vert \xi (R)\right\Vert _{\infty }\right)
\sup_{Q\in U}\int_{B_{r}(Q)}d(Q,R)^{-n+2}dv_{g}\left( R\right) <+\infty
\end{equation*}%
where $\left\Vert \xi (R)\right\Vert _{\infty }=\max_{R\in U}\left(
\left\vert \Delta \xi (R)\right\vert \right) $.

Letting 
\begin{equation*}
F(Q)=\int_{M}\xi (R)H\left( Q,R\right) P\left( u\right) (R)dv_{g}(R)
\end{equation*}%
\begin{equation*}
=\int_{M}\xi (R)h\left( R\right) H(Q,R)dv_{g}(R)
\end{equation*}%
and%
\begin{equation*}
F_{n}(Q)=\int_{M}\xi (R)H\left( Q,R\right) P\left( u_{n}\right) (R)dv_{g}(R)
\end{equation*}%
\begin{equation*}
=\int_{M}\xi (R)H\left( Q,R\right) h_{n}(R)dv_{g}(R)
\end{equation*}%
with .%
\begin{equation*}
h=P(u)\text{, }h_{n}=P(u_{n})\text{.}
\end{equation*}%
Since $H(P,.)\in L^{1}(M)$, and $u_{n}\rightarrow u$ in $H_{2,loc}^{2}\left(
U\right) $, we obtain 
\begin{equation*}
\int_{M}\xi (R)H\left( Q,R\right) \Delta
^{2}u_{n}(R)dv_{g}(R)=\int_{M-B_{\epsilon }(Q)}\xi (R)H\left( Q,R\right)
\Delta ^{2}u_{n}(R)dv_{g}(R)+o(1)\text{ as }\epsilon \rightarrow 0^{+}
\end{equation*}%
\begin{equation*}
=\int_{M-B_{\epsilon }(Q)}\left( \xi (R)\Delta H\left( Q,R\right) +\Delta
\xi H(Q,R)-2\left\langle \nabla \xi ,\nabla H(Q,R\right\rangle \right)
\Delta u_{n}(R)dv_{g}(R)+o(1)\text{ as }\epsilon \rightarrow 0^{+}
\end{equation*}%
\begin{equation*}
=\int_{M-B_{\epsilon }(Q)}\left( \xi (R)\Delta H\left( Q,R\right) +\Delta
\xi H(Q,R)-2\left\langle \nabla \xi ,\nabla H(Q,R\right\rangle \right)
\Delta u(R)dv_{g}(R)+o(1)\text{ as }\epsilon \rightarrow 0^{+}
\end{equation*}%
Taking account of the estimates (\ref{14'}), we have 
\begin{equation*}
\int_{M}\xi (R)H\left( Q,R\right) \Delta ^{2}u_{n}(R)dv_{g}(R)\rightarrow
\int_{M}\xi (R)H\left( Q,R\right) \Delta ^{2}u(R)dv_{g}(R)\text{ as }%
n\rightarrow +\infty \text{.}
\end{equation*}%
We infer that%
\begin{equation*}
\int_{U}\left( F_{n}-F\right) (Q)dv_{g}(Q)=\int_{U}\int_{M}\xi (R)H\left(
Q,R\right) \left( h_{n}-h\right) \left( R\right) dv_{g}(R)
\end{equation*}%
which goes to $0$ as $n$ goes to $+\infty $.

By the same way as above if we put%
\begin{equation*}
G(Q)=\int_{M}H(Q,R)u(R)\left\{ \Delta ^{2}\xi \left( R\right) -\nabla ^{\mu
}\left( a\nabla _{\mu }\xi \right) \left( R\right) \right\} dv_{g}\left(
R\right)
\end{equation*}%
we obtain%
\begin{equation*}
\int_{M}\left\vert H(Q,R)\left\{ \Delta ^{2}\xi \left( R\right) -\nabla
^{\mu }\left( a\nabla _{\mu }\xi \right) \left( R\right) \right\}
\right\vert dv_{g}\left( R\right) \leq
\end{equation*}%
\begin{equation*}
\left\Vert \Delta ^{2}\xi \right\Vert _{\infty }\sup_{Q\in
M}\int_{M}d(Q,R)^{-n+4\beta }dv_{g}\left( R\right) +\left\Vert \nabla \xi
\right\Vert _{\infty }\left\Vert a\right\Vert _{p}\left\Vert \nabla
_{R}H\right\Vert _{\frac{p}{p-1}}<+\infty
\end{equation*}%
where $\left\Vert .\right\Vert _{\infty }$ denotes the supremum norm.

Then%
\begin{equation*}
\int_{U}\left\vert G\left( Q\right) \right\vert dv_{g}\left( Q\right) \leq
\end{equation*}%
\begin{equation*}
\sup_{Q\in U}\left\vert \int_{U}\left( H(Q,R)\Delta ^{2}\xi \left( R\right)
+a\left( R\right) \left\langle \nabla H(Q,R),\left( \nabla \xi \right)
\left( R\right) \right\rangle \right) dv_{g}\left( R\right) \right\vert
\int_{M}\left\vert u\left( R\right) \right\vert dv_{g}\left( R\right)
<+\infty \text{.}
\end{equation*}%
Letting 
\begin{equation*}
G_{n}\left( Q\right) =\int_{M}H(Q,R)u_{n}(R)\left\{ \Delta ^{2}\xi \left(
R\right) -\nabla ^{\mu }\left( a\nabla _{\mu }\xi \right) \left( R\right)
\right\} dv_{g}\left( R\right)
\end{equation*}%
then 
\begin{equation*}
\int_{U}\left\vert G_{n}\left( Q\right) \right\vert dv_{g}\left( Q\right)
<+\infty
\end{equation*}%
and 
\begin{equation*}
\int_{U}\left\vert G_{n}\left( Q\right) -G\left( Q\right) \right\vert
dv_{g}\left( Q\right) \leq
\end{equation*}%
\begin{equation*}
\sup_{Q\in U}\left\vert \int_{U}\left( H(Q,R)\Delta ^{2}\xi \left( R\right)
+a\left( R\right) \left\langle \nabla H(Q,R),\left( \nabla \xi \right)
\left( R\right) \right\rangle \right) dv_{g}\left( R\right) \right\vert
\left\Vert u_{n}-u\right\Vert _{L^{1}\left( M\right) }\text{.}
\end{equation*}%
Then $G_{n}\rightarrow G$ in $L^{1}\left( U\right) $ and almost every in $U$%
. The same is also true for the remaining terms of the right side hand of (%
\ref{23}) and we get the required formula.
\end{proof}

Now we recall the definition of Kato-Stummel's space,

\begin{definition}
\ Kato-Stummel space $K^{n}(U)$ is defined as the space of measurable
functions $f:U\rightarrow R$ such that for every $t>0$, 
\begin{equation*}
\varphi _{f}\left( t,U\right) =\sup_{Q\in M}\int_{B_{t}\left( Q\right) }%
\frac{\left\vert f(S)\right\vert \chi _{U}\left( S\right) }{d\left(
Q,S\right) ^{l}}dv_{g}\left( S\right) <+\infty
\end{equation*}%
with $\ l=(j+1)\left( n-4\right) -jn\left( 1-\frac{1}{p}\right) $, $j\geq 1$
any integer

and 
\begin{equation*}
\lim_{t\rightarrow 0^{+}}\varphi _{f}\left( t,U\right) =0\text{.}
\end{equation*}%
where $U$ denotes an open set of $M$ and $\chi _{U}$ is the characteristic
function of $U$.
\end{definition}

We consider the $C^{\infty }$-function $\eta :R\rightarrow \left[ 0,1\right] 
$ with compact support given by 
\begin{equation*}
\eta (t)=\left\{ 
\begin{array}{c}
1\text{ for }\left\vert t\right\vert \leq \frac{1}{2} \\ 
0\text{ \ \ for }\left\vert t\right\vert \geq 1%
\end{array}%
\right.
\end{equation*}%
and for any $P$, $Q\in M$, we put%
\begin{equation*}
\eta _{\delta }\left( P,Q\right) =\eta \left( \delta ^{-1}d\left( P,Q\right)
\right)
\end{equation*}%
and for $R\in M$, $S\in U$, with $R\neq S$, we put 
\begin{equation*}
F_{\delta }\left( R,S\right) =\int_{M}\frac{\left\vert f\left( T\right)
\right\vert \eta _{\delta }\left( S,T\right) \eta _{\delta }\left(
R,T\right) }{d\left( T,S\right) ^{l}d\left( T,R\right) ^{l}}dv_{g}\left(
T\right)
\end{equation*}%
where $l=(j+1)\left( n-4\right) -jn\left( 1-\frac{1}{p}\right) $, $j\geq 1$,
any integer $0<\delta <d(S,\partial U)/4$ and $d\left( S,\partial U\right) $
is the distance from $S$ to the boundary $\partial U$ of $U$.

First we quote the following lemma \cite{18} which will be used.

\begin{lemma}
\label{lem1} There exists a constant $c\left( n\right) >0$ such that

$F_{\delta }\left( R,S\right) \leq c\left( n\right) \varphi _{f}\left(
\delta ,B_{3\delta }\left( Q\right) \right) \eta _{4\delta }\left(
R,S\right) d(R,S)^{-l}$.
\end{lemma}

Now, we give a representation formula to the solutions of equation (\ref{222}%
)

\begin{lemma}
Let $Q_{o}\in U$ , $r_{o}>0$ such that $B_{2r_{o}}\subset U$ and $h$, $u$ as
in Lemma \ref{lemm2}. For almost every $Q\in B_{r_{o}}\left( Q_{o}\right) $%
\begin{equation*}
u(Q)=\int_{M}H(Q,R)h(R)\eta _{r_{o}}\left( Q,R\right) dv_{g}\left( R\right)
\end{equation*}%
\begin{equation*}
+\int_{M}u(R)H(Q,R)\left( \nabla ^{\mu }\left( a\nabla _{\mu }\eta
_{r_{o}}\left( Q,R\right) \right) -\Delta ^{2}\eta _{r_{o}}(Q,R)\right)
dv_{g}(R)
\end{equation*}%
\begin{equation*}
+\int_{M}u\left( R\right) \left( \Delta H(Q,R)\Delta \eta _{r_{o}}\left(
Q,R\right) \right) dv_{g}\left( R\right) +2\int_{M}u(R)a(R)\left\langle
\nabla H(Q,R),\nabla \eta _{r_{o}}(Q,R)\right\rangle dv_{g}(R)
\end{equation*}%
\begin{equation}
-2\int_{M}u(R)\left\langle \nabla H\left( Q,R\right) ,\nabla \left( \Delta
\eta _{r_{o}}\left( Q,R\right) \right) \right\rangle \left( Q\right)
dv_{g}(R)+8\int_{M}u(R)\left\langle \nabla H(Q,R),\Delta \left( \nabla \eta
_{r_{o}}\left( Q,R\right) \right) \right\rangle dv_{g}(R)  \label{24}
\end{equation}%
\begin{equation*}
+2\int_{M}u(R)a(R)\left\langle \nabla H(Q,R),\nabla \eta
_{r_{o}}(Q,R)\right\rangle dv_{g}(R)+8\int_{M}u(R)\left\langle \nabla
H(Q,R),\Delta \left( \nabla \eta _{r_{o}}\left( Q,R\right) \right)
\right\rangle dv_{g}(R)
\end{equation*}%
\begin{equation*}
-2\int_{M}u(R)\left\langle \nabla ^{2}H(Q,R),\nabla ^{2}\eta _{r_{o}}\left(
Q,R\right) \right\rangle dv_{g}(R)+2\int_{M}u(R)\left( \left\langle \Delta
\nabla H\left( Q,R\right) ,\nabla \eta _{r_{o}}\left( Q,R\right)
\right\rangle \right) dv_{g}(R)
\end{equation*}%
\begin{equation*}
+\frac{1}{V(M)}\int_{M}\eta _{r_{o}}\left( Q,R\right) u(R)dv_{g}\left(
R\right) \text{.}
\end{equation*}%
.
\end{lemma}

\begin{proof}
\ Let $S\in B_{r_{o}}(Q_{o})$ and consider the function%
\begin{equation*}
\xi (Q)=\eta _{r_{o}}\left( S,Q\right)
\end{equation*}%
then $\xi \in C_{o}^{\infty }(U)$ and by Lemma \ref{lemm2}, we write%
\begin{equation*}
\eta _{r_{o}}\left( S,Q\right) u(Q)=\int_{M}H(Q,R)P\left( u\right) (R)\eta
_{r_{o}}\left( R,S\right) dv_{g}\left( R\right)
\end{equation*}%
\begin{equation*}
+\int_{M}u(R)H(Q,R)\left\{ \Delta ^{2}\eta _{r_{o}}\left( S,R\right) -\nabla
^{\mu }\left( a\nabla _{\mu }\eta _{r_{o}}\left( S,R\right) \right) \right\}
dv_{g}\left( R\right)
\end{equation*}%
\begin{equation*}
+2\int_{M}H(Q,R)\left( \Delta u(R).\Delta \eta _{r_{o}}\left( S,R\right)
-\left\langle \nabla \Delta \eta _{r_{o}}\left( R,S\right) +\Delta \nabla
\eta _{r_{o}}\left( R,S\right) ,\nabla u\left( R\right) \right\rangle
\right) dv_{g}\left( R\right)
\end{equation*}%
\begin{equation*}
-2\int_{M}H(Q,R)\left\langle \nabla \Delta u\left( R\right) +\Delta \nabla
u\left( R\right) ,\nabla \eta _{r_{o}}\left( R,S\right) \right\rangle
dv_{g}\left( R\right)
\end{equation*}%
\begin{equation*}
+2\int_{M}H(Q,R)\left( \left\langle \nabla ^{2}\eta _{r_{o}}\left(
R,S\right) ,\nabla ^{2}u\left( R\right) \right\rangle -a(R)\left\langle
\nabla u\left( R\right) ,\nabla \eta _{r_{o}}\left( R,S\right) \right\rangle
\right) dv_{g}\left( R\right)
\end{equation*}%
\begin{equation*}
+\frac{1}{V(M)}\int_{M}\left( \xi u\right) (R)dv_{g}\left( R\right) \text{.}
\end{equation*}%
On the other hand, since by (\ref{18}) $H(Q,.)\in L^{2}(M)$, we have 
\begin{equation*}
\int_{M}H(Q,R)\Delta u.\Delta \eta _{r_{o}}\left( S,R\right)
dv_{g}(R)=\int_{M-B_{\epsilon }(Q)}H(Q,R)\Delta u.\Delta \eta _{r_{o}}\left(
S,R\right) dv_{g}(R)+o(1)\text{ as }\epsilon \rightarrow 0
\end{equation*}%
\begin{equation*}
=\int_{M-B_{\epsilon }(Q)}u\left( R\right) \left( \Delta H(Q,R)\Delta \eta
_{r_{o}}\left( S,R\right) +H(Q,R)\Delta ^{2}\eta _{r_{o}}\left( S,R\right)
\right) dv_{g}\left( R\right)
\end{equation*}%
\begin{equation*}
-2\int_{M-B_{\epsilon }(Q)}u(R)\left\langle \nabla H\left( Q,R\right)
,\nabla \left( \Delta \eta _{r_{o}}\left( S,R\right) \right) \right\rangle
\left( Q\right) dv_{g}(R)+o(1)\text{ as }\epsilon \rightarrow 0\text{.}
\end{equation*}%
By H\"{o}lder's inequality and the estimations (\ref{14'}), we have%
\begin{equation*}
\int_{M-B_{\epsilon }(Q)}\left\vert u\left( R\right) \Delta H(Q,R)\Delta
\eta _{r_{o}}\left( S,R\right) \right\vert dv_{g}(R)\leq C\sup \left\vert
\Delta \eta _{r_{o}}\left( S,R\right) \right\vert \left( \int_{M}\left\vert
u(R)\right\vert ^{2}dv_{g}\right) ^{\frac{1}{2}}
\end{equation*}%
\begin{equation*}
\int_{M-B_{\epsilon }\left( Q\right) }\left(
d(Q,R)^{4-n}+\sum_{j=j_{o}}^{m}C_{j}\left\Vert b\right\Vert
_{p}^{j}d(Q,R)^{4(j+1)-n+\frac{n}{p}}\right) dv_{g}(R)
\end{equation*}%
\begin{equation*}
\leq C\sup \left\vert \Delta \eta _{r_{o}}\left( S,R\right) \right\vert
\left( \int_{M}\left\vert u(R)\right\vert ^{2}dv_{g}\right) ^{\frac{1}{2}%
}\omega _{n-1}\int_{\epsilon }^{\delta }\left(
r^{3}+\sum_{j=j_{o}}^{m}C_{j}\left\Vert b\right\Vert _{p}^{j}r^{4(j+1)-1+%
\frac{n}{p}}dr\right) <+\infty \text{ }
\end{equation*}%
independently of $\epsilon $, where $\omega _{n-1}$ denotes the volume of
the unit sphere of $R^{n}$. Hence 
\begin{equation*}
\int_{M}\left\vert u\left( R\right) \Delta H(Q,R)\Delta \eta _{r_{o}}\left(
S,R\right) \right\vert dv_{g}(R)<+\infty \text{.}
\end{equation*}%
And also 
\begin{equation*}
\int_{M}\left\vert u(R)\left\langle \nabla H\left( Q,R\right) ,\nabla \left(
\Delta \eta _{r_{o}}\left( S,R\right) \right) \right\rangle \left( Q\right)
\right\vert dv_{g}(R)<+\infty .
\end{equation*}%
Consequently%
\begin{equation*}
\int_{M}H(Q,R)\Delta u.\Delta \eta _{r_{o}}\left( S,R\right) dv_{g}(R)=
\end{equation*}%
\begin{equation*}
=\int_{M}u\left( R\right) \left( \Delta H(Q,R)\Delta \eta _{r_{o}}\left(
S,R\right) +H(Q,R)\Delta ^{2}\eta _{r_{o}}\left( S,R\right) -2\left\langle
\nabla H(Q,R),\nabla \Delta \eta _{r_{o}}(S,R)\right\rangle \right)
dv_{g}\left( R\right)
\end{equation*}%
Continuing to use the estimations (\ref{14'}), we get%
\begin{equation*}
\eta _{r_{o}}\left( S,Q\right) u(Q)=\int_{M}H(Q,R)P\left( u\right) (R)\eta
_{r_{o}}\left( S,R\right) dv_{g}\left( R\right)
\end{equation*}%
\begin{equation*}
+\int_{M}H(Q,R)u(R)\left\{ \Delta ^{2}\eta _{r_{o}}\left( S,R\right) -\nabla
^{\mu }\left( a\nabla _{\mu }\eta _{r_{o}}\left( S,R\right) \right) \right\}
dv_{g}\left( R\right)
\end{equation*}%
\begin{equation*}
+2\int_{M}u\left( R\right) \left( \Delta H(Q,R)\Delta \eta _{r_{o}}\left(
S,R\right) +H(Q,R)\Delta ^{2}\eta _{r_{o}}\left( S,R\right) \right)
dv_{g}\left( R\right)
\end{equation*}%
\begin{equation*}
-4\int_{M}u(R)\left\langle \nabla H\left( Q,R\right) ,\nabla \left( \Delta
\eta _{r_{o}}\left( S,R\right) \right) \right\rangle \left( Q\right)
dv_{g}(R)
\end{equation*}%
\begin{equation*}
+2\int_{M}u\left( R\right) \left( \left\langle \nabla H(Q,R),\nabla \Delta
\eta _{r_{o}}S,R+\Delta \nabla \eta _{r_{o}}\left( S,R\right) \right\rangle
\right) dv_{g}\left( R\right)
\end{equation*}%
\begin{equation*}
+2\int_{M}u\left( R\right) H(Q,R)\left( \nabla ^{\mu }\Delta \left( \nabla
_{\mu }\eta _{r_{o}}\left( S,R\right) \right) -\Delta ^{2}\eta
_{r_{o}}\left( S,R\right) \right) dv_{g}\left( R\right)
\end{equation*}%
\begin{equation*}
+4\int_{M}u(R)\left\langle \nabla H(Q,R),\Delta \left( \nabla \eta
_{r_{o}}\left( S,R\right) \right) \right\rangle dv_{g}(R)
\end{equation*}%
\begin{equation*}
-4\int_{M}u(R)\left\langle \nabla ^{2}H(Q,R),\nabla ^{2}\eta _{r_{o}}\left(
S,R\right) \right\rangle dv_{g}(R)
\end{equation*}%
\begin{equation*}
+2\int_{M}u(R)\left( \left\langle \Delta \nabla H\left( Q,R\right) ,\nabla
\eta _{r_{o}}\left( S,R\right) \right\rangle +2\left\langle \Delta \nabla
\eta _{r_{o}}\left( S,R\right) ,\nabla H(Q,R)\right\rangle \right) dv_{g}(R)
\end{equation*}%
\begin{equation*}
-2\int_{M}u(R)\left( \Delta H(Q,R).\Delta \eta _{r_{o}}\left( S,R\right)
+H(Q,R\right) \Delta ^{2}\eta _{r_{o}}\left( S,R\right) dv_{g}(R)
\end{equation*}%
\begin{equation*}
+2\int_{M}u\left( R\right) \left\langle \nabla ^{2}H\left( Q,R\right)
,\nabla ^{2}\eta _{r_{o}}\left( S,R\right) \right\rangle dv_{g}\left(
R\right) +2\int_{M}u\left\langle \nabla H(R,Q),\Delta \nabla \eta
_{r_{o}}\right\rangle dv_{g}\left( R\right)
\end{equation*}%
\begin{equation*}
-2\int_{M}u(R)a(R)\nabla \Delta \nabla \eta
_{r_{o}}(S,R)dv_{g}(R)+2\int_{M}u(R)a(R)\left\langle \nabla H(Q,R),\nabla
\eta _{r_{o}}(S,R)\right\rangle dv_{g}(R)
\end{equation*}%
\begin{equation*}
2\int_{M}u(R)H(Q,R)\nabla ^{\mu }\left( a\nabla _{\mu }\eta _{r_{o}}\left(
S,R\right) \right) dv_{g}(R)+\frac{1}{V(M)}\int_{M}\eta _{r_{o}}\left(
S,R\right) u(R)dv_{g}\left( R\right) \text{.}
\end{equation*}%
And by canceling similar terms, we get%
\begin{equation*}
\eta _{r_{o}}\left( S,Q\right) u(Q)=\int_{M}H(Q,R)h(R)\eta _{r_{o}}\left(
S,R\right) dv_{g}\left( R\right)
\end{equation*}%
\begin{equation*}
+\int_{M}u(R)H(Q,R)\left( \nabla ^{\mu }\left( a\nabla _{\mu }\eta
_{r_{o}}\left( S,R\right) \right) -\Delta ^{2}\eta _{r_{o}}(S,R)\right)
dv_{g}(R)
\end{equation*}%
\begin{equation*}
+\int_{M}u\left( R\right) \left( \Delta H(Q,R)\Delta \eta _{r_{o}}\left(
S,R\right) \right) dv_{g}\left( R\right) +2\int_{M}u(R)a(R)\left\langle
\nabla H(Q,R),\nabla \eta _{r_{o}}(S,R)\right\rangle dv_{g}(R)
\end{equation*}%
\begin{equation*}
-2\int_{M}u(R)\left\langle \nabla H\left( Q,R\right) ,\nabla \left( \Delta
\eta _{r_{o}}\left( S,R\right) \right) \right\rangle \left( Q\right)
dv_{g}(R)+8\int_{M}u(R)\left\langle \nabla H(Q,R),\Delta \left( \nabla \eta
_{r_{o}}\left( S,R\right) \right) \right\rangle dv_{g}(R)
\end{equation*}%
\begin{equation*}
+2\int_{M}u(R)a(R)\left\langle \nabla H(Q,R),\nabla \eta
_{r_{o}}(S,R)\right\rangle dv_{g}(R)+8\int_{M}u(R)\left\langle \nabla
H(Q,R),\Delta \left( \nabla \eta _{r_{o}}\left( S,R\right) \right)
\right\rangle dv_{g}(R)
\end{equation*}%
\begin{equation*}
-2\int_{M}u(R)\left\langle \nabla ^{2}H(Q,R),\nabla ^{2}\eta _{r_{o}}\left(
S,R\right) \right\rangle dv_{g}(R)+2\int_{M}u(R)\left( \left\langle \Delta
\nabla H\left( Q,R\right) ,\nabla \eta _{r_{o}}\left( S,R\right)
\right\rangle \right) dv_{g}(R)
\end{equation*}%
\begin{equation*}
+\frac{1}{V(M)}\int_{M}\eta _{r_{o}}\left( S,R\right) u(R)dv_{g}\left(
R\right) \text{.}
\end{equation*}%
So by letting $S=Q$, we get the desired formula.
\end{proof}

\begin{theorem}
Let $f\in K^{n}\left( U\right) $ and $u\in H_{2,loc}^{2}\left( U\right) $ a
weak solution of the equation 
\begin{equation}
P(u)+fu=0\text{.}  \label{25}
\end{equation}

If $fu\in L_{loc}^{1}\left( U\right) $, then $u$ is locally bounded on $U$.
\end{theorem}

\begin{proof}
Let $I=\left[ 0,1\right] $ and denote by $\chi _{I}$ the characteristic
function of $I$. Let also $0<\delta \leq \delta _{o}\leq \frac{r_{o}}{4}$
where $r_{o}$ is chosen so that $B\left( Q_{o},2r_{o}\right) \subset U$ with 
$Q_{o}\in U$.

First, we have%
\begin{equation*}
\left\vert \nabla _{T}^{i}\eta _{\delta }\left( Q,T\right) \right\vert \leq
c_{i}\left( n\right) \delta ^{-i}\chi _{I}\left( \delta ^{-1}d\left(
T,Q\right) \right) \text{, \ }i=1,...\text{.}4\text{.}
\end{equation*}%
On the other hand if we denote, respectively, by $J_{i}$, $i=1,...,11$, the
terms of the second right hand of the equality (\ref{24}), we obtain%
\begin{equation}
\left\vert J_{1}\right\vert \leq \overline{c}_{1}\left( n\right) \int_{M}%
\frac{\left\vert \left( fu\right) (R)\right\vert \eta _{\delta }\left(
Q,R\right) }{d(Q,R)^{l}}dv_{g}\left( R\right)  \label{26}
\end{equation}%
\begin{equation*}
\leq \overline{c}_{1}\left( n\right) \overset{\_}{J_{1}}
\end{equation*}%
with 
\begin{equation*}
\overset{\_}{J_{1}}=\int_{M}\frac{\left\vert \left( fu\right) (R)\right\vert
\eta _{\delta }\left( Q,R\right) }{d(Q,R)^{l}}dv_{g}\left( R\right)
\end{equation*}%
where $l=(j+1)\left( n-4\right) -jn\left( 1-\frac{1}{p}\right) $, $j\geq 1$,
any integer and $\overline{c}_{1}\left( l\right) $ is a constant depending
on $l$.

Letting $S\in B_{r_{o}}\left( Q_{o}\right) $, multiplying $\overset{\_}{J_{1}%
}$ by $\frac{\left\vert f\left( Q\right) \right\vert \eta _{\delta }\left(
S,Q\right) }{d\left( S,Q\right) ^{l}}$ and integrating over $M$, we get by
the Fubini's formula%
\begin{equation*}
\overline{J}_{1}=\int_{M}\frac{\left\vert f\left( Q\right) \right\vert \eta
_{\delta }\left( S,Q\right) }{d\left( S,Q\right) ^{l}}\left( \int_{M}\frac{%
\left\vert \left( fu\right) (R)\right\vert \eta _{\delta }\left( Q,R\right) 
}{d(Q,R)^{l}}dv_{g}\left( R\right) \right) dv_{g}\left( Q\right)
\end{equation*}%
\begin{equation*}
=\int_{M}\left\vert \left( fu\right) (R)\right\vert \left( \int_{M}\frac{%
\left\vert f\left( Q\right) \right\vert \eta _{\delta }\left( S,Q\right)
\eta _{\delta }\left( Q,R\right) }{d\left( S,Q\right) ^{l}d(Q,R)^{l}}%
dv_{g}\left( Q\right) \right) dv_{g}\left( R\right)
\end{equation*}%
and taking account of Lemma \ref{lem1}, we get%
\begin{equation*}
\left\vert \overline{J}_{1}\right\vert \leq c_{1}(n)\varphi _{f}\left(
\delta ,B_{3\delta }\left( S\right) \right) \int_{M}\left\vert \left(
fu\right) (R)\right\vert d\left( S,R\right) ^{-l}\eta _{4\delta }\left(
S,R\right) dv_{g}\left( R\right)
\end{equation*}%
\begin{equation*}
\leq c_{1}(n)\varphi _{f}\left( \delta ,B_{3\delta }\left( S\right) \right)
\int_{M}\left\vert \left( fu\right) (R)\right\vert d\left( S,R\right)
^{-l}\left( \eta _{4\delta }\left( S,R\right) -\eta _{\delta }\left(
S,R\right) \right) dv_{g}\left( R\right)
\end{equation*}%
\begin{equation*}
+c_{1}(n)\varphi _{f}\left( \delta ,B_{3\delta }\left( S\right) \right)
\int_{M}\left\vert \left( fu\right) (R)\right\vert d\left( S,R\right)
^{-l}\eta _{\delta }\left( S,R\right) dv_{R}\left( R\right)
\end{equation*}%
taking account of%
\begin{equation*}
\eta _{4\delta }\left( S,R\right) -\eta _{\delta }\left( S,R\right) =0\text{
, for }d\left( S,R\right) \leq \frac{\delta }{2}\text{ or }d\left(
S,R\right) \geq 4\delta
\end{equation*}%
we obtain%
\begin{equation*}
\left\vert \overset{\_}{J}_{1}\right\vert \leq 2c_{1}(n)\varphi _{f}\left(
\delta ,B_{3\delta }\left( S\right) \right) \left( \frac{\delta }{2}\right)
^{-l}\left\Vert fu\right\Vert _{L^{1}\left( B_{4\delta +r_{o}}\left(
Q_{o}\right) \right) }+c_{1}(n)\varphi _{f}\left( \delta ,B_{3\delta }\left(
S\right) \right) \overline{J}_{1}
\end{equation*}%
and since $\varphi _{f}\left( \delta ,B_{3\delta }\left( S\right) \right)
\rightarrow 0^{+}$ as $\delta \rightarrow 0^{+}$, we choose $\delta >0$ such
that 
\begin{equation*}
1-c_{1}(n)\varphi _{f}\left( \delta ,B_{3\delta }\left( S\right) \right) >0%
\text{.}
\end{equation*}

Hence%
\begin{equation*}
\left\vert \overline{J}_{1}\right\vert <+\infty
\end{equation*}%
and 
\begin{equation*}
\left\vert J_{1}\right\vert <+\infty
\end{equation*}%
By Lemma \ref{lem6}, we get%
\begin{equation*}
\left\vert J_{2}\right\vert \leq C_{1}\int_{M}\frac{\left\vert u(R)\left\{
\Delta ^{2}\eta _{\delta }\left( Q,R\right) \right\} \right\vert }{d(Q,R)^{l}%
}dv_{g}\left( R\right)
\end{equation*}%
\begin{equation*}
+C_{2}\int_{M}\frac{\left\vert \left( ua\right) (R)\right\vert \left\vert
\nabla _{R}\eta _{\delta }\left( Q,R\right) \right\vert }{d(Q,R)^{l-1}}%
dv_{g}\left( R\right) +C_{o}\int_{M}\frac{\left\vert u(R)\right\vert
\left\vert \nabla a(R)\right\vert \left\vert \nabla \eta _{r_{o}}\left(
S,R\right) \right\vert }{d(Q,R)^{l-2}}dv_{g}\left( R\right)
\end{equation*}%
and by Lemma \ref{lem1}, we obtain%
\begin{equation*}
\left\vert J_{2}\right\vert \leq c_{2}\left( n\right) \delta ^{-4}\left( 
\frac{\delta }{2}\right) ^{-l}\left\Vert u\right\Vert _{L^{1}\left(
B_{2\delta +r_{o}}\right) }+c_{1}\left( n\right) \delta ^{-1}\left( \frac{%
\delta }{2}\right) ^{-l-1}\left\Vert a\right\Vert _{p}\left\Vert
u\right\Vert _{L^{\frac{p}{p-1}}\left( B_{2\delta +r_{o}}\right) }
\end{equation*}%
\begin{equation*}
+c_{o}(n)\delta ^{-1}\left( \frac{\delta }{2}\right) ^{-l}\left\Vert
a\right\Vert _{p}\left\Vert \nabla u\right\Vert _{L^{\frac{p}{p-1}}\left(
B_{2\delta +r_{o}}\right) }\text{.}
\end{equation*}%
Also%
\begin{equation*}
\left\vert J_{3}\right\vert \leq \left( 2^{3}c_{2}\left( n\right) \left( 
\frac{\delta }{2}\right) ^{-l}+2^{5}c_{3}\left( n\right) \left( \frac{\delta 
}{2}\right) ^{-l-2}+2^{5}c_{4}\left( n\right) \left( \frac{\delta }{2}%
\right) ^{-l-4}\right) \left\Vert u\right\Vert _{L^{1}\left( B_{2\delta
+r_{o}}\right) }\text{.}
\end{equation*}%
By the same procedure as above and applying repeatedly Lemma \ref{lem1}, we
get that all the remaining terms of the formula (\ref{24}) are bounded and
the solution $u$ of the equation (\ref{25}) is locally bounded.
\end{proof}

\section{$Q$-curvature type equation}

Let $\left( M,g\right) $ be a compact $n$-dimensional Riemannian manifold, $%
n\geq 5$, we consider the following fourth order equation%
\begin{equation}
\Delta ^{2}u-\nabla ^{i}\left( a(x)\nabla _{i}u\right) +b(x)u=f\left\vert
u\right\vert ^{N-2}u  \label{27}
\end{equation}%
where $a\in L^{s}(M)$, $b\in L^{p}(M)$, with $s>\frac{n}{2}$, $p>\frac{n}{4}$%
, $f$ $\in C^{\infty }(M)$ a positive function and $N=\frac{2n}{n-4}$. To
solve the equation (\ref{27}), we use the variational method.

For any $u\in H_{2}(M)$, we let 
\begin{equation*}
J(u)=\int_{M}\left( \Delta u\right) ^{2}dv_{g}+\int_{M}a(x)\left\vert \nabla
u\right\vert ^{2}dv_{g}+\int_{M}b(x)u^{2}dv_{g}
\end{equation*}%
be the energy functional and consider the Sobolev quotient, for any $u\in
H_{2}(M)-\left\{ 0\right\} $%
\begin{equation*}
Q(u)=\frac{J(u)}{\left( \int_{M}f\left\vert u\right\vert ^{N}dv_{g}\right) ^{%
\frac{2}{N}}}\text{.}
\end{equation*}%
\begin{equation*}
A=\left\{ u\in H_{2}(M):\int_{M}f\left\vert u\right\vert ^{N}dv_{g}=\left(
1+\left\Vert a\right\Vert _{s}+\left\Vert b\right\Vert _{p}\right) ^{\frac{N%
}{2}}\right\}
\end{equation*}%
obviously $A\neq \phi $.

Put 
\begin{equation*}
Q(M)=\inf_{u\in H_{2}(M)-\left\{ 0\right\} }Q(u)=\inf_{u\in A}J(u)\text{.}
\end{equation*}

\begin{theorem}
\label{th1} Suppose that $Q(M)<\left( \sup_{x\in M}f(x)\right) ^{-\frac{2}{N}%
}K(n,2)^{-2}\left( 1+\left\Vert a\right\Vert _{s}+\left\Vert b\right\Vert
_{p}\right) $. The equation (\ref{27}) has a non trivial weak solution $u\in
H_{2}$ satisfying $J(u)=Q(M)$ and $u\in A$.
\end{theorem}

Before starting the proof of Theorem \ref{1}, we state the following lemma
which controls the $L^{q}$-norm of the gradient by the $L^{2}$-norm of the
laplacian.

\begin{lemma}
\label{lem7} Let $(M,g)$ be a compact Riemannian $n$-dimensional manifold ($%
n\geq 3$). Then for any $\epsilon >0$, there exists $C\left( \varepsilon
\right) >0$ such that 
\begin{equation}
\left\Vert \nabla u\right\Vert _{q}\leq \epsilon \left\Vert \Delta
u\right\Vert _{2}+C\left( \epsilon ,q\right) \left\Vert u\right\Vert _{q}
\label{33}
\end{equation}%
where $2\leq q<\frac{2n}{n-2}$.
\end{lemma}

\begin{proof}
The proof of this lemma is similar to that of lemma 2.2 page 16 in \cite{20}
which corresponds for the particular case $q=2$. For convenience we give the
proof. We proceed by contradiction. Let $\epsilon _{o}>0$ and let $\left(
u_{i}\right) _{i\geq 1}$ be a sequence in $H_{2}^{2}(M)$ such that 
\begin{equation}
\left\Vert \nabla u_{i}\right\Vert _{q}>\epsilon _{o}\left\Vert \Delta
u_{i}\right\Vert _{2}+i\left\Vert u_{i}\right\Vert _{q}\text{ and }%
\left\Vert \nabla u_{i}\right\Vert _{q}=1\text{.}  \label{28'}
\end{equation}%
Then%
\begin{equation*}
\left\Vert \Delta u_{i}\right\Vert _{2}+\left\Vert \nabla u_{i}\right\Vert
_{2}+\left\Vert u_{i}\right\Vert _{2}\leq \epsilon _{o}^{-1}+1+i^{-1}
\end{equation*}%
for any $i\geq 1$, so $\left( u_{i}\right) _{i}$ is bounded in $%
H_{2}^{2}\left( M\right) $. Thanks to the compactness of the embedding $%
H_{2}^{2}\left( M\right) \subset H_{1}^{q}\left( M\right) $ ($q<\frac{2n}{n-2%
}$), up to a subsequence $\left( u_{i}\right) $ converges strongly to $u$ in 
$H_{1}^{q}(M)$. By inequality (\ref{28'}) we infer that $\left\Vert \nabla
u\right\Vert _{q}=1$ and $\left\Vert u\right\Vert =0$: a contradiction.
\end{proof}

Now we are in position to prove Theorem \ref{th1}.

\begin{proof}
First we show that $Q(M)$ is finite. For any $u\in A$,

\begin{equation}
J(u)\geq \left\Vert \Delta u\right\Vert _{2}^{2}-\left\Vert a\right\Vert
_{s}\left\Vert \nabla u\right\Vert _{2s^{\prime }}^{s^{\prime }}-\left\Vert
b\right\Vert _{p}\left\Vert u\right\Vert _{\frac{2p}{p-1}}^{2}  \label{28}
\end{equation}%
with 
\begin{equation*}
s^{\prime }=\frac{s}{s-1}\text{.}
\end{equation*}%
Since $s>\frac{n}{2}$ then \ $2<2s^{\prime }<\frac{2n}{n-2}$ and by Lemma %
\ref{lem7} we get that for every $\eta >0$ there is a constant $C=C(\eta ,s)$
such that 
\begin{equation*}
\left\Vert \nabla u\right\Vert _{2s^{\prime }}^{2}\leq \eta \left\Vert
\Delta u\right\Vert _{2}^{2}+C\left\Vert u\right\Vert _{2s^{\prime }}^{2}
\end{equation*}%
and by H\"{o}lder's inequality we obtain%
\begin{equation*}
\left\Vert u\right\Vert _{2s^{\prime }}^{2}\leq \left\Vert u\right\Vert
_{N}^{2}V(M)^{1-\frac{2s^{\prime }}{N}}\leq \max \left( 1,V(M)\right)
\left\Vert u\right\Vert _{N}^{2}
\end{equation*}%
we obtain%
\begin{equation}
\left\Vert \nabla u\right\Vert _{2s^{\prime }}^{2}\leq \eta \left\Vert
\Delta u\right\Vert _{2}^{2}+C\max \left( 1,V(M)\right) \left\Vert
u\right\Vert _{N}^{2}  \label{29}
\end{equation}%
where $N=\frac{2n}{n-4}$and also%
\begin{equation*}
\left\Vert u\right\Vert _{\frac{2p}{p-1}}^{2}\leq \max \left( 1,V(M)\right)
\left\Vert u\right\Vert _{N}^{2}\text{.}
\end{equation*}%
Hence%
\begin{equation}
J(u)\geq \left( 1-\eta \left\Vert a\right\Vert _{s}\right) \left\Vert \Delta
u\right\Vert _{2}^{2}-\left( C\left\Vert a\right\Vert _{s}+\left\Vert
b\right\Vert _{p}\right) \max \left( 1,V(M)\right) \left\Vert u\right\Vert
_{N}^{2}  \label{30}
\end{equation}%
and taking account of the constraint 
\begin{equation*}
\int_{M}f\left\vert u\right\vert ^{N}dv_{g}=\left( 1+\left\Vert a\right\Vert
_{s}+\left\Vert b\right\Vert _{p}\right) ^{\frac{N}{2}}
\end{equation*}%
we get%
\begin{equation}
\left\Vert u\right\Vert _{N}^{2}\leq \left( 1+\left\Vert a\right\Vert
_{s}+\left\Vert b\right\Vert _{p}\right) \min_{x\in M}f(x)^{-\frac{2}{N}}%
\text{.}  \label{31}
\end{equation}%
Now letting $\eta $ sufficiently small in ( \ref{30} ) we obtain that

\begin{equation*}
J(u)\geq -\max \left( 1,V(M)\right) \left( C\left\Vert a\right\Vert
_{s}+\left\Vert b\right\Vert _{p}\right) \left( 1+\left\Vert a\right\Vert
_{s}+\left\Vert b\right\Vert _{p}\right) \left( \min_{x\in M}f(x)\right) ^{-%
\frac{2}{N}}>-\infty \text{.}
\end{equation*}%
Hence $Q(M)=\inf_{u\in A}J(u)$ is finite.

Let $\left( u_{i}\right) _{i}\subset A$ be a minimizing sequence of the
functional $J$ i.e.

\begin{equation*}
J(u_{i})=Q(M)+o(1)\text{.}
\end{equation*}%
So for sufficiently large $i$ 
\begin{equation}
J(u_{i})\leq Q(M)+1  \label{32}
\end{equation}%
By (\ref{30}),(\ref{31}) and (\ref{32}), we obtain for sufficiently large $i$%
\begin{equation*}
\left\Vert \Delta u_{i}\right\Vert _{2}<+\infty
\end{equation*}%
and by (\ref{29}), (\ref{31}) and H\"{o}lder's inequality we infer that%
\begin{equation*}
\left\Vert u_{i}\right\Vert _{H_{2}(M)}<+\infty \text{.}
\end{equation*}%
Up to a subsequence, there is $u\in H_{2}(M)$ such that

$\cdot $\ \ \ $u_{i}\rightarrow u$ \ \ \ \ \ \ weakly in $H_{2}(M)$

$\cdot $\ $\ \ \nabla u_{i}\rightarrow \nabla u$ \ \ strongly in $L^{s}(M)$, 
$s<2^{\ast }=\frac{2n}{n-2}$

$\cdot $ $\ \ \ u_{i}\rightarrow u$ \ \ \ \ \ \ \ strongly in $L^{r}(M)$, $\
r<N$

$\cdot $ $\ \ \ u_{i}\rightarrow u$ \ \ \ \ \ \ \ \ a.e. in $M$.

Letting $v_{i}=u_{i}-u$, we conclude that

$\int_{M}\Delta u\Delta v_{i}dv_{g}\rightarrow 0$, $\int_{M}a\left\langle
\nabla u,\nabla v_{i}\right\rangle dv_{g}\rightarrow 0$ as $i\rightarrow
+\infty $.

and%
\begin{equation*}
\int_{M}\left\vert buv_{i}\right\vert dv_{g}\leq \left\Vert b\right\Vert
_{p}\left( \int_{M}\left\vert u\right\vert ^{\frac{p}{p-1}}\left\vert
v_{i}\right\vert ^{\frac{p}{p-1}}dv_{g}\right) ^{1-\frac{1}{p}}\leq
\left\Vert b\right\Vert _{p}\left\Vert u\right\Vert _{\frac{2p}{p-1}%
}\left\Vert v_{i}\right\Vert _{\frac{2p}{p-1}}
\end{equation*}%
i.e. $\int_{M}buv_{i}dv_{g}\rightarrow 0$, since $\frac{2p}{p-1}<N$.

Consequently%
\begin{equation*}
J(u_{i})=J(u)+J(v_{i})+2\int_{M}\Delta u\Delta
v_{i}dv_{g}+2\int_{M}a\left\langle \nabla u,\nabla v_{i}\right\rangle
dv_{g}+2\int_{M}buv_{i}dv_{g}
\end{equation*}%
\begin{equation*}
=J(u)+J(v_{i})+o(1)
\end{equation*}%
\begin{equation*}
=J(u)+\left\Vert \Delta v_{i}\right\Vert _{2}^{2}+o(1)\text{.}
\end{equation*}

By definition of $Q(M)$, $J(u)\geq Q(M)\left( \int_{M}f\left\vert
u\right\vert ^{N}dv_{g}\right) ^{\frac{2}{N}}$ and $J(u_{i})=Q(M)+o(1)$ and
by definition of the sequence $\left( u_{i}\right) $, we obtain 
\begin{equation}
Q(M)\left( \int_{M}f\left\vert u\right\vert ^{N}dv_{g}\right) ^{\frac{2}{N}%
}+\left\Vert \Delta v_{i}\right\Vert _{2}^{2}\leq Q(M)+o(1)\text{.}
\label{37}
\end{equation}%
Brezis -Lieb lemma allows us to write

\begin{equation*}
\left( 1+\left\Vert a\right\Vert _{s}+\left\Vert b\right\Vert _{p}\right) ^{%
\frac{N}{2}}=\int_{M}f\left\vert u_{i}\right\vert
^{N}dv_{g}=\int_{M}f\left\vert u\right\vert ^{N}dv_{g}+\int_{M}f\left\vert
v_{i}\right\vert ^{N}dv_{g}+o(1)
\end{equation*}%
hence%
\begin{equation*}
1+\left\Vert a\right\Vert _{s}+\left\Vert b\right\Vert _{p}\leq \left(
\int_{M}f\left\vert u\right\vert ^{N}dv_{g}\right) ^{\frac{2}{N}}+\left(
\int_{M}f\left\vert v_{i}\right\vert ^{N}dv_{g}\right) ^{\frac{2}{N}}+o(1)
\end{equation*}%
and the inequality (\ref{37}) will be written as%
\begin{equation}
\left( 1+\left\Vert a\right\Vert _{s}+\left\Vert b\right\Vert _{p}\right)
\left\Vert \Delta v_{i}\right\Vert _{2}^{2}\leq Q(M)\left(
\int_{M}f\left\vert v_{i}\right\vert ^{N}dv_{g}\right) ^{\frac{2}{N}}+o(1)%
\text{.}  \label{38}
\end{equation}%
By Sobolev's inequality we infer that%
\begin{equation*}
\left( 1+\left\Vert a\right\Vert _{s}+\left\Vert b\right\Vert _{p}\right)
\left\Vert \Delta v_{i}\right\Vert _{2}^{2}\leq Q(M)\left( \sup_{x\in
M}f(x)\right) ^{\frac{2}{N}}\left( K(n,2)^{2}+\varepsilon \right) \left\Vert
\Delta v_{i}\right\Vert _{2}^{2}+o(1)\text{.}
\end{equation*}%
Finally%
\begin{equation*}
\left( 1+\left\Vert a\right\Vert _{s}+\left\Vert b\right\Vert
_{p}-Q(M)\left( \sup_{x\in M}f(x)\right) ^{\frac{2}{N}}\left(
K(n,2)^{2}+\varepsilon \right) \right) \left\Vert \Delta v_{i}\right\Vert
_{2}^{2}\leq o(1)\text{.}
\end{equation*}%
If we let 
\begin{equation*}
Q(M)<\left( 1+\left\Vert a\right\Vert _{s}+\left\Vert b\right\Vert
_{p}\right) \left( \sup_{x\in M}f(x)\right) ^{-\frac{2}{N}}K(n,2)^{-2}
\end{equation*}%
and choosing $\varepsilon >0$ small enough such that 
\begin{equation*}
1+\left\Vert a\right\Vert _{s}+\left\Vert b\right\Vert _{p}-Q(M)\left(
\sup_{x\in M}f(x)\right) ^{\frac{2}{N}}\left( K(n,2)^{2}+\varepsilon \right)
>0
\end{equation*}%
we obtain 
\begin{equation*}
\left\Vert \Delta v_{i}\right\Vert _{2}^{2}=o(1)\text{.}
\end{equation*}%
Hence $\left( v_{i}\right) $ converges strongly to $0$ in $H_{2}(M)$ and $%
(u_{i})$ converges strongly to $u$ in $H_{2}(M)$ and in $L^{N}(M)$. We
conclude that $u\in A$ is a non trivial solution of the equation%
\begin{equation}
\Delta ^{2}u-\nabla ^{\mu }(a\nabla _{\mu }u)+bu=f\left\vert u\right\vert
^{N-2}u\text{.}  \label{39}
\end{equation}
\end{proof}

Now, we are going to establish the regularity of the solution to the
equation (\ref{39}). To do so we quote after F. Robert \cite{20}, the
following regularity theorems.

\begin{theorem}
Let $\left( M,g\right) $ be a compact Riemannian manifold of dimension $%
n\geq 1$. Let $p\geq 1$ and let $0\leq m<k$ two integers such that $%
n>p\left( k-m\right) $. Then $H_{k}^{p}(M)$ is embedded in $H_{m}^{q}(M)$,
where $\frac{1}{q}=\frac{1}{p}-\frac{k-m}{n}$.
\end{theorem}

\begin{theorem}
\label{th5} Let $\left( M,g\right) $ be a compact Riemannian manifold of
dimension $n\geq 1$. Let $p\geq 1$ and $k\geq 1$ an integer such that $kp>n$%
. Then $H_{k}^{p}\left( M\right) $ is embedded in $C^{0,\beta }(M)$ for all $%
\beta \in \left( 0,1\right) $ such that $\beta <k-\frac{n}{p}$.
\end{theorem}

The regularity theorem states as follows

\begin{theorem}
\label{th6} In addition to the assumption of Theorem \ref{th1} we suppose
that the function $a\in H_{1}^{s}\left( M\right) $ with $s>\frac{n}{2}$.
Then the solution $u$ of the equation $\Delta ^{2}u-\nabla ^{\mu }(a\nabla
_{\mu }u)+bu=f\left\vert u\right\vert ^{N-2}u$ is in $C^{0,\beta }(M)$ for
all $\beta \in \left( 0,1-\frac{n}{4p}\right) $ with $p>\frac{n}{4}$.
\end{theorem}

\begin{proof}
We adapt some ideas from Madani's paper in case of Yamabe type equation \cite%
{18}. First, we show that the function $h=-f\left\vert u\right\vert ^{N-2}$
is a Kato- Stummel's function. Using H\"{o}lder's inequality, we get%
\begin{equation*}
\sup_{Q\in M}\int_{B_{t}\left( Q\right) }\frac{\left\vert h(S)\right\vert
\chi _{U}\left( S\right) }{d(Q,S)^{l}}dv_{g}(S)\leq \int_{M}\frac{\left\vert
f(S)\right\vert \left\vert u(S)\right\vert ^{N-2}}{d(Q,S)^{l}}dv_{g}(S)
\end{equation*}%
\begin{equation*}
\leq \max_{S\in M}\left\vert f(S)\right\vert \left\Vert u\right\Vert
_{N-2,\rho ^{-l}}^{N-2}\text{.}
\end{equation*}%
By Lemma \ref{lem3}, we infer that 
\begin{equation*}
\left\Vert u\right\Vert _{N-2,\rho ^{-l}}^{N-2}\leq C\left( \left\Vert
\Delta u\right\Vert _{2}^{2}+\left\Vert u\right\Vert _{2}^{2}\right)
\end{equation*}%
where $C>0$ is some constant.

The remaining part to check that the function $h$ is a Kato-Stummel is the
same as in \cite{18} so we omit it.

We conclude that the solution $u$ of the equation (\ref{39}) is locally
bounded and in fact bounded since the manifold $M$ is compact.

Writing 
\begin{equation}
\left( \Delta +1\right) ^{2}u=div(\left( a-2\right) \bigtriangledown
u)+\left( 1-b\right) u+f\left\vert u\right\vert ^{N-2}u\text{.}  \label{39'}
\end{equation}%
Put for brevity $\widetilde{f}=f\left\vert u\right\vert ^{N-2}u$, $q=\left(
1-b\right) u$, $h=div(\left( a-2\right) \bigtriangledown u)$. Since $u$ is
bounded and $b\in L^{p}\left( M\right) $, $p>\frac{n}{4}$ it follows that $%
\widetilde{f}$ is bounded and $q\in L^{p}\left( M\right) $. Easy
computations using H\"{o}lder's inequality show that $h\in L^{1}\left(
M\right) $. From equation (\ref{39'}) we deduce that 
\begin{equation*}
u=\left( \Delta +1\right) ^{-2}\left[ h+\left( 1-b\right) u+f\left\vert
u\right\vert ^{N-2}u\right] \in H_{4}^{1}\left( M\right) \text{.}
\end{equation*}%
Hence $\bigtriangledown ^{i}u\in L^{\frac{2n}{n+2i-4}}$, $i=0,1$. Let $%
1<q_{1}<\frac{n}{n-2}$, using the H\"{o}lder's inequality, we obtain%
\begin{equation*}
\int_{M}\left\vert div\left( a\bigtriangledown u\right) \right\vert
^{q_{1}}dv_{g}\leq C_{q_{1}}\int_{M}\left( \left\vert \bigtriangledown
a\right\vert ^{q_{1}}\left\vert \bigtriangledown u\right\vert
^{q_{1}}+\left\vert a\right\vert ^{q_{1}}\left\vert \bigtriangledown
^{2}u\right\vert ^{q_{1}}\right) dv_{g}
\end{equation*}%
\begin{equation*}
\leq C_{q_{1}}\left( \left\Vert \bigtriangledown a\right\Vert _{\frac{2nq_{1}%
}{n-\left( n-2\right) q_{1}}}^{q_{1}}\left\Vert \bigtriangledown
u\right\Vert _{\frac{2n}{n-2}}^{q_{1}}+\left\Vert a\right\Vert _{\frac{2q_{1}%
}{2-q_{1}}}^{1-\frac{1}{q_{1}}}\left\Vert \bigtriangledown ^{2}u\right\Vert
_{2}^{q_{1}}\right)
\end{equation*}%
with $C_{q_{1}}>0$, is a constant. Since it is easy to see that%
\begin{equation*}
\frac{2nq_{1}}{n-\left( n-2\right) q_{1}}\leq \frac{n}{2}\text{ \ \ and }%
\frac{2q_{1}}{2-q_{1}}<\frac{ns}{n-s}\text{ for }s>\frac{n}{2}
\end{equation*}%
it follows that 
\begin{equation*}
div\left( a\bigtriangledown u\right) \in L^{q_{1}}\left( M\right) \text{.}
\end{equation*}%
Consequently, from equation (\ref{39'}) we get that 
\begin{equation*}
\bigtriangledown ^{i}u\in L^{\frac{nq_{1}}{n-\left( 4-i\right) q_{1}}}\text{%
, }i=0,1,2\text{.}
\end{equation*}%
Let $\frac{n}{n-2}<q_{2}<\frac{nq_{1}}{n-q_{1}}$, using again the H\"{o}%
lder's inequality, we get%
\begin{equation*}
\int_{M}\left\vert div\left( a\bigtriangledown u\right) \right\vert
^{q_{2}}dv_{g}\leq
\end{equation*}%
\begin{equation*}
\leq C_{q_{2}}\left( \left\Vert \bigtriangledown a\right\Vert _{\frac{%
nq_{1}q_{2}}{nq_{1}-\left( n-3q_{1}\right) q_{2}}}\left\Vert
\bigtriangledown u\right\Vert _{\frac{nq_{1}}{n-3q_{1}}}^{q_{2}}+\left\Vert
a\right\Vert _{\frac{nq_{1}q_{2}}{nq_{1}-\left( n-2q_{1}\right) q_{2}}%
}\left\Vert \bigtriangledown ^{2}u\right\Vert _{\frac{nq_{1}}{n-2q_{1}}%
}^{q_{2}}\right) \text{.}
\end{equation*}%
Since$\frac{nq_{1}q_{2}}{nq_{1}-\left( n-3q_{1}\right) q_{2}}<\frac{n}{2}$
and $\frac{nq_{1}q_{2}}{nq_{1}-\left( n-2q_{1}\right) q_{2}}<\frac{ns}{n-s}$
for $s>\frac{n}{2}$, we infer that 
\begin{equation*}
div\left( a\bigtriangledown u\right) \in L^{q_{2}}\left( M\right)
\end{equation*}%
consequently%
\begin{equation*}
\bigtriangledown ^{i}u\in L^{\frac{nq_{2}}{n-\left( 4-i\right) q_{2}}}\text{%
, }i=0,1,2\text{.}
\end{equation*}%
Recurrently, we obtain an increasing sequence $\left( q_{i}\right) $ such
that $div\left( a\bigtriangledown u\right) \in L^{q_{i}}\left( M\right) $.
If there exists a $q_{i_{o}}$ such that $q_{i_{o}}>\frac{n}{4}$, \ then the
right hand side of (\ref{39'}) belongs to $L^{p}\left( M\right) $, with $p>%
\frac{n}{4}$ and by Theorem\ref{th5}, $u\in C^{0,\beta }$, for $\beta \in
\left( 0,1\right) $ such that $\beta <1-\frac{n}{4p}$. The case the sequence 
$\left( q_{i}\right) $ is bounded by $\frac{n}{4}$, denote by $q_{o}$ the
limit of $q_{i}$ and let $q_{o}<q<\frac{nq_{o}}{n-q_{o}}$, the same
arguments as above show that $div\left( a\bigtriangledown u\right) \in
L^{q}\left( M\right) $ which contradicts the fact that the sequence $\left(
q_{i}\right) $ is bounded by $\frac{n}{4}$.
\end{proof}

\begin{remark}
\label{rem1} The assumption $a\in H_{1}^{s}\left( M\right) $ in Theorem \ref%
{th6} is not restrictive: since if $s$ is a real number such that $0<\gamma <%
\frac{n}{s}-1<1$. Let $\rho $ be the function defined in the introduction by
(\ref{1}), the singular function on $M$ given by 
\begin{equation*}
a(x)=\frac{\widetilde{a}(x)}{\rho ^{\gamma }}
\end{equation*}%
where $\widetilde{a}$ is a smooth function on $M$ is such that $%
\bigtriangledown a\in H_{1}^{s}\left( M\right) $.
\end{remark}

Let $\alpha $,$\gamma $ be real numbers which will be precise later and
consider the equation in the distribution sense%
\begin{equation}
\Delta ^{2}u-\nabla ^{\mu }(\frac{a}{\rho ^{\gamma }}\nabla _{\mu }u)+\frac{%
bu}{\rho ^{\alpha }}=f\left\vert u\right\vert ^{N-2}u  \label{40}
\end{equation}%
where $a$ and $b$ are smooth functions on $M$.

We put, for any $u\in H_{2}(M)$%
\begin{equation*}
J_{\gamma ,\alpha }(u)=\left\Vert \Delta u\right\Vert _{2}^{2}+\int_{M}\frac{%
a}{\rho ^{\gamma }}\left\vert \nabla u\right\vert ^{2}dv_{g}+\int_{M}\frac{%
bu^{2}}{\rho ^{\alpha }}dv_{g}
\end{equation*}%
\begin{equation*}
Q_{\gamma ,\alpha }(u)=\frac{J_{\gamma ,\alpha }(u)}{\left(
\int_{M}f\left\vert u\right\vert ^{N}dv_{g}\right) ^{\frac{2}{N}}}
\end{equation*}%
and%
\begin{equation*}
Q_{\gamma ,\alpha }(M)=\inf_{u\in H_{2}(M)-\left\{ 0\right\} }Q_{\gamma
,\alpha }(u)=\inf_{u\in A}J_{\alpha }(u)
\end{equation*}%
where $A=\left\{ u\in H_{2}(M):\int_{M}f\left\vert u\right\vert
^{N}dv_{g}=\left( 1+\left\Vert \frac{a}{\rho ^{\gamma }}\right\Vert
_{s}+\left\Vert \frac{b}{\rho ^{\alpha }}\right\Vert _{p}\right) ^{\frac{N}{2%
}}\right\} $.

As a corollary to Theorem \ref{th1}, we have

\begin{theorem}
\label{th2} Let $\gamma $, $\alpha $ real numbers such that $0<\gamma <\frac{%
n}{s}<2$ \ and $0<\alpha <\frac{n}{p}<4$, if $Q_{\gamma ,\alpha }(M)<\left(
\sup f(x)\right) ^{-\frac{N}{2}}K(n,2)^{-2}\left( 1+\left\Vert \frac{a}{\rho
^{\gamma }}\right\Vert _{s}+\left\Vert \frac{b}{\rho ^{\alpha }}\right\Vert
_{p}\right) $, then the equation (\ref{40}) has a non trivial weak solution $%
u_{\gamma ,\alpha }\in A$ which fulfills $J_{\gamma ,\alpha }(u)=Q_{\gamma
,\alpha }(M)$. Moreover if \ $0<\gamma <\frac{n}{s}-1<1$, then the solution $%
u_{\gamma ,\alpha }$ is\ of class $C^{0,\beta }(M)$ \ with $\beta \in \left(
0,1-\frac{n}{4p}\right) $.
\end{theorem}

\begin{proof}
Let $\overline{a}=\frac{a}{\rho ^{\gamma }}$, $\overline{b}=\frac{b}{\rho
^{\alpha }}$, if $0<\gamma <\frac{n}{s}$, $0<\alpha <\frac{n}{p}<4$ then, $%
\overline{a}\in L^{s}(M),\overline{b}\in L^{p}(M)$ and the first part of
Theorem \ref{th2} follows from Theorem \ref{th1} . If $0<\gamma <\frac{n}{s}%
-1<1$, then $\overline{a}\in H_{1}^{s}(M)$ and as a Corollary of Theorem \ref%
{th6} \ we \ get the last part of Theorem \ref{th2}.
\end{proof}

\section{The sharp case $\protect\gamma =2$, $\protect\alpha =4$}

In the previous section we have shown that if $0<\gamma <\frac{n}{s}<2$, $%
0<\alpha <\frac{n}{p}<4$ and $Q_{\gamma ,\alpha }(M)<\left( \sup_{x\in
M}f(x)\right) ^{-\frac{2}{N}}K(n,2)^{-2}\left( 1+\left\Vert \frac{a}{\rho
^{\gamma }}\right\Vert _{s}+\left\Vert \frac{b}{\rho ^{\alpha }}\right\Vert
_{p}\right) $, then the equation in the distribution sense

\begin{equation}
\Delta ^{2}u-\nabla ^{\mu }(\frac{a}{\rho ^{\gamma }}\nabla _{\mu }u)+\frac{%
bu}{\rho ^{\alpha }}=f\left\vert u\right\vert ^{N-2}u  \label{41}
\end{equation}%
where $a$ and $b$ are smooth functions, has a weak solution 
\begin{equation*}
u_{\gamma ,\alpha }\in A=\left\{ u\in H_{2}(M):\int_{M}f\left\vert
u\right\vert ^{N}dv_{g}=\left( 1+\left\Vert \frac{a}{\rho ^{\gamma }}%
\right\Vert _{s}+\left\Vert \frac{b}{\rho ^{\alpha }}\right\Vert _{p}\right)
^{\frac{N}{2}}\right\}
\end{equation*}%
such that $J_{\gamma ,\alpha }(u_{\alpha })=Q_{\gamma ,\alpha }(M)$. If we
restrict ourself to $0<\gamma <\frac{n}{s}-1<1$, $u_{\alpha }\in C^{0,\beta
}\left( M\right) $ with $\beta \in \left( 0,1\right) $.

In this section we will show the following: let $K\left( n,1,-2\right) $ and 
$K(n,2,-4)$ be the best constants in the Hardy inclusion $%
H_{1}^{2}(M)\hookrightarrow L^{\frac{2n}{n-2}}\left( M,\rho ^{-2}\right) $
and $H_{2}^{2}(M)\hookrightarrow L^{N}\left( M,\rho ^{-4}\right) $
respectively and where $N=\frac{2n}{n-4}$.

Denote by $\delta \left( M\right) $ the injectivity radius of the compact
Riemannian manifold $M$ and let $\omega _{n-1}$ be the volume of the $n-1$
dimensional Euclidean unit sphere $S^{n-1}$.

\begin{theorem}
Let $a$, $b$ and $f$ be smooth functions on $M$ with $f$ positive. Suppose
that 
\begin{equation*}
Q_{2,4}(M)K(n,2)^{2}\left( \sup_{x\in M}f\right) ^{\frac{2}{N}}<\left(
1+\left\Vert \frac{a}{\rho ^{\gamma }}\right\Vert _{s}+\left\Vert \frac{b}{%
\rho ^{\alpha }}\right\Vert _{p}\right) \left( 1+b(P)K(n,2,-4)^{2}\right) .
\end{equation*}%
If moreover we have 
\begin{equation*}
1+b(P)K(n,2,-4)^{2}>0,1+a\left( P\right) K\left( n,1,-2\right)
^{2}+b(P)K(n,2,-4)^{2}>0
\end{equation*}

then the equation in the distribution sense%
\begin{equation*}
\Delta ^{2}u-\nabla ^{\mu }(\frac{a}{\rho ^{2}}\nabla _{\mu }u)+\frac{bu}{%
\rho ^{4}}=f\left\vert u\right\vert ^{N-2}u
\end{equation*}%
has a non trivial solution $u_{2,4}\in A$, which fulfilled $%
J_{2,4}(u)=Q_{2,4}(M)$.
\end{theorem}

\begin{proof}
We adapt ideas from Madani's paper \cite{18}. First we show that $Q_{2,4}(M)$
is finite. Since $b$ is continuous, for any $\varepsilon >0$, there is $%
\delta >0$ such that for any $Q\in M$, with $d(P,Q)<\delta $ and $%
b(Q)>b(P)-\varepsilon $, then%
\begin{equation}
\int_{M}\frac{bu^{2}}{\rho ^{4}}dv_{g}\geq \left( b(P)-\varepsilon \right)
\int_{B_{\delta }(P)}\frac{u^{2}}{\rho ^{4}}dv_{g}-\frac{\left\Vert
b\right\Vert _{\infty }}{\delta ^{4}}\int_{M-B_{\delta }\left( P\right)
}u^{2}dv_{g}\text{.}  \label{42}
\end{equation}%
By Lemma \ref{lem3}, we have%
\begin{equation}
\int_{B_{\delta }(P)}\frac{u^{2}}{\rho ^{4}}dv_{g}\leq \left(
K(n,2,-4)^{2}+\varepsilon \right) \left\Vert \Delta u\right\Vert
_{2}^{2}+A(\varepsilon )\left\Vert u\right\Vert _{2}^{2}\text{.}  \label{43}
\end{equation}%
Combining (\ref{42}) and (\ref{43}), we get%
\begin{equation}
\int_{M}\frac{bu^{2}}{\rho ^{4}}dv_{g}\geq \left( \min (b(P),0)-\varepsilon
\right) \left( K(n,2,-4)^{2}+\varepsilon \right) \left\Vert \Delta
u\right\Vert _{2}^{2}  \label{44}
\end{equation}%
\begin{equation*}
+\left( \left( \min (b(P),0)-\varepsilon \right) A(\varepsilon )-\frac{%
\left\Vert b\right\Vert _{\infty }}{\delta ^{4}}\right) \left\Vert
u\right\Vert _{2}^{2}
\end{equation*}%
where $\left\Vert b\right\Vert _{\infty }=\sup_{x\in M}\left\vert
b(x)\right\vert $, and since, 
\begin{equation*}
\left\Vert u\right\Vert _{2}^{2}\leq \left( 1+\left\Vert \frac{a}{\rho
^{\gamma }}\right\Vert _{s}+\left\Vert \frac{b}{\rho ^{\alpha }}\right\Vert
_{p}\right) \left( \min_{x\in M}f(x)\right) ^{-\frac{2}{N}}V(M)^{1-\frac{2}{N%
}}
\end{equation*}%
we infer the following inequality%
\begin{equation}
\int_{M}\frac{bu^{2}}{\rho ^{4}}dv_{g}\geq \left( \min (b(P),0)-\varepsilon
\right) \left( K(n,2,-4)^{2}+\varepsilon \right) \left\Vert \Delta
u\right\Vert _{2}^{2}  \label{45}
\end{equation}%
\begin{equation*}
+\left( \left( \min (b(P),0)-\varepsilon \right) A(\varepsilon )-\frac{%
\left\Vert b\right\Vert _{\infty }}{\delta ^{4}}\right) \left( 1+\left\Vert 
\frac{a}{\rho ^{\gamma }}\right\Vert _{s}+\left\Vert \frac{b}{\rho ^{\alpha }%
}\right\Vert _{p}\right) \left( \min_{x\in M}f(x)\right) ^{-\frac{2}{N}%
}V(M)^{1-\frac{2}{N}}\text{.}
\end{equation*}%
Also, we have 
\begin{equation*}
\int_{M}\frac{a\left\vert \nabla u\right\vert ^{2}}{\rho ^{2}}dv_{g}\geq
\left( \min (a(P),0)-\varepsilon \right) \int_{B_{\delta }(P)}\frac{%
\left\vert \nabla u\right\vert ^{2}}{\rho ^{2}}dv_{g}-\frac{\left\Vert
a\right\Vert _{\infty }}{\delta ^{2}}\int_{M-B_{\delta }\left( P\right)
}\left\vert \nabla u\right\vert ^{2}dv_{g}
\end{equation*}%
and by Hardy's inequality given by Lemma \ref{lem2}%
\begin{equation*}
\int_{M}\frac{a\left\vert \nabla u\right\vert ^{2}}{\rho ^{2}}dv_{g}\geq
\left( \min (a(P),0)-\varepsilon \right) (K(n,1,-2)+\varepsilon )\left\Vert
\nabla ^{2}u\right\Vert _{2}^{2}
\end{equation*}%
\begin{equation*}
+\left( \left( \min (a(P),0)-\varepsilon \right) A(1,\varepsilon )-\frac{%
\left\Vert a\right\Vert _{\infty }}{\delta ^{2}}\right) \left\Vert \nabla
u\right\Vert _{2}^{2}\text{.}
\end{equation*}%
By the Bochner-Lichnerowicz-Weitzenbock formula%
\begin{equation*}
\left\Vert \Delta u\right\Vert _{2}^{2}=\int_{M}\left\vert \nabla
^{2}u\right\vert dv_{g}+\int_{M}Ric_{g}\left( \left( \nabla u\right)
^{\natural },\left( \nabla u\right) ^{\natural }\right) dv_{g}
\end{equation*}%
where $\natural $ is the musical operator, we get%
\begin{equation*}
\int_{M}\frac{a\left\vert \nabla u\right\vert ^{2}}{\rho ^{2}}dv_{g}\geq
\left( \min (a(P),0)-\varepsilon \right) (K(n,1,-2)+\varepsilon )\left\Vert
\Delta u\right\Vert _{2}^{2}
\end{equation*}%
\begin{equation*}
+\left( \left( \min (a(P),0)-\varepsilon \right) A(1,\varepsilon )-\frac{%
\left\Vert a\right\Vert _{\infty }}{\delta ^{2}}+c\right) \left\Vert \nabla
u\right\Vert _{2}^{2}
\end{equation*}%
where $c$ is a constant such%
\begin{equation*}
\left\vert \int_{M}Ric_{g}\left( \left( \nabla u\right) ^{\natural },\left(
\nabla u\right) ^{\natural }\right) dv_{g}\right\vert \leq c\left\Vert
\nabla u\right\Vert _{2}^{2}\text{.}
\end{equation*}%
Now taking again account of relation (\ref{33}), we obtain%
\begin{equation*}
\int_{M}\frac{a\left\vert \nabla u\right\vert ^{2}}{\rho ^{2}}dv_{g}\geq
\end{equation*}%
\begin{equation*}
\left( \left( \min (a(P),0)-\varepsilon \right) K(n,1,-2)^{2}+\varepsilon
)+\left( \left( \min (a(P),0)-\varepsilon \right) A(1,\varepsilon )-\frac{%
\left\Vert a\right\Vert _{\infty }}{\delta ^{2}}+c\right) \frac{\eta }{%
1-\eta \beta }\right) \left\Vert \Delta u\right\Vert _{2}^{2}
\end{equation*}%
\begin{equation*}
+\left( \left( \min (a(P),0)-\varepsilon \right) A(1,\varepsilon )-\frac{%
\left\Vert a\right\Vert _{\infty }}{\delta ^{2}}+c\right) \frac{C\left( \eta
\right) }{1-\eta \beta }\left\Vert u\right\Vert _{2}^{2}\text{.}
\end{equation*}%
As in previous sections, we obtain%
\begin{equation*}
J_{2,4}(u)\geq \left[ 1+\left( \min (a(P),0)-\varepsilon \right)
K(n,1,-2)^{2}+\varepsilon )\right.
\end{equation*}%
\begin{equation*}
\left. +\left( \left( \min (a(P),0)-\varepsilon \right) A(1,\varepsilon )-%
\frac{\left\Vert a\right\Vert _{\infty }}{\delta ^{2}}+c\right) \frac{\eta }{%
1-\eta \beta }+\left( \min (b(P),0)-\varepsilon \right) \left(
K(n,2,-4)^{2}+\varepsilon \right) \right] \left\Vert \Delta u\right\Vert
_{2}^{2}
\end{equation*}%
\begin{equation*}
+\left( \left( \min (a(P),0)-\varepsilon \right) A(1,\varepsilon )-\frac{%
\left\Vert a\right\Vert _{\infty }}{\delta ^{2}}\right) \frac{C\left( \eta
\right) }{1-\eta \beta }\left\Vert u\right\Vert _{N}^{2}V(M)^{1-\frac{2}{N}}
\end{equation*}%
\begin{equation}
+\left( \left( \min (b(P),0)-\varepsilon \right) A(\varepsilon )-\frac{%
\left\Vert b\right\Vert _{\infty }}{\delta ^{4}}\right) \left( 1+\left\Vert 
\frac{a}{\rho ^{\gamma }}\right\Vert _{s}+\left\Vert \frac{b}{\rho ^{\alpha }%
}\right\Vert _{p}\right) \left( \min_{x\in M}f(x)\right) ^{-\frac{2}{N}%
}V(M)^{1-\frac{2}{N}}\text{.}  \label{46}
\end{equation}%
Noting that 
\begin{equation}
\lim_{\gamma \rightarrow 2^{-}}\sup \left\Vert \frac{a}{\rho ^{\gamma }}%
\right\Vert _{s}\leq \left\Vert a\right\Vert _{\infty }\left( \omega
_{n-1}\delta (M)\right) ^{\frac{2}{n}}<\infty  \tag{46'}  \label{46'}
\end{equation}%
where $\delta (M)$ is the injectivity radius and $\omega _{n-1}$ denotes the
volume of the $n-1$ Euclidean unit sphere.

So if 
\begin{equation*}
1+(a(P)K(n,1,-2)^{2}+b(P)K(n,2,-4)^{2}>0
\end{equation*}%
and letting $\varepsilon $ and $\eta $ small enough the inequality (\ref{46}%
) becomes%
\begin{equation*}
J_{2,4}(u)\geq \left( \left( \min (a(P),0)-\varepsilon \right)
A_{1}(\varepsilon )-\frac{\left\Vert a\right\Vert _{\infty }}{\delta ^{2}}%
\right) \frac{C\left( \eta \right) }{1-\eta \beta }\left( \min_{x\in
M}f(x)\right) ^{-\frac{2}{N}}V(M)^{1-\frac{2}{N}}
\end{equation*}%
\begin{equation*}
+\left( \left( \min (b(P),0)-\varepsilon \right) A(\varepsilon )-\frac{%
\left\Vert b\right\Vert _{\infty }}{\delta ^{4}}\right) \left( 1+\left\Vert 
\frac{a}{\rho ^{\gamma }}\right\Vert _{s}+\left\Vert \frac{b}{\rho ^{\alpha }%
}\right\Vert _{p}\right) \left( \min_{x\in M}f(x)\right) ^{-\frac{2}{N}%
}V(M)^{1-\frac{2}{N}}>-\infty \text{.}
\end{equation*}%
Consequently%
\begin{equation*}
Q_{2,4}(M)>-\infty \text{.}
\end{equation*}%
In the second step, we will show that $Q_{\gamma ,\alpha }(M)\rightarrow
Q_{2,4}(M)$ as $\gamma \rightarrow 2^{-}$ , $\alpha \rightarrow 4^{-}$.

Let $0<\delta <\min (1,\delta \left( M\right) )$, where $\delta (M)$ denotes
the injectivity radius, then 
\begin{equation*}
\int_{M}\frac{bu^{2}}{\rho ^{\alpha }}dv_{g}=\int_{B_{\delta }(P)}\frac{%
bu^{2}}{\rho ^{\alpha }}dv_{g}+\int_{M-B_{\delta }(P)}\frac{bu^{2}}{\rho
^{\alpha }}dv_{g}
\end{equation*}%
and by the Lebesgue's dominated convergence theorem, we obtain that%
\begin{equation*}
\int_{M}\frac{bu^{2}}{\rho ^{\alpha }}dv_{g}\rightarrow \int_{M}\frac{bu^{2}%
}{\rho ^{4}}dv_{g}\text{ as }\alpha \rightarrow 4^{-}\text{.}
\end{equation*}%
The same arguments are also true for%
\begin{equation*}
\int_{M}\frac{a\left\vert \nabla u\right\vert ^{2}}{\rho ^{\gamma }}%
dv_{g}\rightarrow \int_{M}\frac{a\left\vert \nabla u\right\vert ^{2}}{\rho
^{2}}dv_{g}\text{ as }\gamma \rightarrow 2^{-}\text{.}
\end{equation*}%
Hence 
\begin{equation*}
J_{\gamma ,\alpha }(u)\rightarrow J_{2,4}(u)\text{\ as }\gamma \rightarrow
2^{-}\text{ and }\alpha \rightarrow 4^{-}
\end{equation*}%
and by passing to the infimum over $u$ such that $\int_{M}f\left\vert
u\right\vert ^{N}dv_{g}=\left( 1+\left\Vert \frac{a}{\rho ^{\gamma }}%
\right\Vert _{s}+\left\Vert \frac{b}{\rho ^{\alpha }}\right\Vert _{p}\right)
^{\frac{N}{2}}$, we get%
\begin{equation*}
Q_{\gamma ,\alpha }(M)\rightarrow Q_{2,4}(M)\text{ as }\gamma \rightarrow
2^{-}\text{ and }\alpha \rightarrow 4^{-}\text{.}
\end{equation*}%
In this third step, we will show that the sequence $\left( u_{\gamma ,\alpha
}\right) $ is uniformly bounded in $H_{2}(M)$. This sequence satisfies 
\begin{equation*}
\left\Vert u_{\gamma ,\alpha }\right\Vert _{2}^{2}\leq \left( 1+\left\Vert 
\frac{a}{\rho ^{\gamma }}\right\Vert _{s}+\left\Vert \frac{b}{\rho ^{\alpha }%
}\right\Vert _{p}\right) \left( \min_{x\in M}f(x)\right) ^{-\frac{2}{N}%
}V(M)^{1-\frac{2}{N}}
\end{equation*}%
and taking into account of (\ref{46'}) it follows that $\left\Vert u_{\gamma
,\alpha }\right\Vert _{2}^{2}<+\infty $. Now if $1+a\left( P\right)
K(n,1,-2)+b(P)K(n,2,-4)>0$, then%
\begin{equation*}
\left\Vert \Delta u_{\gamma ,\alpha }\right\Vert _{2}^{2}<+\infty
\end{equation*}%
and by the inequality (\ref{9}), 
\begin{equation*}
\left\Vert \nabla u_{\gamma ,\alpha }\right\Vert _{2}^{2}<+\infty \text{.}
\end{equation*}%
Up to a subsequence $\left( u_{u_{\gamma ,\alpha }}\right) $ converges
weakly in $H_{2}(M)$, $L^{N}(M)$ , $L^{2}(M,\rho ^{-4})$ and strongly to $u$
in $L^{r}(M)$ with $r<\frac{2n}{n-4}$ and $\nabla u_{\gamma ,\alpha }$
converges strongly to $\nabla u$ in $L^{s}(M)$ with $s<\frac{2n}{n-2}$, as $%
\gamma \rightarrow 2^{-}$ and $\alpha \rightarrow 4^{-}$. For any $v\in
H_{2}(M)$%
\begin{equation}
\int_{M}\Delta u_{\gamma ,\alpha }\Delta vdv_{g}+\int_{M}\frac{a}{\rho
^{\gamma }}g(\nabla u_{\gamma ,\alpha },\nabla v)dv_{g}+\int_{M}\frac{b}{%
\rho ^{\alpha }}u_{\gamma ,\alpha }vdv_{g}=Q_{\gamma ,\alpha
}(M)\int_{M}f\left\vert u_{\gamma ,\alpha }\right\vert ^{N-2}u_{\gamma
,\alpha }vdv_{g}\text{.}  \label{47}
\end{equation}%
The weak convergence in $H_{2}(M)$ and the strong convergence of $\nabla
u_{\gamma ,\alpha }\rightarrow \nabla u$ allow us to write

\begin{equation*}
\int_{M}\Delta u_{\gamma ,\alpha }\Delta vdv_{g}\rightarrow \int_{M}\Delta
u\Delta vdv_{g}\text{ \ }
\end{equation*}%
and%
\begin{equation*}
\text{ }\int_{M}\frac{a}{\rho ^{\gamma }}g(\nabla u_{\gamma ,\alpha },\nabla
v)dv_{g}\rightarrow \int_{M}\frac{a}{\rho ^{2}}g(\nabla u,\nabla v)dv_{g}
\end{equation*}%
The convergence of the third integral%
\begin{equation*}
\left\vert \int_{M}\frac{b}{\rho ^{\alpha }}u_{\gamma ,\alpha
}vdv_{g}-\int_{M}\frac{b}{\rho ^{4}}uvdv_{g}\right\vert \leq \left\vert
\int_{M}\frac{b}{\rho ^{4}}v\left( u_{\gamma ,\alpha }-u\right)
dv_{g}\right\vert
\end{equation*}%
\begin{equation*}
+\left\vert \int_{M}\frac{b}{\rho ^{4}}u_{\gamma ,\alpha }vdv_{g}-\int_{M}%
\frac{b}{\rho ^{\alpha }}u_{\gamma ,\alpha }vdv_{g}\right\vert
\end{equation*}%
is assured by the weak convergence in $L^{2}(M,\rho ^{-4})$ and the
dominated Lebesque convergence theorem. Since $\left( u_{\gamma ,\alpha
}\right) $ is bounded in $L^{N}(M)$, the sequence $\left( \left\vert
u_{\gamma ,\alpha }\right\vert ^{N-2}u_{\gamma ,\alpha }\right) $ is bounded
in $L^{\frac{N}{N-1}}(M)$, hence $Q_{\gamma ,\alpha }(M)\int_{M}f\left\vert
u_{\gamma ,\alpha }\right\vert ^{N-2}u_{\gamma ,\alpha }vdv_{g}\rightarrow
Q_{2,4}(M)\int_{M}f\left\vert u\right\vert ^{N-2}uvdv_{g}$. Consequently $u$
is a weak solution of equation (\ref{41}).

In this last step, we will prove that $u$ is not trivial. By Sobolev's
inequality given by Lemma \ref{lem5}, we have 
\begin{equation}
\left( 1+\left\Vert \frac{a}{\rho ^{\gamma }}\right\Vert _{s}+\left\Vert 
\frac{b}{\rho ^{\alpha }}\right\Vert _{p}\right) \sup_{x\in M}f(x)^{-\frac{2%
}{N}}\leq \left\Vert u_{\gamma ,\alpha }\right\Vert _{N}^{2}\leq
(K(n,2)^{2}+\varepsilon )\left\Vert \Delta u_{\gamma ,\alpha }\right\Vert
_{2}^{2}+A(\varepsilon )\left\Vert u_{\gamma ,\alpha }\right\Vert _{2}^{2}%
\text{.}  \label{48}
\end{equation}%
Since $u_{\gamma ,\alpha }$ are solutions (\ref{41}), for any $\overline{%
\eta }>0$,%
\begin{equation}
\left\Vert \Delta u_{\gamma ,\alpha }\right\Vert _{2}^{2}=(1+\overline{\eta }%
)\left( Q_{\gamma ,\alpha }(M)-\int_{M}\frac{a}{\rho ^{\gamma }}\left\vert
\nabla u_{\gamma ,\alpha }\right\vert ^{\ 2}dv_{g}-\int_{M}\frac{bu_{\gamma
,\alpha }^{2}}{\rho ^{\alpha }}dv_{g}\right) -\overline{\eta }\left\Vert
\Delta u_{\gamma ,\alpha }\right\Vert _{2}^{2}\text{.}  \label{49}
\end{equation}%
And since as it is shown in previous sections

\begin{equation}
\left\Vert \nabla u_{\gamma ,\alpha }\right\Vert _{2}^{2}\leq \frac{\eta }{%
1-\eta \beta }\left\Vert \Delta u_{\gamma ,\alpha }\right\Vert _{2}^{2}+%
\frac{C(\eta )}{1-\eta \beta }\left\Vert u_{\gamma ,\alpha }\right\Vert
_{2}^{2}  \label{50}
\end{equation}%
where $\beta >0$ and arbitrary $\eta >0$ such that $1-\eta \beta >0$.

And also by Sobolev's inequality, given by Lemma \ref{lem5} 
\begin{equation*}
\int_{M}\frac{bu_{\gamma ,\alpha }^{2}}{\rho ^{\alpha }}dv_{g}\geq
(b(P)-\varepsilon )\int_{B_{\delta }(P)}\frac{bu_{\gamma ,\alpha }^{2}}{\rho
^{\alpha }}dv_{g}-\frac{\left\Vert a\right\Vert _{\infty }}{\delta ^{4}}%
\left\Vert u_{\gamma ,\alpha }\right\Vert _{2}^{2}
\end{equation*}%
\begin{equation}
\geq \left( \min (b(P),0)-\varepsilon \right) \left[ (K(n,2,-\alpha
)^{2}+\varepsilon _{1})\left\Vert \Delta u_{\gamma ,\alpha }\right\Vert
_{2}^{2}+A(\varepsilon _{1})\left\Vert u_{\gamma ,\alpha }\right\Vert
_{2}^{2}\right] -\frac{\left\Vert a\right\Vert _{\infty }}{\delta ^{4}}%
\left\Vert u_{\gamma ,\alpha }\right\Vert _{2}^{2}\text{.}  \label{51}
\end{equation}%
Plugging (\ref{50}) and (\ref{51}) in (\ref{49}), we get%
\begin{equation*}
\left\Vert \Delta u_{\gamma ,\alpha }\right\Vert _{2}^{2}\leq (1+\overline{%
\eta })\left[ Q_{\gamma ,\alpha }(M)+\frac{\eta }{1-\eta \beta }\left\Vert
a\right\Vert _{\infty }\left\Vert \Delta u_{\gamma ,\alpha }\right\Vert
_{2}^{2}+\frac{C(\eta )}{1-\eta \beta }\left\Vert u_{\gamma ,\alpha
}\right\Vert _{2}^{2}\right.
\end{equation*}%
\begin{equation}
\left. -\left( \min (b(P),0)-\varepsilon \right) \left[ (K(n,2,-\alpha
)^{2}+\varepsilon _{1})\left\Vert \Delta u_{\gamma ,\alpha }\right\Vert
_{2}^{2}+A(\varepsilon _{1})\left\Vert u_{\gamma ,\alpha }\right\Vert
_{2}^{2}\right] +\frac{\left\Vert a\right\Vert _{\infty }}{\delta ^{4}}%
\left\Vert u_{\gamma ,\alpha }\right\Vert _{2}^{2}\right]  \label{52}
\end{equation}%
\begin{equation*}
-\overline{\eta }\left\Vert \Delta u_{\gamma ,\alpha }\right\Vert _{2}^{2}
\end{equation*}%
and taking $\overline{\eta }$ so that 
\begin{equation*}
\overline{\eta }=(1+\overline{\eta })\frac{\eta }{1-\eta \beta }\left\Vert
a\right\Vert _{\infty }
\end{equation*}%
we get%
\begin{equation*}
\left[ 1+(1+\overline{\eta })\left( \min (b(P),0)-\varepsilon \right)
(K(n,2,-\alpha )^{2}+\varepsilon _{1})\right] \left\Vert \Delta u_{\gamma
,\alpha }\right\Vert _{2}^{2}\leq
\end{equation*}%
\begin{equation}
(1+\overline{\eta })\left[ Q_{\gamma ,\alpha }(M)+\left( \frac{C(\eta )}{%
1-\eta \beta }-\left( \min (b(P),0)-\varepsilon \right) A(\varepsilon _{1})+%
\frac{\left\Vert a\right\Vert _{\infty }}{\delta ^{4}}\right) \left\Vert
u_{\gamma ,\alpha }\right\Vert _{2}^{2}\right] \text{.}  \label{53}
\end{equation}%
So if 
\begin{equation}
1+b(P)K(n,2,-\alpha )^{2}>0  \label{54}
\end{equation}%
by letting $\varepsilon ,\varepsilon _{1},$ and $\eta $ small enough, we get%
\begin{equation}
\left\Vert \Delta u_{\gamma ,\alpha }\right\Vert _{2}^{2}\leq \frac{(1+%
\overline{\eta })Q_{\gamma ,\alpha }(M)+\left( \frac{C(\eta )}{1-\eta \beta }%
-\left( \min (b(P),0)-\varepsilon \right) A(1,\varepsilon )+\frac{\left\Vert
a\right\Vert _{\infty }}{\delta ^{4}}\right) \left\Vert u_{\gamma ,\alpha
}\right\Vert _{2}^{2}}{1+(1+\overline{\eta })\left( \min
(b(P),0)-\varepsilon \right) (K(n,2,-\alpha )^{2}+\varepsilon _{1})}
\label{55}
\end{equation}%
and replacing in (\ref{48})%
\begin{equation*}
\left[ (K(n,2)^{2}+\varepsilon )\frac{\left( \frac{C(\eta )}{1-\eta \beta }%
-\left( \min (b(P),0)-\varepsilon \right) A(1,\varepsilon )+\frac{\left\Vert
a\right\Vert _{\infty }}{\delta ^{4}}\right) }{1+(1+\overline{\eta })\left(
\min (b(P),0)-\varepsilon \right) (K(n,2,-\alpha )^{2}+\varepsilon _{1})}%
+A(\varepsilon )\right] \left\Vert u_{\gamma ,\alpha }\right\Vert
_{2}^{2}\geq
\end{equation*}%
\begin{equation*}
\left( 1+\left\Vert \frac{a}{\rho ^{\gamma }}\right\Vert _{s}+\left\Vert 
\frac{b}{\rho ^{\alpha }}\right\Vert _{p}\right) \left( \sup_{x\in
M}f\right) ^{-\frac{2}{N}}-\frac{(1+\overline{\eta })Q_{\gamma ,\alpha
}(M)(K(n,2)^{2}+\varepsilon )}{1+(1+\overline{\eta })\left( \min
(b(P),0)-\varepsilon \right) (K(n,2,-\alpha )^{2}+\varepsilon _{1})}\text{.}
\end{equation*}%
Since 
\begin{equation*}
Q_{\gamma ,\alpha }(M)=Q_{2,4}(M)+o(1)
\end{equation*}%
and in addition of (\ref{54}) and the following assumption 
\begin{equation*}
Q_{2,4}(M)K(n,2)^{2}\left( \sup_{x\in M}f\right) ^{\frac{2}{N}}<\left(
1+\left\Vert \frac{a}{\rho ^{\gamma }}\right\Vert _{s}+\left\Vert \frac{b}{%
\rho ^{\alpha }}\right\Vert _{p}\right) \left( 1+b(P)K(n,2,-4)^{2}\right)
\end{equation*}%
we get that the solution $u$ of the sharp equation is not trivial.
\end{proof}

\section{Geometric interpretation}

Consider a flat manifold $\left( M,h\right) $ for example a flat torus and
let $g=Ah$ where $A:M\rightarrow \left( 0,+\infty \right) $ is a positive
radial function given by $A(x)$ $=e^{-\rho ^{2-\sigma }}$ where $\sigma \in
\left( 0,2\right) $ and $\rho $ is the function defined by (\ref{1}).
Clearly if $0<\sigma <\frac{n}{p}-2<4$, then the metric $g=Ah\in
H_{4}^{p}\left( M,T^{\ast }M\otimes T^{\ast }M\right) $ with $p>\frac{n}{4}$.

The respective expressions of the Ricci tensor and the scalar curvature of $%
(M$,$g)$ are then 
\begin{equation*}
Ric_{g}=-\frac{1}{2}\left[ -\Delta \log A+\frac{n}{2}\left\vert \nabla \log
A\right\vert ^{2}\right] g\text{.}
\end{equation*}%
and that of the scalar curvature is 
\begin{equation*}
R_{g}=-\frac{n}{2}\left[ -\Delta \log A+\frac{n}{2}\left\vert \nabla \log
A\right\vert ^{2}\right]
\end{equation*}%
We infer the following expressions%
\begin{equation*}
\alpha =\frac{\left( n-2\right) ^{2}+4}{2\left( n-1\right) \left( n-2\right) 
}R_{g}.g-\frac{4}{n-2}Ric_{g}
\end{equation*}%
\begin{equation*}
=-\frac{n^{3}-4n^{2}+8}{4(n-1)\left( n-2\right) }\left[ -\Delta \log A+\frac{%
n}{2}\left\vert \nabla \log A\right\vert ^{2}\right] .g\text{.}
\end{equation*}%
Put%
\begin{equation}
\widetilde{\alpha }=-\frac{n^{2}-2n-4}{4(n-1)}\left[ -\Delta \log A+\frac{n}{%
2}\left\vert \nabla \log A\right\vert ^{2}\right] \text{.}  \label{56}
\end{equation}%
and%
\begin{equation*}
\widetilde{\beta }=\frac{n-4}{2}Q_{g}^{n}
\end{equation*}%
then%
\begin{equation*}
\widetilde{\beta }=\frac{n-4}{4\left( n-1\right) }\Delta R_{g}+\frac{\left[
n^{3}-4n^{2}+16\left( n-1\right) \right] (n-4)}{16\left( n-1\right)
^{2}\left( n-2\right) ^{2}}R_{g}^{2}-\frac{n-4}{\left( n-2\right) ^{2}}%
\left\vert Ric_{g}\right\vert ^{2}
\end{equation*}%
\begin{equation}
=\frac{n\left( n-4\right) }{8(n-1)}\left[ \frac{\left( n^{3}+4\right) \left(
n-4\right) }{8\left( n-1\right) \left( n-2\right) ^{2}}\left( -\Delta \log A+%
\frac{n}{2}\left\vert \nabla \log A\right\vert ^{2}\right) ^{2}\right.
\label{57}
\end{equation}%
\begin{equation*}
\left. -\left( -\Delta ^{2}\log A+\frac{n}{2}\Delta \left\vert \nabla \log
A\right\vert ^{2}\right) \right] \text{.}
\end{equation*}%
We deduce%
\begin{equation*}
\widetilde{\alpha }^{2}-4\widetilde{\beta }=\frac{3n^{3}-8n^{2}-18n+18}{%
4\left( n-1\right) ^{2}(n-2)^{2}}\left( -\Delta \log A+\frac{n}{2}\left\vert
\nabla \log A\right\vert ^{2}\right) ^{2}
\end{equation*}%
\begin{equation*}
+\frac{n\left( n-4\right) }{2(n-1)}\left( -\Delta ^{2}\log A+\frac{n}{2}%
\Delta \left\vert \nabla \log A\right\vert ^{2}\right) \text{.}
\end{equation*}%
If we let%
\begin{equation*}
a_{n}=\frac{3n^{3}-8n^{2}-18n+18}{4\left( n-1\right) ^{2}(n-2)^{2}}
\end{equation*}%
and 
\begin{equation*}
b_{n}=\frac{n\left( n-4\right) }{2(n-1)}
\end{equation*}%
we write%
\begin{equation}
\widetilde{\alpha }^{2}-4\widetilde{\beta }=a_{n}\left( -\Delta \log A+\frac{%
n}{2}\left\vert \nabla \log A\right\vert ^{2}\right) ^{2}+b_{n}\left(
-\Delta ^{2}\log A+\frac{n}{2}\Delta \left\vert \nabla \log A\right\vert
^{2}\right) \text{.}  \label{58}
\end{equation}%
Observe also that if 
\begin{equation*}
c_{n}=-\frac{n-2}{n-1}
\end{equation*}%
\begin{equation*}
\widetilde{\alpha }^{2}-4\widetilde{\beta }+2\Delta \widetilde{\alpha }%
=a_{n}\left( -\Delta \log A+\frac{n}{2}\left\vert \nabla \log A\right\vert
^{2}\right) ^{2}+c_{n}\left( -\Delta ^{2}\log A+\frac{n}{2}\Delta \left\vert
\nabla \log A\right\vert ^{2}\right)
\end{equation*}%
The radial laplacian writes as 
\begin{equation*}
\Delta =-\frac{1}{\rho ^{n-1}A^{\frac{n}{2}}}\frac{\partial }{\partial \rho }%
\left( A^{\frac{n}{2}-1}\rho ^{n-1}\frac{\partial }{\partial \rho }\right)
\end{equation*}%
\begin{equation*}
=-\frac{1}{A\rho ^{n-1}}\frac{\partial }{\partial \rho }\left( \rho ^{n-1}%
\frac{\partial }{\partial \rho }\right) -\frac{n-2}{2A}\frac{\partial }{%
\partial \rho }\log A\frac{\partial }{\partial \rho }
\end{equation*}%
\begin{equation*}
=-\frac{1}{A}\frac{\partial ^{2}}{\partial \rho ^{2}}-\frac{1}{A}\left( 
\frac{n-1}{\rho }+\frac{n-2}{2}\frac{\partial }{\partial \rho }\log A\right) 
\frac{\partial }{\partial \rho }
\end{equation*}%
we get%
\begin{equation*}
\Delta \log A=\frac{2-\sigma }{A}\left[ \left( n-\sigma \right) \rho
^{-\sigma }-\left( 2-\sigma \right) \frac{n-2}{2}\rho ^{2\left( 1-\sigma
\right) }\right]
\end{equation*}%
and%
\begin{equation*}
\left\vert \nabla \log A\right\vert ^{2}=\frac{\left( 2-\sigma \right) ^{2}}{%
A}\rho ^{2\left( 1-\sigma \right) }
\end{equation*}%
so%
\begin{equation*}
-\Delta \log A+\frac{n}{2}\left\vert \nabla \log A\right\vert ^{2}=\frac{%
2-\sigma }{A}\left[ -\left( n-\sigma \right) \rho ^{-\sigma }+\left(
n-1\right) \left( 2-\sigma \right) \rho ^{2\left( 1-\sigma \right) }\right]
\end{equation*}%
which shows that $\widetilde{\alpha }>0$ if $\rho $ is sufficiently small
and obviously if $0<\sigma <\frac{n}{s}<2$ then $\widetilde{\alpha }\in
L^{s}(M)$ with $s>\frac{n}{2}$ and $\widetilde{\alpha }\in H_{1}^{s}(M)$ if $%
0<\sigma <\frac{n}{s}-1<1$.

On the other hand, 
\begin{equation*}
-\frac{1}{A}\frac{\partial ^{2}}{\partial \rho ^{2}}\left( -\Delta \log A+%
\frac{n}{2}\left\vert \nabla \log A\right\vert ^{2}\right) =
\end{equation*}%
\begin{equation*}
=\frac{2-\sigma }{A^{2}}\left[ \left( n-\sigma \right) \sigma \left( \sigma
+1\right) \rho ^{\sigma -2}+\left( 2-\sigma \right) \left( n-1\right) \left(
1-\sigma \right) \left( \left( 4\sigma -1\right) n-\sigma \right) -\left(
1-\sigma \right) \left( 2-\sigma \right) ^{2}n\rho ^{2-\sigma }\right] \rho
^{-2\sigma }
\end{equation*}%
\begin{equation*}
=\frac{2-\sigma }{A^{2}}\left( n-\sigma \right) \sigma \left( \sigma
+1\right) \rho ^{-\sigma -2}\text{ }+o(1)\text{ as }\rho \rightarrow 0^{+}
\end{equation*}%
and%
\begin{equation*}
-\frac{1}{A}\left( \frac{n-1}{\rho }-\left( 2-\sigma \right) \frac{n}{2}\rho
^{1-\sigma }\right) \frac{\partial }{\partial \rho }\left( -\Delta \log A+%
\frac{n}{2}\left\vert \nabla \log A\right\vert ^{2}\right) =-\frac{2-\sigma 
}{A^{2}}\left( n-1-\left( 2-\sigma \right) \frac{n}{2}\rho ^{2-\sigma
}\right) \times
\end{equation*}%
\begin{equation*}
\left[ \left( 2-\sigma \right) ^{2}n\rho ^{2-3\sigma }+\left( 2-\sigma
\right) \left( n-n\sigma +\sigma \right) \rho ^{-2\sigma }+\left( n-\sigma
\right) \sigma \rho ^{-2-\sigma }\right]
\end{equation*}%
\begin{equation*}
=-\frac{2-\sigma }{A^{2}}\left( n-1\right) \left( n-\sigma \right) \sigma
\rho ^{-\sigma -2}\text{ }+o(1)\text{ as }\rho \rightarrow 0^{+}
\end{equation*}%
Hence%
\begin{equation*}
-\Delta ^{2}\log A=-\frac{2-\sigma }{A^{2}}\sigma \left( n-\sigma \right)
^{2}\rho ^{-\sigma -2}+o(1)\text{ as }\rho \rightarrow 0^{+}\text{.}
\end{equation*}%
Also, we have 
\begin{equation*}
\Delta \left\vert \nabla \log A\right\vert ^{2}=-\frac{\left( 2-\sigma
\right) ^{2}}{A^{2}}\left[ \left( \left( 2-\sigma \right) \left( 2-\sigma
\right) \rho ^{2-\sigma }+\left( 1-\sigma \right) \right) \rho ^{2-3\sigma
}+2\left( 1-\sigma \right) \left( 1-2\sigma \right) \rho ^{-2\sigma }\right.
\end{equation*}%
\begin{equation*}
\left. \left( \left( 2-\sigma \right) \rho ^{2-3\sigma }+2\left( 1-\sigma
\right) \rho ^{-2\sigma }\right) \left( n-1-\left( 2-\sigma \right) \frac{n}{%
2}\rho ^{2-\sigma }\right) \right]
\end{equation*}%
\begin{equation*}
=-\frac{\left( 2-\sigma \right) ^{2}}{A^{2}}\left[ 2\left( 1-\sigma \right)
\left( n-2\sigma \right) \rho ^{-2\sigma }+\left( 2-\sigma \right) \left(
1-\left( n-1\right) \left( 1-\sigma \right) \right) \rho ^{2-3\sigma }\right.
\end{equation*}%
\begin{equation*}
\left. +\left( 2-\sigma \right) ^{2}\rho ^{2-\sigma }\right] =-2\frac{\left(
2-\sigma \right) ^{2}}{A^{2}}\left( 1-\sigma \right) \left( n-2\sigma
\right) \rho ^{-2\sigma }+o(1)\text{ as }\rho \rightarrow 0^{+}
\end{equation*}%
and 
\begin{equation*}
-\Delta ^{2}\log A+\frac{n}{2}\Delta \left\vert \nabla \log A\right\vert
^{2}=-\frac{2-\sigma }{A^{2}}\sigma \left( n-\sigma \right) ^{2}\rho
^{-\sigma -2}+o(1)\text{ as }\rho \rightarrow 0^{+}
\end{equation*}%
hence 
\begin{equation}
-\Delta ^{2}\log A+\frac{n}{2}\Delta \left\vert \nabla \log A\right\vert
^{2}<0  \label{59}
\end{equation}%
for sufficiently small $\rho >0$.

Consequently%
\begin{equation}
\widetilde{\alpha }^{2}-4\widetilde{\beta }+2\Delta \widetilde{\alpha }>0
\label{60}
\end{equation}%
for sufficiently small $\rho >0$. Clearly $\widetilde{\beta }\in L^{p}(M)$
with $\ p>\frac{n}{4}$ provided that $0<\sigma <\frac{n}{p}-2<2$.

Now,we need the following lemma, already obtained by Madani in \cite{18} for
the Sobolev space $H_{2}^{p}\left( M\right) $ with $p>\frac{n}{2}$.

\begin{lemma}
\label{lem9} For $p>\frac{n}{4}$, $\ H_{4}^{p}(M)$ is an algebra i.e. for
any $\varphi $, $\psi \in H_{4}^{p}(M)$ we have $\varphi \psi \in
H_{4}^{p}(M)$.
\end{lemma}

\begin{proof}
It suffices to show that the fourth order covariant derivative $\nabla
^{4}\left( \varphi \psi \right) \in L^{p}(M)$.%
\begin{equation*}
\nabla ^{4}\left( \varphi \psi \right) =\varphi \nabla ^{4}\psi +\psi \nabla
^{4}\varphi +3\left( \nabla ^{2}\varphi \otimes \nabla ^{2}\psi +\nabla
^{2}\psi \otimes \nabla ^{2}\varphi \right)
\end{equation*}%
\begin{equation}
+3\left( \nabla \psi \otimes \nabla ^{3}\varphi +\nabla \varphi \otimes
\nabla ^{3}\psi \right) +\nabla ^{3}\psi \otimes \nabla \varphi +\nabla
^{3}\varphi \otimes \nabla \psi \text{.}  \label{610}
\end{equation}%
Since $p>\frac{n}{4}$, by Theorem \ref{th5}, $\varphi $ and $\psi $ are
continuous functions on $M$, then bounded and by the continuous of the
Sobolev embedding $H_{4-i}^{p}\left( M\right) \hookrightarrow L^{q_{i}}(M)$
with $q_{i}\leq \frac{pn}{n-\left( 4-i\right) p}$ and $i=0$,$1$,$2$,$3$, we
get%
\begin{equation*}
\left\Vert \nabla ^{2}\varphi \otimes \nabla ^{2}\psi \right\Vert _{p}\leq
\left\Vert \nabla ^{2}\varphi \right\Vert _{2p}\left\Vert \nabla ^{2}\psi
\right\Vert _{2p}<+\infty
\end{equation*}%
also, for any $\theta $ such that $1<\theta <\frac{n}{n-p}$ 
\begin{equation*}
\left\Vert \nabla \psi \otimes \nabla ^{3}\varphi \right\Vert _{p}\leq
\left\Vert \nabla ^{3}\varphi \right\Vert _{p\theta }\left\Vert \nabla \psi
\right\Vert _{p\left( 1-\frac{1}{\theta }\right) }<+\infty \text{.}
\end{equation*}%
The same is also true for the other terms of \ \ref{610}.
\end{proof}

The Paneitz-Branson operator $P$ expresses as

\begin{equation*}
P(u)=\Delta _{g}^{2}u-div_{g}\left( \widetilde{\alpha }du\right) +\widetilde{%
\beta }u\text{.}
\end{equation*}%
Given a smooth positive function $f$ on $M$, the problem is to find a metric 
$\widetilde{g}$ in the Sobolev space $H_{4}^{p}\left( M,T^{\ast }M\otimes
T^{\ast }M\right) $, with $p>\frac{n}{4}$, conformal to the metric $g$, of $%
Q $-curvature is $f$.

If $\ \widetilde{g}=u^{\frac{4}{n-4}}g$, $u>0$. $u$ will be a weak solution
in $H_{4}^{p}(M)$ of the following equation

\begin{equation}
\Delta _{g}^{2}u-div_{g}\left( \widetilde{\alpha }du\right) +\widetilde{%
\beta }u=fu^{N-1}  \label{61}
\end{equation}%
with $N=\frac{2n}{n-4}$.

For any $u\in H_{2}(M)$, we let 
\begin{equation*}
J(u)=\int_{M}\left( \Delta _{g}u\right) ^{2}dv_{g}+\int_{M}\widetilde{\alpha 
}\left\vert \nabla u\right\vert _{g}^{2}dv_{g}+\int_{M}\widetilde{\beta }%
u^{2}dv_{g}
\end{equation*}%
\begin{equation*}
B=\left\{ u\in H_{2}(M):\int_{M}fu^{N}dv_{g}=\left( 1+\left\Vert \widetilde{%
\alpha }\right\Vert _{s}+\left\Vert \widetilde{\beta }\right\Vert
_{p}\right) ^{\frac{N}{2}}\right\}
\end{equation*}%
obviously $B\neq \phi $.

\begin{theorem}
Let $\left( M^{n},h\right) $ be a compact flat smooth $n$-manifold with $n>6$
and let $g=Ah$ where $A=e^{-\rho ^{2-\sigma }}$, $0<\sigma <\inf \left( 
\frac{n}{s}-1,\frac{n}{p}-2\right) $,with $s>\frac{n}{2}$, $p>\frac{n}{4}$,
and the function $\rho $ defined by (\ref{1}) is supposed sufficiently
small. Let $f$ be a $C^{\infty }$ positive function on $M$ such that the
maximum of the function $f$ is attained at a point $R$ $\in M$ where the
function $\rho (R)\neq 0$ and 
\begin{equation*}
Q(M)<\left( \sup_{x\in M}f(x)\right) ^{-\frac{1}{N}}K(n,2)^{-2}\left(
1+\left\Vert \widetilde{\alpha }\right\Vert _{s}+\left\Vert \widetilde{\beta 
}\right\Vert _{p}\right) \text{.}
\end{equation*}%
Then there exists a metric $\widetilde{g}\in H_{4}^{p}\left( M,T^{\ast
}M\otimes T^{\ast }M\right) $ conformal to $g$ such that $f$ is the $Q$%
-curvature of the manifold $\left( M,\widetilde{g}\right) $.
\end{theorem}

\begin{proof}
The existence of the metric $\widetilde{g}$ conformal to $g$ reduces to the
existence of a positive solution to equation (\ref{61}). Put 
\begin{equation*}
Q(M)=\inf_{u\in H_{2}(M)-\left\{ 0\right\} }Q(u)=\inf_{u\in B}J(u)\text{.}
\end{equation*}%
Let $u$ be the solution previously constructed in section (4) of the
equation (\ref{61}), we have already shown in section (4) that $u\in
H_{4}^{q}\left( M\right) $ with $4q>n$. To have a weak positive solution to
equation (\ref{61}), we look for a positive solution $v$ to the following
equation 
\begin{equation}
\Delta v+\frac{\widetilde{\alpha }}{2}v=\left\vert \Delta u+\frac{\widetilde{%
\alpha }}{2}u\right\vert  \label{60'}
\end{equation}%
where $\widetilde{\alpha }$ is given by (\ref{56}). Since $\widetilde{\alpha 
}>0$ the operator $I\left( \varphi \right) =\int_{M}\left( \left\vert \nabla
\varphi \right\vert ^{2}+\frac{\widetilde{\alpha }}{2}\varphi ^{2}\right)
dv_{g}$ is coercive on $H_{1}^{2}\left( M\right) $: since if it is not the
case there exists a sequence $\left( \varphi _{m}\right) _{m\in N^{\ast }}$
such that $\left\Vert \varphi _{m}\right\Vert _{\frac{2s}{s-1}}=1$ with $s>%
\frac{n}{2}$ and 
\begin{equation}
I\left( \varphi _{m}\right) <\frac{1}{m}\left\Vert \varphi _{m}\right\Vert
_{H_{1}^{2}\left( M\right) }  \label{61'}
\end{equation}%
taking account of the fact that $\widetilde{\alpha }\in L^{s}\left( M\right) 
$ we get 
\begin{equation*}
\left( 1-\frac{1}{m}\right) \left\Vert \nabla \varphi _{m}\right\Vert
_{2}^{2}\leq \frac{1}{m}\max \left( 1,V(M)\right) +\left\Vert \frac{%
\widetilde{\alpha }}{2}\right\Vert _{s}
\end{equation*}%
so for any $m\geq 2$, \ $\left\Vert \nabla \varphi _{m}\right\Vert
_{2}<\infty $ and the condition $\left\Vert \varphi _{m}\right\Vert _{\frac{%
2s}{s-1}}=1$ implies that $\left\Vert \varphi _{m}\right\Vert _{2}<\infty $.
Consequently the sequence $\left( \varphi _{m}\right) _{m\geq 2}$ is bounded
in $H_{1}^{2}\left( M\right) $. Hence up to a subsequence

$\varphi _{m}\rightarrow \varphi $ weakly in $H_{1}^{2}\left( M\right) $

$\varphi _{m}\rightarrow \varphi $ strongly in $L^{q}\left( M\right) $ with $%
q<\frac{2n}{n-2}=2^{\ast }$

and $\left\Vert \varphi \right\Vert _{H_{1}^{2}(M)}\leq \lim \inf \left\Vert
\varphi _{m}\right\Vert _{H_{1}^{2}(M)}$.

We deduce that%
\begin{equation}
\left\Vert \nabla \varphi \right\Vert _{2}\leq \lim \inf \left\Vert \nabla
\varphi _{m}\right\Vert _{2}\text{.}  \label{61''}
\end{equation}%
Now since $\frac{2s}{s-1}<2^{\ast }$ and $\widetilde{\alpha }\in L^{s}\left(
M\right) $ we get that 
\begin{equation}
\int_{M}\frac{\widetilde{\alpha }}{2}\varphi _{m}^{2}dv_{g}\rightarrow
\int_{M}\frac{\widetilde{\alpha }}{2}\varphi ^{2}dv_{g}  \label{61'''}
\end{equation}%
By relations (\ref{61'}), (\ref{61''}) , (\ref{61'''}) and the fact that $%
\left( \varphi _{m}\right) _{m}$ is bounded in $H_{1}^{2}\left( M\right) $
and $\widetilde{\alpha }>0$ we infer that%
\begin{equation*}
0\leq I\left( \varphi \right) \leq \lim_{m}I(\varphi _{m})=0
\end{equation*}%
i.e. $\left\Vert \nabla \varphi \right\Vert _{2}=0$ and $\int_{M}\frac{%
\widetilde{\alpha }}{2}\varphi ^{2}dv_{g}=0$; the first relation implies
that $\varphi $ is constant a.e. in $M$ and the second one implies $%
\widetilde{\alpha }=0$ a.e. in $M$ which contradicts the definition of $%
\widetilde{\alpha }.$ Hence the Operator $I\left( \varphi \right) $ is
coercive on $H_{1}^{2}\left( M\right) $. Hence by ( \cite{18}, Prop. 6 ) and
the fact that $u$ is a non trivial solution of equation (\ref{61}) in $%
H_{4}^{p}\left( M\right) $ we infer that equation (\ref{60}') has a
nontrivial and nonnegative solution $v\in C^{0,\beta }(M)$ with $\beta \in
\left( 0,1\right) $.

we have to show that $v>0$ on $M$; to do so we need a positive Green's
function to the operator $L=\Delta +\frac{\widetilde{\alpha }}{2}$. For
smooth $\widetilde{\alpha }$, $L$ admits a Green's function. In the case $%
\widetilde{\alpha }\in L^{p}(M)$, F. Madani \cite{18} deduces the existence
of Green's function to $L$ from the smooth case by exploiting the weak
conformal invariant of $L$ ( see \cite{18} ). To study the positivity of the
Green's function of $L$, we need its explicit expression so we give a direct
construction of this latter function. We follow Aubin's construction ( see 
\cite{1} ).

Let $G(x,.)$ denote the Green's function of the Laplacian $\Delta $ i.e. the
solution in distribution sense of the equation $\Delta G(x,.)=\delta _{x}-%
\frac{1}{V(M)}$. The Green's function $G\left( x,.\right) $ is determined up
to a constant in this case and the following estimation holds $\left\vert
G\left( x,y\right) \right\vert \leq Cd(x,y)^{2-n}$ for $x\neq y$ , where $C$
is a constant. Let $\Gamma _{1}(x,.)=-LG(x,.)$ where $L=\Delta +\frac{%
\widetilde{\alpha }}{2}$ and for $j\in N^{\ast }$, $\Gamma _{j+1}\left(
x,y\right) =\int_{M}\Gamma _{j}\left( x,z\right) \frac{\widetilde{\alpha }%
\left( z\right) }{2}G(z,y)dv_{g}\left( z\right) $.The functions $\Gamma _{i}$
are well defined. By Giraud's lemma recurrently we obtain that for any $%
y\neq x$ \ 
\begin{equation*}
\left\vert \Gamma _{j}(x,y)\right\vert \leq \left\{ 
\begin{array}{c}
C_{j}\left\Vert \frac{\widetilde{\alpha }}{2}\right\Vert
_{p}d(x,y)^{2j-n\left( 1+\frac{1}{p}\right) }\text{ \ if }j<\frac{n}{2}%
\left( 1+\frac{1}{p}\right) \\ 
C_{j}\left\Vert \frac{\widetilde{\alpha }}{2}\right\Vert _{p}\left(
1+\left\vert \log d(x,y)\right\vert \right) \text{ if }j=\frac{n}{2}\left( 1+%
\frac{1}{p}\right) \\ 
C_{j}\left\Vert \frac{\widetilde{\alpha }}{2}\right\Vert _{p}\text{ \ \ \ \
\ \ \ \ \ \ \ \ \ \ \ \ \ \ \ \ \ \ \ \ \ \ \ \ if\ \ }j>\frac{n}{2}\left( 1+%
\frac{1}{p}\right)%
\end{array}%
\right.
\end{equation*}%
and since $p>\frac{n}{2}$, it follows that $\Gamma _{j}(x,.)\in L^{1}\left(
M\right) $, for all $j\in N^{\ast }$.

Consider now the following function $H(x,.)=G(x,.)+\sum_{j=1}^{k}\Gamma
_{j+1}\left( x,.\right) +u_{x}$ where $u_{x}\in H_{1}^{2}\left( M\right) $;

$H(x,.)$ will be a Green's function to $L$ i.e. $L_{y}H(x,y)=0$ for any $%
y\neq x$ if $u_{x}$ is a solution in distribution sense to the following
equation%
\begin{equation}
\Delta u_{x}+\frac{\widetilde{\alpha }}{2}u_{x}=-\Gamma _{k}\left(
x,.\right) \text{.}  \label{62'}
\end{equation}%
First, since $p>\frac{n}{2}$ we remark that $\Gamma _{k}\left( x,.\right)
\in L^{p}(M)$ and the same arguments as for equation (\ref{60'}) show that $%
u_{x}$ is a non trivial and non negative continuous solution of equation (%
\ref{62'}). Hence $u_{x}$ is bounded on $M$, consequently \ $H(x,y$) is
bounded and as $G(x,.)$ is determined up to a constant ( see \cite{1}), the
same is true for $H(x,.)$ and we choose the constant such that $H(x,.)$ is
positive. Now the solution $v$ to equation (\ref{60'}) is represented by%
\begin{equation*}
v(x)=\int_{M}\left\vert \Delta u\left( y\right) +\frac{\widetilde{\alpha }}{2%
}u\left( y\right) \right\vert H\left( x,y\right) dv_{g}\left( y\right) \text{%
.}
\end{equation*}%
Hence $v>0$, since if it is not the case and since $H\left( x,y\right) >0$
for $x\neq y$ we obtain that $\Delta u+\frac{\widetilde{\alpha }}{2}u=0$.
Now since the operator $L$ is invertible we get that $u$ is a trivial
solution to equation (\ref{61}) which contradicts the fact that $u$ is not
identically $0$, proved in section (4).

Now we put%
\begin{equation*}
w=v\pm u
\end{equation*}%
so 
\begin{equation}
\Delta w+\frac{\widetilde{\alpha }}{2}w=\left\vert \Delta u+\frac{\widetilde{%
\alpha }}{2}u\right\vert \pm \Delta u+\frac{\widetilde{\alpha }}{2}u\geq 0
\label{62}
\end{equation}%
and the same arguments as for the equation (\ref{60'}) lead to $w\geq 0$ and
we infer that%
\begin{equation*}
v\geq \left\vert u\right\vert
\end{equation*}

On the hand, we have 
\begin{equation*}
\int_{M}fv^{N}dv_{g}\geqslant \int_{M}fu^{N}dv_{g}=1+\left\Vert \widetilde{%
\alpha }\right\Vert _{s}+\left\Vert \widetilde{\beta }\right\Vert _{p}\text{,%
}
\end{equation*}%
we let $0<k<1$, such that 
\begin{equation*}
\int_{M}f\left( kv\right) ^{N}dv_{g}=1+\left\Vert \widetilde{\alpha }%
\right\Vert _{s}+\left\Vert \widetilde{\beta }\right\Vert _{p}
\end{equation*}%
and put $\widehat{v}=kv$, since $v>0$, then \ $\widehat{v}>0$ and satisfies $%
\int_{M}f\widehat{v}^{N}dv_{g}=1+\left\Vert \widetilde{\alpha }\right\Vert
_{s}+\left\Vert \widetilde{\beta }\right\Vert _{p}$ so $\widehat{v}\in B$.

Independently, we have%
\begin{equation*}
\int_{M}\left( \left( \Delta _{g}v\right) ^{2}+\widetilde{\alpha }\left\vert
\triangledown v\right\vert _{g}^{2}+\frac{\widetilde{\alpha }^{2}}{4}%
v^{2}\right) dv_{g}=\int_{M}\left( \left( \Delta _{g}u\right) ^{2}+%
\widetilde{\alpha }\left\vert \triangledown u\right\vert _{g}^{2}+\frac{%
\widetilde{\alpha }^{2}}{4}u^{2}\right) dv_{g}
\end{equation*}%
\begin{equation*}
+\frac{1}{2}\int_{M}\left( u^{2}-v^{2}\right) \Delta _{g}\widetilde{\alpha }%
dv_{g}
\end{equation*}%
and evaluating%
\begin{equation*}
S=\int_{M}\left( \left( \Delta _{g}\widehat{v}\right) ^{2}+\widetilde{\alpha 
}\left\vert \triangledown \widehat{v}\right\vert _{g}^{2}+\widetilde{\beta }%
\widehat{v}^{2}\right) dv_{g}-Q\left( M\right)
\end{equation*}%
\begin{equation*}
=k^{2}\int_{M}\left( \left( \Delta _{g}u\right) ^{2}+\widetilde{\alpha }%
\left\vert \triangledown u\right\vert _{g}^{2}+\widetilde{\beta }%
u^{2}\right) dv_{g}+k^{2}\int_{M}(\widetilde{\beta }-\frac{\widetilde{\alpha 
}^{2}}{4})\left( v^{2}-u^{2}\right) dv_{g}
\end{equation*}%
\begin{equation*}
+\frac{1}{2}k^{2}\int_{M}\left( u^{2}-v^{2}\right) \Delta _{g}\widetilde{%
\alpha }dv_{g}-Q(M)
\end{equation*}%
\begin{equation*}
=\left( k^{2}-1\right) Q(M)+k^{2}\int_{M}(\widetilde{\beta }-\frac{%
\widetilde{\alpha }^{2}}{4}-2\Delta \widetilde{\alpha })\left(
v^{2}-u^{2}\right) dv_{g}.
\end{equation*}%
If $\rho $\ is sufficiently small then by (\ref{60}), we get 
\begin{equation*}
\widetilde{\beta }-\frac{\widetilde{\alpha }^{2}}{4}-2\Delta \widetilde{%
\alpha }\leq 0\text{.}
\end{equation*}%
Since $k^{2}-1<0$, $Q(M)\geq 0$, we infer that 
\begin{equation*}
S\leq 0\text{.}
\end{equation*}%
Hence 
\begin{equation*}
Q\left( M\right) \geq \int_{M}\left( \left( \Delta _{g}\widehat{v}\right)
^{2}+\widetilde{\alpha }\left\vert \triangledown \widehat{v}\right\vert
_{g}^{2}+\widetilde{\beta }\widehat{v}^{2}\right) dv_{g}
\end{equation*}%
and by the definition of $Q(M)$, we deduce that 
\begin{equation*}
Q\left( M\right) =\int_{M}\left( \left( \Delta _{g}\widehat{v}\right) ^{2}+%
\widetilde{\alpha }\left\vert \triangledown \widehat{v}\right\vert _{g}^{2}+%
\widetilde{\beta }\widehat{v}^{2}\right) dv_{g}.
\end{equation*}%
Consequently the infimum $Q(M)$ over $B$ is attained by the positive
function $\widehat{v}\in B$. If we write the Euler-Lagrange equation for $%
\widehat{v}$, we find that $\widehat{v}$ is a weak positive solution in $%
H_{2}\left( M\right) $ to equation (\ref{61}). Similar arguments as in the
proof of Theorem \ref{th6} lead to conclude that $\widehat{v}\in
H_{4}^{p}(M) $. Now since the function $A=\rho ^{2-\sigma }$, with $0<\sigma
<\frac{n}{p}-2<2$ belongs to $H_{4}^{p}(M)$, it follows by Lemma \ref{lem9}
that $A\widehat{v}\in H_{4}^{p}(M)$. Consequently the metric $\widetilde{g}=%
\widehat{v}g=\left( A\widehat{v}\right) h\in H_{4}^{p}\left( M,T^{\ast
}M\otimes T^{\ast }M\right) $
\end{proof}

\section{Proof of theorems \protect\ref{th01} and \protect\ref{th02}}

Let $P\in M$ such that $f(P)$ is the maximum of $f$ on $M$ and the metric is
of class $C^{\infty }$ on the ball $B_{2\varrho }(P)$ where $0<2\varrho
<\delta $ and $\delta $ is the injectivity radius.

Consider the function 
\begin{equation*}
\varphi _{\epsilon }(r)=\frac{\eta (r)}{\left( r^{2}+\epsilon ^{2}\right) ^{%
\frac{n-4}{2}}}
\end{equation*}%
where $\eta (r)$ is a $C^{\infty }$ function on $M$ given by%
\begin{equation*}
\eta (r)=\left\{ 
\begin{array}{c}
1\text{ \ \ \ on \ \ \ }B_{\varrho }\left( P\right) \\ 
0\text{ \ on \ }M-B_{2\varrho }\left( P\right)%
\end{array}%
\text{ \ \ \ }\right. \text{.}
\end{equation*}%
For $n>6$, by H\"{o}lder's inequality, we get

\begin{equation*}
B=\int_{B_{\varrho }\left( P\right) }a(x)\left\vert \nabla \varphi
_{\epsilon }(r)\right\vert ^{2}dv_{g}\leq \left( \int_{B_{\varrho }\left(
P\right) }\left\vert a(x)\right\vert ^{s}dv_{g}\right) ^{\frac{1}{s}}\left(
\int_{B_{\varrho }\left( P\right) }\left\vert \nabla \varphi _{\epsilon
}(r)\right\vert ^{\frac{2s}{s-1}}dv_{g}\right) ^{1-\frac{1}{s}}
\end{equation*}%
and taking account of%
\begin{equation*}
\left\vert \nabla \varphi _{\epsilon }(r)\right\vert =(n-4)\frac{r}{\left(
r^{2}+\epsilon ^{2}\right) ^{\frac{n-2}{2}}}
\end{equation*}%
on $B_{\epsilon }(P)$ and 
\begin{equation*}
dv_{g}=(1-\frac{1}{6}R_{ij}x^{i}x^{j})+o(r^{2})
\end{equation*}%
where we have written $R_{g}$ instead of $R_{g}(P)$ we get%
\begin{equation*}
B^{\prime }=\int_{B_{\varrho }(P)}\left\vert \nabla \varphi _{\epsilon
}(r)\right\vert ^{\frac{2s}{s-1}}dv_{g}=\left( n-4\right) ^{\frac{2s}{s-1}%
}\omega _{n-1}\int_{0}^{\varrho }(1-\frac{R_{g}}{6n}r^{2})\frac{r^{\frac{2s}{%
s-1}+n-1}}{\left( r^{2}+\epsilon ^{2}\right) ^{\left( n-2\right) \frac{s}{s-1%
}}}dr+o\left( \epsilon ^{2}\right)
\end{equation*}%
if we put 
\begin{equation*}
t=\left( \frac{r}{\epsilon }\right) ^{2}
\end{equation*}%
\begin{equation}
B^{\prime }=\frac{1}{2}\epsilon ^{-n+4+\frac{2s}{s-1}}\left( n-4\right) ^{%
\frac{2s}{s-1}}\omega _{n-1}\int_{0}^{\frac{\varrho ^{2}}{\epsilon ^{2}}}(1-%
\frac{R_{g}}{6n}\epsilon ^{2}t)\frac{t^{\frac{s}{s-1}+\frac{n-2}{2}}}{\left(
1+t\right) ^{\left( n-2\right) \frac{s}{s-1}}}dt+o\left( \epsilon ^{2}\right)
\label{63}
\end{equation}%
and letting for any real numbers $p$, $q$ with $p-q>1$,%
\begin{equation*}
I_{p}^{q}=\int_{0}^{+\infty }\frac{t^{q}}{\left( 1+t\right) ^{p}}dt
\end{equation*}%
and if 
\begin{equation}
\left( n-8\right) s+n+2>0  \label{63''''}
\end{equation}%
which is fulfilled if $n\geq 8$ and for $n=7$ we must have $\frac{7}{2}<s<$ $%
9$, we get%
\begin{equation*}
B^{\prime }=\frac{1}{2}\epsilon ^{\left( -n+3\right) \frac{2s}{s-1}+n}\left(
n-4\right) ^{\frac{2s}{s-1}}\omega _{n-1}\left( I_{\left( n-2\right) \frac{s%
}{s-1}}^{\frac{s}{s-1}+\frac{n-2}{2}}-\epsilon ^{2}\frac{R_{g}}{6n}I_{\left(
n-2\right) \frac{s}{s-1}}^{\frac{s}{s-1}+\frac{n-2}{2}+1}\right) +o(\epsilon
^{2})
\end{equation*}%
and taking account of%
\begin{equation}
I_{p}^{q+1}=\frac{q+1}{p-q-2}I_{p}^{q}\text{ if \ }p-q-2>0  \label{63''}
\end{equation}%
we deduce that 
\begin{equation*}
B^{\prime 1-\frac{1}{s}}=\left( \frac{1}{2}\right) ^{1-\frac{1}{s}}\epsilon
^{-n+4+2-\frac{n}{s}}\left( n-4\right) ^{2}\omega _{n-1}^{1-\frac{1}{s}%
}\left( I_{\left( n-2\right) \frac{s}{s-1}}^{\frac{s}{s-1}+\frac{n}{2}%
-1}\right) ^{1-\frac{1}{s}}
\end{equation*}%
\begin{equation*}
\times \left( 1-\epsilon ^{2}\frac{R_{g}}{6n}\frac{\frac{n}{2}+\frac{s}{s-1}%
}{\left( n-3\right) \frac{s}{s-1}-\frac{n}{2}-1}+o\left( \epsilon
^{2}\right) \right) ^{1-\frac{1}{s}}
\end{equation*}%
and since it is not difficult to show that 
\begin{equation*}
\lim_{\epsilon \rightarrow 0^{+}}\int_{B_{2\varrho }\left( P\right)
-B_{\varrho }\left( P\right) }\left\vert \nabla \varphi _{\epsilon
}(r)\right\vert ^{\frac{2s}{s-1}}dv_{g}=0
\end{equation*}%
we infer that 
\begin{equation*}
B=\int_{M}a(x)\left\vert \nabla \varphi _{\epsilon }(r)\right\vert
^{2}dv_{g}\leq \left( \frac{1}{2}\right) ^{1-\frac{1}{s}}\left\Vert
a\right\Vert _{s}\epsilon ^{-n+4+2-\frac{n}{s}}\left( n-4\right) ^{2}\omega
_{n-1}^{1-\frac{1}{s}}\left( I_{\left( n-2\right) \frac{s}{s-1}}^{\frac{s}{%
s-1}+\frac{n}{2}-1}\right) ^{1-\frac{1}{s}}\times
\end{equation*}%
\begin{equation*}
\left( 1-\epsilon ^{2}\frac{R_{g}}{6n}\frac{s-1}{s}\frac{\frac{n}{2}+\frac{s%
}{s-1}}{\left( n-3\right) \frac{s}{s-1}-\frac{n}{2}-1}+o\left( \epsilon
^{2}\right) \right) \text{.}
\end{equation*}%
H\"{o}lder's inequality leads to 
\begin{equation*}
C=\int_{M}b(x)\left( \varphi _{\epsilon }(r)\right) ^{2}dv_{g}
\end{equation*}%
\begin{equation*}
\leq \left\Vert b\right\Vert _{p}\left( \int_{M}\varphi _{\epsilon }(r)^{%
\frac{2p}{p-1}}dv_{g}\right) ^{1-\frac{1}{p}}
\end{equation*}%
\begin{equation*}
=\left\Vert b\right\Vert _{p}\left( \int_{B_{\varrho }\left( P\right)
}\varphi _{\epsilon }(r)^{\frac{2p}{p-1}}dv_{g}+\left( \int_{B_{2\varrho
}\left( P\right) -B_{\varrho }\left( P\right) }\varphi _{\epsilon }(r)^{%
\frac{2p}{p-1}}dv_{g}\right) \right) ^{1-\frac{1}{p}}
\end{equation*}%
and 
\begin{equation*}
\int_{B_{\varrho }\left( P\right) }\varphi _{\epsilon }(r)^{\frac{2p}{p-1}%
}dv_{g}
\end{equation*}%
\begin{equation*}
=\omega _{n-1}\int_{0}^{\varrho }\frac{r^{n-1}}{\left( \epsilon
^{2}+r^{2}\right) ^{\left( n-4\right) \frac{p}{p-1}}}\left( 1-\frac{R_{g}}{6n%
}r^{2}+o\left( r^{2}\right) \right) dr
\end{equation*}%
hence if 
\begin{equation}
\left( n-4\right) \frac{p}{p-1}-\frac{n}{2}-1>0  \label{63'}
\end{equation}%
which is fulfilled in case $n\geq 10$, for $n=7$, $8$, $9$ \ to have (\ref%
{63'}) \ satisfied we must have respectively $\frac{7}{4}<p<3$, $2<p<5$, $%
\frac{9}{4}<p<11$. 
\begin{equation*}
\int_{B_{\varrho }\left( P\right) }\varphi _{\epsilon }(r)^{\frac{2p}{p-1}%
}dv_{g}=\frac{1}{2}\epsilon ^{-2\left( n-4\right) \frac{p}{p-1}+n}\omega
_{n-1}\left( I_{\left( n-4\right) \frac{p}{p-1}}^{\frac{n}{2}-1}-\frac{R_{g}%
}{6n}\epsilon ^{2}I_{\left( n-4\right) \frac{p}{p-1}}^{\frac{n}{2}}+o\left(
\epsilon ^{2}\right) \right)
\end{equation*}%
and by relation (\ref{63''}) we get 
\begin{equation*}
I_{\left( n-4\right) \frac{p}{p-1}}^{\frac{n}{2}}=\frac{n(p-1)}{pn-10p+n+2}%
I_{\left( n-4\right) \frac{p}{p-1}}^{\frac{n}{2}-1}\text{.}
\end{equation*}%
Hence we infer that%
\begin{equation*}
\left( \int_{B_{\varrho }\left( P\right) }\varphi _{\epsilon }(r)^{\frac{2p}{%
p-1}}dv_{g}\right) ^{1-\frac{1}{p}}=\left( \frac{1}{2}\right) ^{1-\frac{1}{p}%
}\epsilon ^{-n+4+4-\frac{n}{p}}\omega _{n-1}^{1-\frac{1}{p}}\left( I_{\left(
n-4\right) \frac{p}{p-1}}^{\frac{n}{2}-1}\right) ^{1-\frac{1}{p}}\times
\end{equation*}%
\begin{equation*}
\left( 1-\frac{\left( p-1\right) ^{2}}{6p\left( pn-10p+n+2\right) }%
R_{g}\epsilon ^{2}+o\left( \epsilon ^{2}\right) \right)
\end{equation*}%
and since it is easy to show that $\int_{B_{2\varrho }\left( P\right)
-B_{\varrho }\left( P\right) }\varphi _{\epsilon }(r)^{\frac{2p}{p-1}%
}dv_{g}\rightarrow 0$ as $\epsilon \rightarrow 0^{+}$, we deduce that 
\begin{equation*}
C\leq \left( \frac{1}{2}\right) ^{1-\frac{1}{p}}\epsilon ^{-n+4+\left( 4-%
\frac{n}{p}\right) }\omega _{n-1}^{1-\frac{1}{p}}\left\Vert b\right\Vert
_{p}\left( I_{\left( n-4\right) \frac{p}{p-1}}^{\frac{n}{2}-1}\right) ^{1-%
\frac{1}{p}}\left( 1-\frac{\left( p-1\right) ^{2}}{6p\left(
pn-10p+n+2\right) }R_{g}\epsilon ^{2}+o\left( \epsilon ^{2}\right) \right) 
\text{.}
\end{equation*}%
The expansion of $\int_{B_{\delta }\left( P\right) }f(x)\varphi _{\epsilon
}^{N}\left( r\right) dv_{g}$ has been computed in (\cite{7}) and is given by 
\begin{equation*}
\int_{B_{\delta }\left( P\right) }f(x)\varphi _{\epsilon }^{N}\left(
r\right) dv_{g}=\frac{\omega _{n-1}I_{n}^{\frac{n}{2}-1}}{2}\epsilon
^{-n}\left( f(P)-\frac{\epsilon ^{2}}{n-2}\left( \frac{\Delta f}{2}+\frac{%
f(P)R_{g}}{6}\right) +o\left( \epsilon ^{2}\right) \right)
\end{equation*}%
hence%
\begin{equation*}
\left( \int_{B_{\delta }\left( P\right) }f(x)\varphi _{\epsilon }^{N}\left(
r\right) dv_{g}\right) ^{-\frac{2}{N}}=
\end{equation*}%
\begin{equation*}
=\left( \frac{\omega _{n-1}I_{n}^{\frac{n}{2}-1}}{2}\epsilon
^{-n}f(P)\right) ^{-\frac{n-4}{n}}\left( 1+\epsilon ^{2}\frac{n-4}{n(n-2)}%
\left[ \frac{\Delta f\left( P\right) }{2f(P)}+\frac{R_{g}}{6}\right]
+o\left( \epsilon ^{2}\right) \right) \text{.}
\end{equation*}%
\begin{equation}
=\frac{2^{\frac{n-4}{n}}\epsilon ^{n-4}}{\left( I_{n}^{\frac{n}{2}-1}\omega
_{n-1}f(P)\right) ^{\frac{n-4}{n}}}\left( 1+\epsilon ^{2}\frac{n-4}{n\left(
n-2\right) }\left( \frac{\Delta f\left( P\right) }{2f(P)}+\frac{R_{g}}{6}%
\right) +o\left( \epsilon ^{2}\right) \right) \text{.}  \label{63'''}
\end{equation}%
Also, we have%
\begin{equation*}
-\Delta \varphi _{\epsilon }=\frac{1}{r^{n-1}}\frac{\partial }{\partial r}%
\left( r^{n-1}\frac{\partial \varphi _{\epsilon }}{\partial r}\right) +\frac{%
\partial \varphi _{\epsilon }}{\partial r}\frac{\partial }{\partial r}\log 
\sqrt{\left\vert g\right\vert }
\end{equation*}%
where $\left\vert g\right\vert =\det \left( g\right) $. The following
calculations are the same as in \cite{4} so we recall it briefly. On the
ball $B_{\delta }\left( P\right) $, we have%
\begin{equation*}
\left( \Delta \varphi _{\epsilon }(r)\right) ^{2}=\left( n-4\right)
^{2}\left( \frac{\left( 2r^{2}+n\epsilon ^{2}\right) ^{2}}{\left( \epsilon
^{2}+r^{2}\right) ^{n}}+\frac{\left( r\frac{\partial }{\partial r}\log \sqrt{%
\left\vert g\right\vert }\right) ^{2}}{\left( \epsilon ^{2}+r^{2}\right)
^{n-2}}+2r\frac{2r^{2}+n\epsilon ^{2}}{\left( \epsilon ^{2}+r^{2}\right)
^{n-1}}\frac{\partial }{\partial r}\log \sqrt{\left\vert g\right\vert }%
\right)
\end{equation*}%
so%
\begin{equation*}
\int_{B_{\delta }\left( P\right) }\left( \Delta \varphi _{\epsilon }\right)
^{2}dv_{g}=\left( n-4\right) ^{2}\omega _{n-1}\int_{0}^{\varrho }\left( 
\frac{\left( 2r^{2}+n\epsilon ^{2}\right) ^{2}}{\left( \epsilon
^{2}+r^{2}\right) ^{n}}+\frac{\left( r\frac{\partial }{\partial r}\log \sqrt{%
\left\vert g\right\vert }\right) ^{2}}{\left( \epsilon ^{2}+r^{2}\right)
^{n-2}}+2r\frac{2r^{2}+n\epsilon ^{2}}{\left( \epsilon ^{2}+r^{2}\right)
^{n-1}}\frac{\partial }{\partial r}\log \sqrt{\left\vert g\right\vert }%
\right)
\end{equation*}%
\begin{equation*}
\times \left( 1-\frac{R_{g}}{6n}r^{2}+o(r^{2})\right) r^{n-1}dr
\end{equation*}%
and taking account of $\frac{\partial }{\partial r}\log \sqrt{\left\vert
g\right\vert }=-\frac{R_{g}}{3n}r+o(r)$, we get 
\begin{equation*}
\int_{0}^{\varrho }\frac{\left( 2r^{2}+n\epsilon ^{2}\right) ^{2}}{\left(
\epsilon ^{2}+r^{2}\right) ^{n}}\left( 1-\frac{R_{g}}{6n}r^{2}+o(r^{2})%
\right) r^{n-1}dr=
\end{equation*}%
\begin{equation}
=\frac{1}{2}\epsilon ^{-n+4}I_{n}^{\frac{n}{2}-1}\left[ \frac{n(n-2)\left(
n+2\right) }{n-4}-\frac{n^{2}+4}{6\left( n-6\right) }R_{g}\epsilon
^{2}+o\left( \epsilon ^{2}\right) \right]  \label{64}
\end{equation}%
also%
\begin{equation*}
\int_{0}^{\varrho }\frac{\left( r\frac{\partial }{\partial r}\log \sqrt{%
\left\vert g\right\vert }\right) ^{2}}{\left( \epsilon ^{2}+r^{2}\right)
^{n-2}}\left( 1-\frac{R_{g}}{6n}r^{2}+o(r^{2})\right) r^{n-1}dr=\frac{R_{g}}{%
3n}\int_{0}^{\varrho }\frac{r^{n+3}}{\left( \epsilon ^{2}+r^{2}\right) ^{n-2}%
}\left( 1+o(r)\right) dr
\end{equation*}%
\begin{equation}
=-\epsilon ^{-n+6}\frac{R_{g}}{3n}\left( 1+o\left( \epsilon \right) \right)
I_{n-1}^{\frac{n}{2}+1}  \label{65}
\end{equation}%
and 
\begin{equation}
\int_{0}^{\varrho }2r\frac{2r^{2}+n\epsilon ^{2}}{\left( \epsilon
^{2}+r^{2}\right) ^{n-1}}\frac{\partial }{\partial r}\log \sqrt{\left\vert
g\right\vert }\left( 1-\frac{R_{g}}{6n}r^{2}+o(r^{2})\right)
r^{n-1}dr=-\epsilon ^{-n+4}.o\left( \epsilon ^{3}\right)  \label{66}
\end{equation}%
Hence (\ref{64}), (\ref{65}) and (\ref{66}) lead to

\begin{equation}
A=\int_{M}\left( \Delta \varphi _{\epsilon }\right)
^{2}dv_{g}=\int_{B_{\epsilon }\left( P\right) }\left( \Delta \varphi
_{\epsilon }\right) ^{2}dv_{g}+\int_{B_{2\epsilon }(P)-B_{\epsilon }\left(
P\right) }\left( \Delta \varphi _{\epsilon }\right) ^{2}dv_{g}  \label{66'}
\end{equation}%
\begin{equation*}
=\frac{1}{2}\epsilon ^{-n+4}\left( n-4\right) n\left( n^{2}-4\right) \omega
_{n-1}I_{n}^{\frac{n}{2}-1}\left[ 1-R_{g}\epsilon ^{2}\left( \frac{n^{2}+4}{%
6\left( n-6\right) }+\frac{4\left( n-1\right) \left( n-2\right) }{3\left(
n-4\right) \left( n-6\right) }\right) \frac{n-4}{n\left( n^{2}-4\right) }%
+o\left( \epsilon ^{2}\right) \right]
\end{equation*}%
\newline
Resuming, we obtain%
\begin{equation*}
J\left( \varphi _{\epsilon }\right) \leq \left( n-4\right) n\left(
n^{2}-4\right) \left( \frac{\omega _{n-1}I_{n}^{\frac{n}{2}-1}}{2}\right) ^{%
\frac{4}{n}}f(P)^{-\frac{2}{N}}\left( 1+\epsilon ^{2}\frac{n-4}{n\left(
n-2\right) }\left( \frac{\Delta f\left( P\right) }{2f(P)}+\frac{R_{g}}{6}%
\right) +o\left( \epsilon ^{2}\right) \right)
\end{equation*}%
\begin{equation*}
\times \left[ 1+\epsilon ^{2-\frac{n}{s}}\left( \frac{1}{2}\right) ^{-\frac{1%
}{s}}\frac{n-4}{n(n^{2}-4)}\left\Vert a\right\Vert _{s}\omega _{n-1}^{1-%
\frac{1}{s}}\left( I_{\left( n-2\right) \frac{s}{s-1}}^{\frac{s}{s-1}+\frac{n%
}{2}-1}\right) ^{1-\frac{1}{s}}\right.
\end{equation*}%
\begin{equation*}
\left. +\epsilon ^{4-\frac{n}{p}}\left( \frac{1}{2}\right) ^{-\frac{1}{p}}%
\frac{1}{\left( n-4\right) n\left( n^{2}-4\right) }\omega _{n-1}^{1-\frac{1}{%
p}}\left\Vert b\right\Vert _{p}\left( I_{\left( n-4\right) \frac{p}{p-1}}^{%
\frac{n}{2}-1}\right) ^{1-\frac{1}{p}}\right.
\end{equation*}%
\begin{equation*}
\left. -\epsilon ^{2}\left( \frac{n^{2}+4}{6\left( n-6\right) }+\frac{%
4\left( n-1\right) \left( n-2\right) }{3\left( n-4\right) \left( n-6\right) }%
\right) \frac{n-4}{n\left( n^{2}-4\right) }R_{g}+o\left( \epsilon
^{2}\right) \right]
\end{equation*}%
\begin{equation*}
=\left( n-4\right) n\left( n^{2}-4\right) \left( \frac{\omega _{n-1}I_{n}^{%
\frac{n}{2}-1}}{2}\right) ^{\frac{4}{n}}f(P)^{-\frac{2}{N}}\left( 1-\epsilon
^{2}\left( \frac{n^{2}+4n-20}{6\left( n-6\right) (n^{2}-4)}R_{g}-\frac{n-4}{%
n\left( n-2\right) }\frac{\Delta f(P)}{2f(P)}\right) \right) +o\left(
\epsilon ^{2}\right) \text{.}
\end{equation*}%
Now since $\epsilon $ is arbitrary small and $n<\min (2s,4p)$ we have%
\begin{equation*}
\epsilon ^{2-\frac{n}{s}}\left( \frac{1}{2}\right) ^{-\frac{1}{s}}\frac{n-4}{%
n(n^{2}-4)}\left\Vert a\right\Vert _{s}\omega _{n-1}^{1-\frac{1}{s}}\left(
I_{\left( n-2\right) \frac{s}{s-1}}^{\frac{s}{s-1}+\frac{n}{2}-1}\right) ^{1-%
\frac{1}{s}}\leq \left\Vert a\right\Vert _{s}
\end{equation*}%
and 
\begin{equation*}
\epsilon ^{4-\frac{n}{p}}\left( \frac{1}{2}\right) ^{-\frac{1}{p}}\frac{1}{%
\left( n-4\right) n\left( n^{2}-4\right) }\omega _{n-1}^{1-\frac{1}{p}%
}\left\Vert b\right\Vert _{p}\left( I_{\left( n-4\right) \frac{p}{p-1}}^{%
\frac{n}{2}-1}\right) ^{1-\frac{1}{p}}\leq \left\Vert b\right\Vert _{p}.
\end{equation*}%
Hence%
\begin{equation*}
J\left( \varphi _{\epsilon }\right) \leq n\left( n-4\right) \left(
n^{2}-4\right) \left( \frac{I_{n}^{\frac{n}{2}-1}\omega _{n-1}}{2}\right) ^{%
\frac{4}{n}}f(P)^{-\frac{2}{N}}\left( 1+\left\Vert a\right\Vert
_{s}+\left\Vert b\right\Vert _{p}\right)
\end{equation*}%
\begin{equation*}
\times \left[ 1-\epsilon ^{2}\left\{ \frac{n^{2}+4n-20}{6\left( n-6\right)
(n^{2}-4)}R_{g}-\frac{n-4}{2n\left( n-2\right) }\frac{\Delta f(P)}{f(P)}%
\right\} \frac{1}{1+\left\Vert a\right\Vert _{s}+\left\Vert b\right\Vert _{p}%
}\right] +o\left( \epsilon ^{2}\right)
\end{equation*}%
Recall the value of the best Sobolev's constant, 
\begin{equation*}
K\left( n,2\right) ^{-2}=n\left( n^{2}-4\right) \left( n-4\right) \left( 
\frac{I_{n}^{\frac{n}{2}-1}\omega _{n-1}}{2}\right) ^{\frac{4}{n}}
\end{equation*}%
so if 
\begin{equation*}
\frac{n^{2}+4n-20}{6\left( n-6\right) (n^{2}-4)}R_{g}-\frac{n-4}{2n\left(
n-2\right) }\frac{\Delta f(P)}{f(P)}>0
\end{equation*}%
we get%
\begin{equation*}
Q(M)<K\left( n,2\right) ^{-2}f(P)^{-\frac{2}{N}}\left( 1+\left\Vert
a\right\Vert _{s}+\left\Vert b\right\Vert _{p}\right) \text{ \ \ \ \ as \ \
\ \ \ }\epsilon \rightarrow 0^{+}\text{.}
\end{equation*}%
For $n=6$, direct computations give, for $\frac{3}{2}<p<2$ and $3<s<4$, 
\begin{equation*}
J\left( \varphi _{\epsilon }\right) \leq K\left( n,2\right) ^{-2}f(P)^{-%
\frac{1}{3}}\left( 1+\left\Vert a\right\Vert _{s}+\left\Vert b\right\Vert
_{p}\right) \left( 1-4\epsilon ^{2}\frac{\log \left( \frac{1}{\epsilon ^{2}}%
\right) \left( \frac{\omega _{5}}{2}\right) ^{\frac{3}{4}}\left(
f(P)I_{6}^{2}\right) ^{-\frac{1}{3}}}{1+\left\Vert a\right\Vert
_{s}+\left\Vert b\right\Vert _{p}}\frac{R_{g}}{3}\right)
\end{equation*}%
\begin{equation*}
+o\left( \epsilon ^{2}\right)
\end{equation*}%
so, if $R_{g}>0$, we obtain 
\begin{equation*}
Q(M)<K\left( n,2\right) ^{-2}f(P)^{-\frac{1}{3}}\left( 1+\left\Vert
a\right\Vert _{s}+\left\Vert b\right\Vert _{p}\right) \text{ \ as }\epsilon
\rightarrow 0^{+}\text{.}
\end{equation*}

In Corollary \ref{Cor2} the coefficients of our equation are of the form $%
\frac{a}{\rho ^{\gamma }}$ and $\frac{b}{\rho ^{\alpha }}$ where the
functions $a$ and $b$ are smooth functions on $M$ and $0<\gamma <\frac{n}{s}%
<2$, $0<\alpha <\frac{n}{p}<4$. Considering the expansions at $P$ for the
functions $a$ and $b$ we get for $n\geq 6$ 
\begin{equation}
B=\int_{M}\frac{a(x)}{\rho ^{\gamma }}\left\vert \nabla u\right\vert
^{2}dv_{g}=\frac{1}{2}\left( n-4\right) ^{2}\epsilon ^{-n+6-\gamma }\omega
_{n-1}a(P)I_{n-2}^{\frac{n-\gamma }{2}}\left( 1+o\left( \epsilon ^{2}\right)
\right) \text{.}  \label{67}
\end{equation}%
If $n-8+\alpha >0$ that means $\alpha >1$ for $n=7$ and $\alpha >2$ for $n=6$%
, we obtain 
\begin{equation}
C=\int_{M}\frac{b(x)u^{2}}{\rho ^{\alpha }}dv_{g}=\frac{1}{2}\epsilon
^{-n+8-\alpha }\omega _{n-1}b(P)I_{n-4}^{\frac{n-\alpha }{2}-1}\left(
1+o\left( \epsilon ^{2}\right) \right) \text{.}  \label{68}
\end{equation}%
Now letting $\epsilon $ sufficiently small so that%
\begin{equation*}
\frac{n-4}{n\left( n^{2}-4\right) }\epsilon ^{2-\gamma }\omega _{n-1}a(P)%
\frac{I_{n-2}^{\frac{n-\gamma }{2}}}{I_{n}^{\frac{n}{2}-1}}\leq \left\Vert 
\frac{a}{\rho ^{\gamma }}\right\Vert _{s}
\end{equation*}%
and 
\begin{equation*}
\frac{1}{\left( n-4\right) n\left( n^{2}-4\right) }\epsilon ^{4-\alpha
}\omega _{n-1}b(P)\frac{I_{n-4}^{\frac{n-\alpha }{2}-1}}{I_{n}^{\frac{n}{2}%
-1}}\leq \left\Vert \frac{b}{\rho ^{\alpha }}\right\Vert _{p}
\end{equation*}%
we infer that, for $n>6$ and $1<\alpha <\frac{n}{p}<4$, and if the following
condition holds%
\begin{equation*}
\frac{n^{2}+4n-20}{6\left( n-6\right) (n^{2}-4)}R_{g}-\frac{n-4}{2n\left(
n-2\right) }\frac{\Delta f(P)}{f(P)}>0
\end{equation*}%
then 
\begin{equation*}
Q_{\gamma ,\alpha }(M)<K\left( n,2\right) ^{-2}f(P)^{-\frac{2}{N}}\left(
1+\left\Vert \frac{a}{\rho ^{\gamma }}\right\Vert _{s}+\left\Vert \frac{b}{%
\rho ^{\alpha }}\right\Vert _{p}\right) \text{ as }\epsilon \rightarrow 0^{+}%
\text{.}
\end{equation*}%
In the case $n=6$ and $2<\alpha <\frac{n}{p}<4$ similar computations allow
us to claim 
\begin{equation*}
Q_{\gamma ,\alpha }(M)<K\left( n,2\right) ^{-2}f(P)^{-\frac{1}{3}}\left(
1+\left\Vert \frac{a}{\rho ^{\gamma }}\right\Vert _{s}+\left\Vert \frac{b}{%
\rho ^{\alpha }}\right\Vert _{p}\right) \text{ \ as }\epsilon \rightarrow
0^{+}
\end{equation*}%
provided that $R_{g}>0$.

\begin{acknowledgement}
The author would like to thank the referee for valuable comments that have
been implemented in the final version of the paper.
\end{acknowledgement}

\end{document}